\documentclass[11pt,onecolumn]{article}
\usepackage[left=3cm,top=2.5cm,right=3cm,bottom=2.5cm]{geometry}
\usepackage{graphicx,subfigure,epstopdf,grffile,times}
\usepackage{amssymb,amsmath,amsthm,amsopn,mathrsfs,amsfonts}
\usepackage{topcapt,multirow,caption,enumerate}
\usepackage[title]{appendix}
\usepackage[final]{pdfpages}
\numberwithin{equation}{section}
\graphicspath{{./Figures/}}	
\newcounter{remarkcounter} 
\theoremstyle{definition}
\newtheorem{remark}[remarkcounter]{Remark}
\newtheorem{theorem}{Theorem}[section]

\newtheorem{corollary}[theorem]{Corollary}
\newtheorem{lemma}[theorem]{Lemma}
\newtheorem{definition}{Definition}[section]
\newtheorem{ass}{Assumption}[section]

\DeclareMathOperator{\eps}{\varepsilon}
\def\ip3{IP$_3$}

\title{\vspace{-20pt}Generic Torus Canards}
\author{Theodore Vo
\footnote{\footnotesize{
Department of Mathematics and Statistics, 
Boston University, Boston, MA 02215, USA 
(\textcolor{blue}{theovo@bu.edu})
}}}

\begin{document}
\maketitle
%=========================================================================================

%------------------------------------------------------------
\vspace{-10pt}
\begin{abstract}	\label{sec:abstract}
\noindent Torus canards are special solutions of slow/fast systems that alternate between attracting and repelling manifolds of limit cycles of the fast subsystem. A relatively new dynamic phenomenon, torus canards have been found in neural applications to mediate the transition from tonic spiking to bursting via amplitude-modulated spiking. In $\mathbb{R}^3$, torus canards are degenerate: they require one-parameter families of 2-fast/1-slow systems in order to be observed and even then, they only occur on exponentially thin parameter intervals. The addition of a second slow variable unfolds the torus canard phenomenon, making them generic and robust. That is, torus canards in slow/fast systems with (at least) two slow variables occur on open parameter sets. So far, generic torus canards have only been studied numerically, and their behaviour has been inferred based on averaging and canard theory. This approach, however, has not been rigorously justified since the averaging method breaks down near a fold of periodics, which is exactly where torus canards originate. In this work, we combine techniques from Floquet theory, averaging theory, and geometric singular perturbation theory to show that the average of a torus canard is a folded singularity canard. In so doing, we devise an analytic scheme for the identification and topological classification of torus canards in $\mathbb{R}^4$. We demonstrate the predictive power of our results in a model for intracellular calcium dynamics, where we explain the mechanisms underlying a novel class of elliptic bursting rhythms, called amplitude-modulated bursting, by constructing the torus canard analogues of mixed-mode oscillations. We also make explicit the connection between our results here with prior studies of torus canards and torus canard explosion in $\mathbb{R}^3$, and discuss how our methods can be extended to slow/fast systems of arbitrary (finite) dimension. 

\vspace{8pt}\noindent \textbf{Keywords}\qquad Torus canard, canard, geometric singular perturbation theory, folded singularity, averaging, bursting, spiking, amplitude-modulation, torus bifurcation

\vspace{10pt}\noindent \textbf{AMS subject classifications}\quad 34E17, 34C29, 34C15, 37N25, 34E15, 37G15, 34C20, 34C45
% 34E17  	Canard solutions
% 34C29  	Averaging method
% 34C15  	Nonlinear oscillations, coupled oscillators
% 37N25	Dynamical systems in biology
% 34E15  	Singular perturbations, general theory
% 37G15  Bifurcations of limit cycles and periodic orbits
% 34C20  	Transformation and reduction of equations and systems, normal forms
% 34C45  	Invariant manifolds
\end{abstract}
%------------------------------------------------------------

%---------------------------------------------------------------------------------
\section{Introduction}	\label{sec:intro}
%---------------------------------------------------------------------------------

Many biological systems exhibit complex oscillatory dynamics that evolve over multiple time-scales, such as the spiking and bursting activity of neurons, sinus rhythms in the beating of the heart, and intracellular calcium signalling. Such rhythms are often described by singularly perturbed systems of ordinary differential equations 
\begin{equation}	\label{eq:genericslowfast}
\begin{split}
\dot{x} &= f(x,y,\eps), \\
\dot{y} &= \eps g(x,y,\eps),
\end{split}
\end{equation}
where $0<\eps \ll 1$ is the ratio of slow and fast time-scales, $x \in \mathbb{R}^n$ is fast, $y \in \mathbb{R}^k$ is slow, and $f$ and $g$ are smooth functions. 
A relatively new type of oscillatory dynamic feature discovered in slow/fast systems with $n \geq 2$ is the so-called torus canard \cite{Kramer2008}. Torus canards are solutions of \eqref{eq:genericslowfast} that closely follow a family of attracting limit cycles of the fast subsystem of \eqref{eq:genericslowfast}, and then closely follow a family of repelling limit cycles of the fast subsystem of \eqref{eq:genericslowfast} for substantial times before being repelled. This unusual behaviour in the phase space typically manifests in the time course evolution as amplitude modulation of the rapid spiking waveform, as shown in Figure \ref{fig:amspiking}. 

\begin{figure}[h]
\centering
\includegraphics[width=5in]{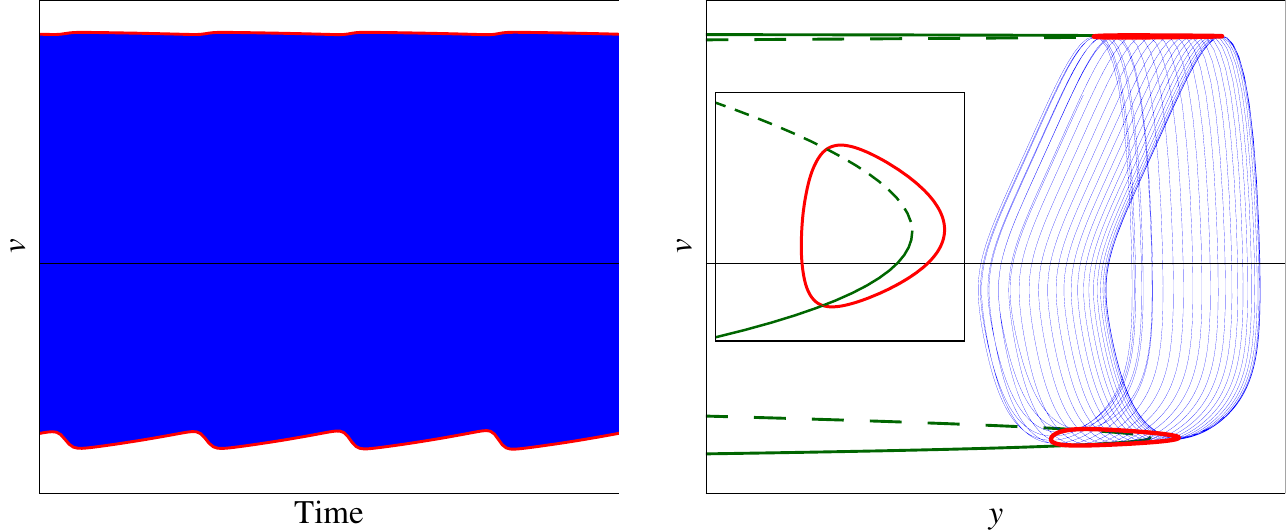}
\put(-364,142){(a)}
\put(-178,142){(b)}
\caption{A torus canard solution in a slow/fast system of the form \eqref{eq:genericslowfast} with a single slow variable $y$ and $n \geq 2$ fast variables, one of which is $v$. (a) The time evolution of the torus canard in this case is an amplitude-modulated spiking rhythm, which consists of rapid spiking (blue) wherein the envelope of the waveform (red) also oscillates. (b) The projection of the torus canard into the slow/fast phase plane shows that the torus canard arises in the neighbourhood of where an attracting family of limit cycles of the fast subsystem (green, solid) and a repelling family of limit cycles of the fast subsystem (green, dashed) meet (inset). The trajectory alternately spends long times following both the attracting and repelling branches of limit cycles.}
\label{fig:amspiking}
\end{figure}

First discovered in a model for the neuronal activity in cerebellar Purkinje cells \cite{Kramer2008}, torus canards were observed as quasi-periodic solutions that would appear during the transition between bursting and rapid spiking states of the system. Further insight into the dynamics of the torus canards in this cell model was presented in \cite{Benes2011}, where a 2-fast/1-slow rotated van der Pol-type equation with symmetry breaking was studied. Since then, torus canards have been encountered in several other neural models \cite{Burke2012}, such as Hindmarsh-Rose (subHopf/fold cycle bursting), Morris-Lecar-Terman (circle/fold cycle bursting), and Wilson-Cowan-Izhikevich (fold/fold cycle bursting), where they again appeared in the transition between spiking and bursting states. Additional studies have identified torus canards in chemical oscillators \cite{Straube2006}, and have shown that torus canards are capable of interacting with other dynamic features to create even more complicated oscillatory rhythms \cite{Desroches2012}. 

Three common threads link all of the examples mentioned above. First and foremost, the torus canards occur in the neighbourhood of a fold bifurcation of limit cycles, also known as a saddle-node of periodics (SNPOs), of the fast subsystem (see Figure \ref{fig:amspiking}(b) for instance). That is, the torus canards arise in the regions of phase space where an attracting set of limit cycles meets a repelling set of limit cycles. Second, the torus canards occur for parameter sets which are $\mathcal{O}(\eps)$ close to a torus bifurcation of the full system. Third, in these examples, there is only one slow variable, and the torus canards are restricted to exponentially thin parameter sets. In other words, the torus canards in these examples are degenerate. 

%Torus canards in $\mathbb{R}^3$ are the one-higher-dimensional analogue of planar canard cycles \cite{Burke2012}. 
Torus canards in $\mathbb{R}^3$ require a one-parameter family of 2-fast/1-slow systems in order to be observed, and they undergo a very rapid transition from rapid spiking to bursting (i.e., torus canard explosion) in an exponentially thin parameter window $\mathcal{O}(\eps)$ close to a torus bifurcation of the full system. In principle, the addition of a second slow variable unfolds the torus canard phenomenon, making the torus canards generic and robust. This is analogous to the unfolding of planar canard cycles via the addition of a second slow variable. That is, canard solutions in $\mathbb{R}^3$ are generic and robust, and their properties are encoded in folded singularities of the reduced flow \cite{Szmolyan2001}. 

So far, to our knowledge, the only case study of torus canards in systems with more than one slow variable is in a model for respiratory rhythm generation in the pre-B\"{o}tzinger complex \cite{Roberts2015}, which is a 6-fast/2-slow system. There, the torus canards were studied numerically by averaging the slow motions over limit cycles of the fast subsystem and examining the averaged slow drift along the manifold of periodics. In particular, folded singularities of the averaged slow flow were numerically identified and the properties of the torus canards were inferred based on canard theory. From their observations, the authors in \cite{Roberts2015} conjectured that the average of a torus canard is a folded singularity canard. 

This leads to the generic torus canard problem, which can be stated simply as follows. There is currently no analytic way to identify, classify, and analyze torus canards in the same way that canards in $\mathbb{R}^3$ can be classified and analyzed based on their associated folded singularity. It has been suggested that averaging methods should be used \cite{Cymbalyuk2005,Roberts2015} to reduce the torus canard problem to a folded singularity problem in a related averaged system. However, this approach is not rigorously justified since the averaging method breaks down in a neighbourhood of a fold of limit cycles, which is precisely where the torus canards are located. Our main goal then is to extend the averaging method to the torus canard regime and hence solve the generic torus canard problem in $\mathbb{R}^4$.

There are three types of results in this article: theoretical, numerical, and phenomenological. The theoretical contribution is that we extend the averaging method to folded manifolds of limit cycles and hence to the torus canard regime. In so doing, we inherit Fenichel theory \cite{Fenichel1979,Jones1995} for persistent manifolds of limit cycles and in particular, we are able to make use of the powerful theoretical framework of canard theory \cite{Szmolyan2001,Wechselberger2005}. We provide analytic criteria for the identification and characterization of torus canards based on an underlying class of novel singularities for differential equations, which we call \emph{toral folded singularities}. We illustrate our assertions by studying a spatially homogeneous model for intracellular calcium dynamics \cite{Politi2006}. In applying our results to this model, we discover a novel type of bursting rhythm, which we call \emph{amplitude-modulated bursting} (see Figure \ref{fig:ambursting} for an example). 
We show that these amplitude-modulated bursting solutions can be well-understood using our torus canard theory. In the process, we provide the first numerical computations of intersecting invariant manifolds of limit cycles.
The new phenomenological result that stems from our analysis is that we construct the torus canard analogue of a canard-induced mixed-mode oscillation \cite{Brons2006}.  

\begin{figure}[h!]
\centering
\includegraphics[width=5in]{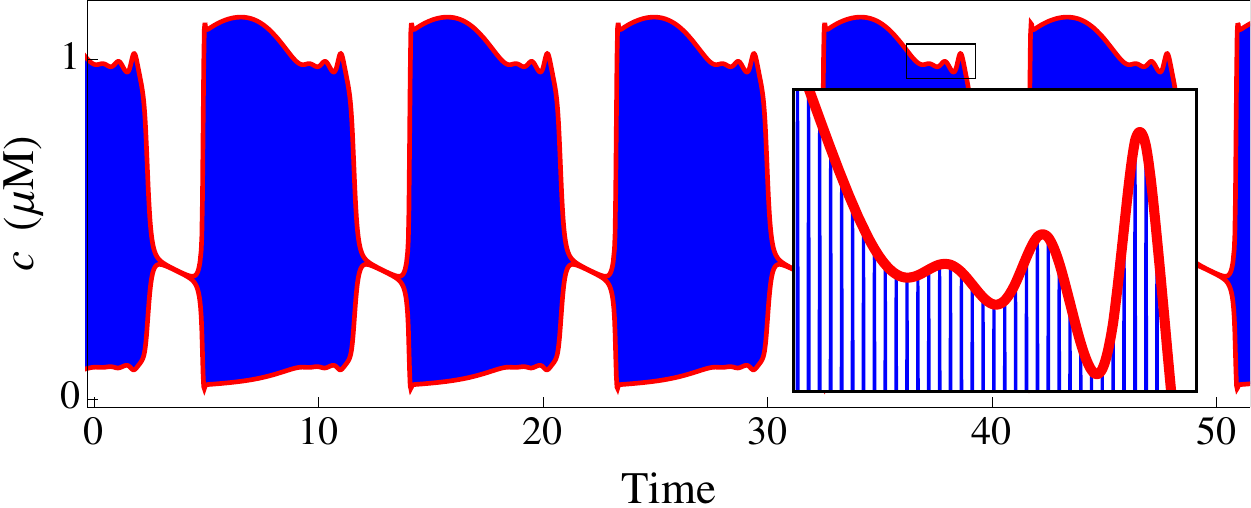}
\caption{Novel amplitude-modulated bursting rhythm discovered in a model for intracellular calcium dynamics (Section \ref{sec:PH}). The amplitude-modulated bursts alternate between active phases where the trajectory (blue) rapidly oscillates, and silent phases where the trajectory remains quiescent. During the active phase, the envelope (red) of the rapidly oscillating waveform exhibits small-amplitude oscillations which extend the burst duration. These amplitude-modulated bursts are torus canard-induced mixed-mode oscillations (Section \ref{sec:TCIMMO}).}
\label{fig:ambursting}
\end{figure}

\begin{remark}
Earlier reports of torus canards have been seen in the literature, even though that terminology was not used. 
In \cite{Neishtadt1987}, it was remarked that bifurcation delay may result when a trajectory crosses from a set of attracting states to a set of repelling states where the states may be either fixed points or limit cycles. 
In \cite{Izhikevich2000b,Izhikevich2001}, a canonical form for subcritical elliptic bursting near a Bautin bifurcation of the fast subsystem was studied. The canonical model consists of two fast (polar) variables $(r,\theta)$ and a single slow variable $u$. In these polar coordinates, the oscillatory states of the fast subsystem may be identified as stationary radii. Within this framework, torus canards occur as canard cycles of the planar $(r,u)$ subsystem, and in parameter space they arise in the rapid and continuous transition between the spiking and bursting regimes of the canonical model. 
\end{remark}

The outline of the paper is as follows. In Sections \ref{sec:averaging} and \ref{sec:classification}, we give the main theoretical results of the article. Namely, we state the generic torus canard problem in $\mathbb{R}^4$ in the case of two fast variables and two slow variables, and then combine techniques from Floquet theory \cite{Chicone2006}, averaging theory \cite{Sanders2007}, and geometric singular perturbation theory \cite{Fenichel1979,Jones1995} to show that the average of a torus canard is a folded singularity canard. In so doing, we devise analytic criteria for the identification and topological classification of torus canards based on their underlying toral folded singularity. We examine the main topological types of toral folded singularities and show that they encode properties of the torus canards, such as the number of torus canards that persist for $0< \eps \ll 1$. We then discuss bifurcations of torus canards and make the connection between torus canards and the torus bifurcation that is often observed in the full system. 

We apply our results to the Politi-H\"{o}fer model for intracellular calcium dynamics \cite{Politi2006} in Sections \ref{sec:PH}, \ref{sec:PHTC} and \ref{sec:TCIMMO}. We examine the bifurcation structure of the model and identify characteristic features that signal the presence of torus canards. Using our torus canard theory, we explain the dynamics underlying the novel class of amplitude-modulated bursting rhythms. We show that the amplitude modulation is organised locally in the phase space by twisted, intersecting invariant manifolds of limit cycles. These sections serve the dual purpose of illustrating the predictive power of our analysis, and also giving a representative example of how to implement those results in practice. 

In Section \ref{sec:explosion}, we make the connection between our current work on torus canards and prior work on torus canards in $\mathbb{R}^3$ explicit. We show that the theoretical framework developed in Sections \ref{sec:averaging} and \ref{sec:classification} can be used to compute the spiking/bursting boundary in the parameter spaces of 2-fast/1-slow systems by simply tracking the toral folded singularity. We illustrate these results in the Morris-Lecar-Terman, Hindmarsh-Rose, and Wilson-Cowan-Izhikevich models for neural bursting. 

In Section \ref{sec:arbitrarydimensions}, we extend our averaging method for folded manifolds of limit cycles to slow/fast systems with two fast variables and $k$ slow variables, where $k$ is any positive integer. Moreover, we provide asymptotic error estimates for the averaging method on folded manifolds of limit cycles. 
%We then discuss the extension to systems with an arbitrary number of fast variables. 
We then conclude in Section \ref{sec:discussion}, where we summarize the main results of the article, discuss their implications, and highlight several interesting open problems.

%---------------------------------------------------------------------------------
\section{Averaging Method for Folded Manifolds of Limit Cycles}	\label{sec:averaging}
%---------------------------------------------------------------------------------

In this section, we study generic torus canards in $\mathbb{R}^4$ in the case of two fast variables and two slow variables. 
In Section \ref{subsec:assumptions}, we state the assumptions of the generic torus canard problem in $\mathbb{R}^4$. Within this framework, we develop an averaging method for folded manifolds of limit cycles in Section \ref{subsec:theoretical} and derive a canonical form for the dynamics around a torus canard. In Section \ref{subsec:averagedcoefficients}, we list (algorithmically) the averaged coefficients that appear in the canonical form.

%----------------------------------------------------
\subsection{Setup of the Generic Torus Canard Problem In $\mathbb{R}^4$} 	\label{subsec:assumptions}
%----------------------------------------------------

We consider four-dimensional singularly perturbed systems of ordinary differential equations of the form
\begin{equation}	\label{eq:main}
\begin{split}
\dot{x} &= f(x,y,\eps), \\
\dot{y} &= \eps g(x,y,\eps), 
\end{split}
\end{equation}
where $0<\eps \ll 1$ measures the time-scale separation, $x \in \mathbb{R}^2$ is fast, $y \in \mathbb{R}^2$ is slow, $f$ and $g$ are sufficiently smooth functions, and $f, g$ and their derivatives are $\mathcal{O}(1)$ with respect to $\eps$. 

\begin{ass}	\label{ass:man}
The layer problem of system \eqref{eq:main}, given by 
\begin{equation}	\label{eq:layer}
\begin{split}
\dot{x} &= f(x,y,0),
\end{split}
\end{equation}
possesses a manifold $\mathcal{P}$ of limit cycles, parametrized by the slow variables. For each fixed $y \in \mathcal{P}$, let $\Gamma(t,y)$ denote the corresponding limit cycle and assume that $\Gamma(t,y)$ has finite, non-zero period $T(y)$. That is, 
\[ \mathcal{P} := \left\{ (\Gamma(t,y),y) \in \mathbb{R}^4 : \dot{\Gamma} = f(\Gamma(t,y),y,0) \text{ and } \Gamma(t,y)=\Gamma(t+T(y),y) \right\}. \]
\end{ass}
\bigskip

The Floquet exponents of $\Gamma(t,y)$ are given by
\[ \varphi_1 = 0, \text{ and }\, \varphi_2 = \frac{1}{T(y)}\int_0^{T(y)} \operatorname{tr} \, D_x f (\Gamma(t,y),y,0)\, dt, \]
where $\varphi_1$ corresponds to a Floquet multiplier equal to unity, which reflects the fact that $\Gamma(t,y)$ is neutrally stable to shifts along the periodic orbit \cite{Chicone2006}. The stability then, of the periodic orbit $\Gamma(t,y)$, is encoded in the Floquet exponent $\varphi_2$. If $\varphi_2<0$, then $\Gamma(t,y)$ is an asymptotically stable solution of \eqref{eq:layer} and if $\varphi_2>0$, then $\Gamma(t,y)$ is an unstable solution of \eqref{eq:layer}. 

\begin{ass}	\label{ass:fold}
The layer problem \eqref{eq:layer} possesses a manifold $\mathcal{P}_L$ of SNPOs given by
\[ \mathcal{P}_L := \left\{ (x,y) \in \mathcal{P} : \varphi_2 = \frac{1}{T(y)} \int_0^{T(y)} \operatorname{tr}\, D_x f (\Gamma(t,y),y,0)\, dt = 0 \right\}.  \]
Moreover, we assume that the manifold of periodics is a non-degenerate folded manifold so that $\mathcal{P}$ can be partitioned into attracting and repelling subsets, separated by the manifold of SNPOs. That is,
\begin{equation} 	\label{eq:Pdecomposition}
\mathcal{P} = \mathcal{P}_a \cup \mathcal{P}_L \cup \mathcal{P}_r,
\end{equation}
where $\mathcal{P}_a$ is the subset of $\mathcal{P}$ along which $\varphi_2<0$, and $\mathcal{P}_r$ is the subset of $\mathcal{P}$ along which $\varphi_2>0$.
\end{ass}
\medskip

We refer forward to Section \ref{subsec:averagedcoefficients} for a more precise formulation of the non-degeneracy condition that $\varphi_2$ changes sign along the manifold of SNPOs. A schematic of our setup is shown in Figure \ref{fig:setup}.

\begin{figure}[h]
\centering
\includegraphics[width=5in]{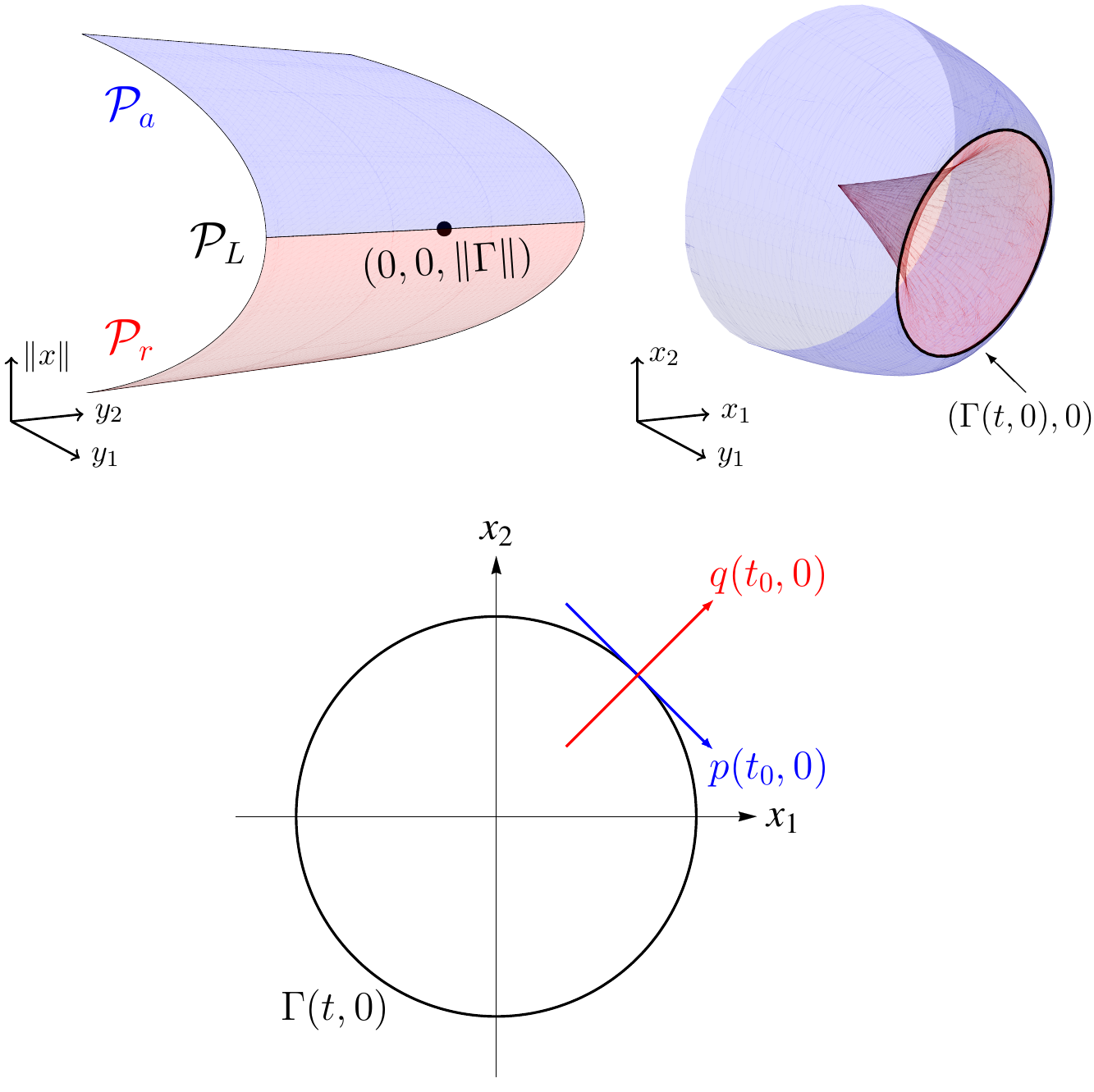}
\put(-360,340){(a)}
\put(-160,340){(b)}
\put(-284,182){(c)}
\caption{Projections of the geometric configuration under consideration. (a) Folded manifold of periodics in $(y_1,y_2,\lVert x \rVert)$-coordinates. Each point on the folded manifold $\mathcal{P}$ corresponds to a limit cycle of the layer problem \eqref{eq:layer}. 
 (b) Attracting (blue) and repelling (red) manifolds of limit cycles joined by the folded limit cycle $\Gamma(t,0)$ (which corresponds to the black marker in (a)) shown in the cross-section $y_2=0$. (c) The folded limit cycle $\Gamma(t,0)$, indicated by the black marker in (a), shown in the cross-section $y_1=y_2=0$, with unit tangent and normal vectors, $p(t_0,0)$ and $q(t_0,0)$, respectively for some fixed $t_0$.}
\label{fig:setup}
\end{figure}

\begin{ass}	\label{ass:homoclinic}
If the layer problem \eqref{eq:layer} has a critical manifold $\mathcal{S}$, then $\mathcal{S}$ and $\mathcal{P}_L$ are disjoint. 
\end{ass}

Assumption \ref{ass:homoclinic} guarantees that the periodic orbits of the layer problem in a neighbourhood of the manifold of SNPOs have finite period. Note that we are not eliminating the possibility of the manifold of limit cycles from intersecting the critical manifold $\mathcal{S}$, as would be the case near a set of Hopf bifurcations of the layer problem. Instead, we restrict the problem so that the SNPOs of \eqref{eq:layer} stay a reasonable distance from the critical manifold. 

An important step in the analysis to follow is identifying unit tangent and unit normal vectors to the periodic orbit, $\Gamma(t,y)$, of the layer problem. One choice of unit tangent and normal vectors, $p$ and $q$, to the periodic $\Gamma$ for fixed $y$, is given by
\[ p(t,y) = \frac{f(\Gamma,y,0)}{\lVert f(\Gamma,y,0) \rVert}, \,\, \text{ and } \,\, q(t,y) = \frac{J f(\Gamma,y,0)}{\Vert f(\Gamma,y,0) \rVert} = \frac{1}{\lVert f(\Gamma,y,0) \rVert} \begin{pmatrix} f_2(\Gamma,y,0) \\ -f_1(\Gamma,y,0) \end{pmatrix}, \]
where $\lVert \cdot \rVert$ denotes the standard Euclidean norm and $J$ is the skew-symmetric matrix $\begin{pmatrix} 0 & -1 \\ 1 & 0 \end{pmatrix}$.

%----------------------------------------------------
\subsection{Averaging Theorem for Folded Manifolds of Limit Cycles} 	\label{subsec:theoretical}
%----------------------------------------------------

The idea of averaging theory is to find a flow that approximates the slow flow on the family of periodic orbits of the layer problem \cite{Kuehn2015}. These averaging methods can be used to show that the effective slow dynamics on a family of asymptotically stable periodics are determined by an appropriately averaged system \cite{Pontryagin1960}. That is, the slow drift on $\mathcal{P}_a$ can be approximated by averaging out the fast oscillations, and the error in the approximation is $\mathcal{O}(\eps)$. However, to our knowledge, there are currently no theoretical results about the slow drift near folded manifolds of periodics. The following theorem extends the averaging method from normally hyperbolic manifolds of limit cycles to folded manifolds of limit cycles. 

\begin{theorem}[\bf{Averaging on Folded Manifolds of Limit Cycles}]	\label{thm:averaging}
Consider system \eqref{eq:main} under Assumptions \ref{ass:man}, \ref{ass:fold}, and \ref{ass:homoclinic}, and let $(\Gamma(t,y),y) \in \mathcal{P}_L$. Then there exists a sequence of near-identity transformations such that the averaged dynamics of \eqref{eq:main} in a neighbourhood of $(\Gamma(t,y),y)$ are approximated by
\begin{equation}	\label{eq:averaging}
\begin{split}
\dot{R} &= \overline{a}_1 u_1 + \overline{a}_2 u_2 + \overline{b} R^2 + \overline{c}_1 R\,u_1 + \overline{c}_2 R\, u_2+ \mathcal{O}(\eps,R^3, R^2(u_1+u_2), (u_1+u_2)^2), \\
\dot{u}_1 &= \eps \left( \overline{g}_1 + \overline{d}_1 R + \overline{e}_{11} u_1 +\overline{e}_{12}u_2 + \mathcal{O}(\eps,R(u_1+u_2),R^2) \right), \\
\dot{u}_2 &= \eps \left( \overline{g}_2 + \overline{d}_2 R + \overline{e}_{21} u_1 +\overline{e}_{22}u_2 + \mathcal{O}(\eps,R(u_1+u_2),R^2) \right),
\end{split}
\end{equation}
where an overbar denotes an average over one period of $\Gamma(t,y)$, and the coefficients in system \eqref{eq:averaging} can be computed explicitly (see Section \ref{subsec:averagedcoefficients}). 
\end{theorem}

\begin{remark}
The fast variable $R$ in system \eqref{eq:averaging} can be thought of as the averaged radial perturbation to $\Gamma(t,y)$ in the direction of $q$, and the slow variables $u$ describe the averaged evolution of $y$. 
\end{remark}

\begin{proof}
We present the proof of Theorem \ref{thm:averaging} in Section \ref{sec:arbitrarydimensions}. The idea of the proof is to switch to a coordinate frame that moves with the limit cycles, apply a coordinate transformation that removes the linear radial perturbation, and then average out the rapid oscillations.
\end{proof}

The significance of Theorem \ref{thm:averaging} is that the averaged radial-slow dynamics described by system \eqref{eq:averaging} are autonomous, singularly perturbed, and occur in the neighbourhood of a folded critical manifold. As such, system \eqref{eq:averaging} falls under the framework of canard theory \cite{Szmolyan2001,Wechselberger2005,Wechselberger2012}.

\begin{remark}	
Theorem \ref{thm:averaging} is a formal extension of the averaging method to folded manifolds of limit cycles.  
We defer the statement of asymptotic error estimates (i.e., the validity of this averaging method) to Section \ref{sec:arbitrarydimensions}. 
%after we have introduced the notion of toral folded singularities and presented the dynamics of torus canards in the main cases.
\end{remark} 

%----------------------------------------------------
\subsection{Coefficients of the Averaged Vector Field \eqref{eq:averaging}}		\label{subsec:averagedcoefficients}
%----------------------------------------------------

Here we list the averaged coefficients that appear in Theorem \ref{thm:averaging} (and Theorem \ref{thm:averagingarbitraryslow}). We denote the period of $\Gamma(t,y)$ by $T$. The functions $f, g$, and their derivatives are all evaluated at the limit cycle $(\Gamma(t,y),y)$ of the layer problem.
Recall that $p$ and $q$ denote unit tangent and unit normals to $\Gamma(t,y)$, respectively. We give the coefficients for the case of $2$-fast variables and $k$-slow variables, where $k \geq 1$.

Let $\Phi(t)$ be the fundamental solution defined by 
\[ \frac{d\Phi}{dt} =  \left( \operatorname{tr}\, D_x f - \frac{f \cdot (D_x f)\,f}{\lVert f \rVert^2} \right)  \, \Phi, \quad \Phi(0) = 1. \]
We will show in Section \ref{sec:arbitrarydimensions} that $\Phi(t)$ is $T$-periodic and bounded for all time (Lemma \ref{lemma:linear}). The coefficients of the linear slow terms and the quadratic radial term in the radial equation are given (component-wise) by
\begin{align*}
a_j &= \frac{1}{\Phi(t)} \left( (D_y f)^T \, q \right)_j, \quad j=1,2,\ldots, k,   \\ %= \frac{1}{\Phi(t)} q \cdot \frac{\partial f}{\partial y_j}, \\
b &= \frac{1}{2} \Phi(t)\, q \cdot \begin{pmatrix} (q\cdot \nabla_x)^2 f_1 \\ (q\cdot \nabla_x)^2 f_2 \end{pmatrix}.
\end{align*}
Note that the coefficient $b$ of the quadratic $R$ term is a scalar. We compute the auxiliary quantities $\alpha_j$ and $\beta$ as solutions of 
\begin{align*}
\frac{d\alpha_j}{dt} &= a_j - \frac{1}{T} \int_0^T a_j \, dt, \quad \alpha_j(0)=0, \\
\frac{d\beta}{dt} &= b - \frac{1}{T} \int_0^T b \, dt, \quad \beta(0) = 0,
\end{align*}
for $j=1,2,\ldots, k$. We refer forward to equation \eqref{eq:NIT} for the interpretation of $\alpha_j$ and $\beta$. Using these auxiliary functions, we can compute the coefficients of the mixed terms in the radial equation according to
\[ c_j = (H^T \,q)_j +2\alpha_j b -2\beta \,\overline{a}_j,  \quad j=1,2,\ldots,k, \]
where the $2\times k$ matrix $H$ is given by 
\[ H_{ij} := q \cdot \nabla_x \left( \frac{\partial f_i}{\partial y_j} \right), \quad i=1,2, \quad j=1,2,\ldots,k, \]
and $\overline{a}_j$ is the average of $a_j$ over one period of $\Gamma(t,y)$. 

The coefficients of the linear $R$-terms in the slow equations are 
\[ d_j = \Phi(t)\, \left( (D_xg) \,q \right)_j  = \Phi(t)\, \left( \nabla_x\, g_j \right) \cdot q, \quad j=1,2,\ldots,k. \]
Finally, the $k \times k$ matrix of coefficients of the linear $u$-terms in the slow equations is
\[ e_{ij} = \frac{\partial g_i}{\partial y_j} + d_i \alpha_j, \quad i=1,2,\ldots,k,  \quad j=1,2,\ldots, k. \]

We can now simply list the averaged coefficients as
\[ 
\overline{\zeta}_j = \frac{1}{T} \int_0^T \zeta_j \, dt, \quad
\overline{b} = \frac{1}{T} \int_0^T b\, dt, \quad
\overline{e}_{ij} = \frac{1}{T} \int_0^T e_{ij} \, dt, 
\]
for $i=1,2$ and $j=1,2,\ldots, k$, where $\zeta = a, c, g$, or $d$. 
Note that the non-degeneracy condition in Assumption \ref{ass:fold} is given by the requirement that the averaged coefficient of the quadratic radial term is non-zero (i.e., $\overline{b} \neq 0$). We point out that the leading order terms in the averaged slow directions are simply given by the averages of the slow components of the vector field over one period of $\Gamma(t,y)$.

%---------------------------------------------------------------------------------
\section{Classification of Toral Folded Singularities \& Torus Canards}	\label{sec:classification}
%---------------------------------------------------------------------------------

We now study the dynamics of the averaged radial-slow system \eqref{eq:averaging} using geometric singular perturbation techniques. 
In Section \ref{subsec:toralfoldedsing}, we define the notion of a toral folded singularity -- a special limit cycle in the phase space from which torus canard dynamics can originate. 
We provide a topological classification of toral folded singularities and their associated torus canards in Section \ref{subsec:toralclass}. 
We then present the dynamics of the torus canards in the main cases, including those that exist near toral folded nodes (Section \ref{subsec:toralFN}), toral folded saddles (Section \ref{subsec:toralFS}), and toral folded saddle-nodes (Section \ref{subsec:toralFSN}).

%----------------------------------------------------
\subsection{Toral Folded Singularities} 	\label{subsec:toralfoldedsing}
%----------------------------------------------------

We begin our geometric singular perturbation analysis of system \eqref{eq:averaging} by rewriting it in the more succinct form
\begin{equation}	\label{eq:Ruslowfast}
\begin{split}
R^{\prime} &= F(R,u_1,u_2,\eps), \\
u_1^{\prime} &= \eps G_1(R,u_1,u_2,\eps), \\
u_2^{\prime} &= \eps G_2(R,u_1,u_2,\eps),
\end{split}
\end{equation}
where $F, G_1, G_2$ correspond to the right-hand-sides of \eqref{eq:averaging}, and the prime denotes derivatives with respect to the fast time $t=t_f$ (which is related to the slow time $t$ by $t=\eps t_f$). The idea is to decompose the dynamics of \eqref{eq:Ruslowfast} into its slow and fast motions by taking the singular limit on the slow and fast time-scales. 

The fast dynamics are approximated by solutions of the layer problem, 
\begin{equation}	\label{eq:Rulayer}
\begin{split}
R^\prime &= F(R,u_1,u_2,0), 
\end{split}
\end{equation}
where $u_1$ and $u_2$ are parameters, and \eqref{eq:Rulayer} was obtained by taking the singular limit $\eps \to 0$ in \eqref{eq:Ruslowfast}. The set of equilibria of \eqref{eq:Rulayer}, given by 
\[ \mathcal{S} := \left\{ (R,u_1,u_2) \in \mathbb{R}^3 : F(R,u_1,u_2,0) = 0 \right\}, \]
is called the critical manifold and is a key object in the geometric singular perturbations approach. Assuming at least one of $\overline{a}_1$ and $\overline{a}_2$ is non-zero, the critical manifold has a local graph representation, $u_1 = u_{1S}(R,u_2)$, say. In the case of \eqref{eq:Rulayer}, the critical manifold is (locally) a parabolic cylinder in the $(R,u_1,u_2)$ phase space. Linear stability analysis of the layer problem \eqref{eq:Rulayer} shows that the attracting and repelling sheets, $\mathcal{S}_a$ and $\mathcal{S}_r$, of the critical manifold are separated by a curve of fold bifurcations, defined by
\[ \mathcal{L} := \left\{ (R,u_1,u_2) \in \mathcal{S} : F_R = 0 \right\}.  \]
Note that $\mathcal{S}_a$ and $\mathcal{S}_r$ correspond to the attracting and repelling manifolds of limit cycles, $\mathcal{P}_a$ and $\mathcal{P}_r$, respectively, introduced in \eqref{eq:Pdecomposition}. Moreover, the fold curve $\mathcal{L}$ corresponds to the manifold of SNPOs, $\mathcal{P}_L$. 

To describe the slow dynamics along the critical manifold $\mathcal{S}$, we switch to the slow time-scale ($t = \eps t_f$) in system \eqref{eq:Ruslowfast} and take the singular limit $\eps \to 0$ to obtain the reduced system
\begin{equation}	\label{eq:Ruslow}
\begin{split}
0 &= F(R,u_1,u_2,0), \\
\dot{u_1} &= G_1(R,u_1,u_2,0), \\
\dot{u_2} &= G_2(R,u_1,u_2,0),
\end{split}
\end{equation}
where the overdot denotes derivatives with respect to $t$. Typically, to obtain a complete description of the flow on $\mathcal{S}$, we would need to compute the dynamics in an atlas of overlapping coordinate charts. In this case, and as is the case in many applications, we can use the graph representation of $\mathcal{S}$ to project the dynamics of \eqref{eq:Ruslow} onto a single coordinate chart $(R,u_2)$. The dynamics on $\mathcal{S}$ are then given by
\begin{equation}	\label{eq:Rureduced}
\begin{split}
-F_R \dot{R} &= F_{u_1} G_1 + F_{u_2} G_2, \\
\dot{u}_2 &= G_2,
\end{split}
\end{equation}
where all functions and their derivatives are evaluated along $\mathcal{S}$. An important feature of the reduced flow \eqref{eq:Rureduced} highlighted by this projection is that the reduced flow is singular along the fold curve $\mathcal{L}$. That is, solutions of the reduced flow blow-up in finite time at the fold curve and are expected to fall off the critical manifold. To remove this finite-time blow-up of solutions, we introduce the phase space dependent time-transformation $dt_s = -F_R\, dt_d$, which gives the \emph{desingularized system}
\begin{equation}	\label{eq:Rudesing}
\begin{split}
\dot{R} &= F_{u_1} G_1 + F_{u_2} G_2, \\
\dot{u}_2 &= -F_R \, G_2,
\end{split}
\end{equation}
where we have recycled the overdot to denote derivatives with respect to $t_d$, and $u_1 = u_{1S}(R,u_2)$. On the attracting sheets $\mathcal{S}_a$, the desingularized flow \eqref{eq:Rudesing} is (topologically) equivalent to the reduced flow \eqref{eq:Rureduced}. However, on the repelling sheets $\mathcal{S}_r$ (where $F_R>0$), the time transformation reverses the orientation of trajectories, and the reduced flow \eqref{eq:Rureduced} is obtained by reversing the direction of the flow of the desingularized system \eqref{eq:Rudesing}. Thus, the reduced flow can be understood by examining the desingularized flow and keeping track of the dynamics on both sheets of $\mathcal{S}$.

The desingularized system possesses two types of equilibria: \emph{ordinary} and \emph{folded}. The set of ordinary singularities
\[ \mathcal{E} := \left\{ (R,u_1,u_2) \in \mathcal{S} : G_1 = 0 \,\, \text{ and } \,\, G_2 = 0 \right\}, \]
consists of isolated points which are equilibria of both the reduced and desingularized flows and are $\mathcal{O}(\eps)$ close to equilibria of the fully perturbed problem \eqref{eq:Ruslowfast} provided they remain sufficiently far from the fold curve $\mathcal{L}$. The set of folded singularities
\[ \mathcal{M} := \left\{ (R,u_1,u_2) \in \mathcal{L} : \begin{pmatrix} F_{u_1} \\ F_{u_2} \end{pmatrix} \cdot \begin{pmatrix} G_1 \\ G_2 \end{pmatrix} = 0 \right\}, \]
consists of isolated points along the fold curve where the right-hand-side of the $R$-equation in \eqref{eq:Rudesing} (and \eqref{eq:Rureduced}) vanishes. Whilst folded singularities are equilibria of the desingularized system, they are not equilibria of the reduced system. Instead, folded singularities are points where the $R$-equation of the reduced system has a zero-over-zero type indeterminacy, which means trajectories may potentially pass through the folded singularity with finite speed and cross from one sheet of the critical manifold to another. That is, a folded singularity is a distinguished point on the fold curve where the reduced vector field may actually be regular. 

\begin{definition}[\bf{Toral Folded Singularity}] 	\label{def:identification}
System \eqref{eq:main} under assumptions \ref{ass:man}, \ref{ass:fold}, and \ref{ass:homoclinic}, possesses a \emph{folded singularity of limit cycles}, or a \emph{toral folded singularity} for short, if it has a limit cycle $(\Gamma(t,y),y) \in \mathcal{P}_L$, such that 
\begin{align} 	\label{eq:switching}
\begin{pmatrix} \overline{a}_1 \\ \overline{a}_2 \end{pmatrix} \cdot \begin{pmatrix} \overline{g}_1 \\ \overline{g}_2 \end{pmatrix} =0,
\end{align}
where $\overline{a}_j, \overline{g}_j$ for $j=1,2$, are the averaged coefficients in Theorem \ref{thm:averaging} (listed in Section \ref{subsec:averagedcoefficients}).
\end{definition}

\begin{remark}
In slow/fast systems with two slow variables and one fast variable, a folded singularity, $p$, of the reduced flow on a folded critical manifold is a point on the fold of the critical manifold where there is a violation of transversality:  
\[ \left. \begin{pmatrix} f_{y_1} \\ f_{y_2} \end{pmatrix} \cdot \begin{pmatrix} g_1 \\ g_2 \end{pmatrix} \right|_{p} = 0. \]
Geometrically, this  corresponds to the scenario in which the projection of the reduced flow into the slow variable plane is tangent to the fold curve at $p$. Definition \ref{def:identification} gives the averaged analogue for torus canards. More precisely, a toral folded singularity is a folded limit cycle $(\Gamma,y) \in \mathcal{P}_L$ such that the projection of the averaged slow drift along $\mathcal{P}$ into the averaged slow variable plane is tangent to the projection of $\mathcal{P}_L$ into the $(u_1,u_2)$-plane at the toral folded singularity (see Figure \ref{fig:PHtoralfn} for an example).
\end{remark}

We see that a toral folded singularity is a folded singularity of the $(R,u_1,u_2)$ system where solutions of the slow flow along the folded critical manifold can cross with finite speed from $\mathcal{S}_a$ to $\mathcal{S}_r$ (or vice versa). That is, a toral folded singularity allows for singular canard solutions of the averaged radial-slow flow. A singular canard solution of the $(R,u_1,u_2)$ system corresponds, in turn, to a solution in the original $(x,y)$ system that slowly drifts along the manifold of periodics $\mathcal{P}$ and crosses from $\mathcal{P}_a$ to $\mathcal{P}_r$ (or vice versa) via a toral folded singularity. 
%Based on Definitions \ref{def:identification} and \ref{def:class}, we define singular torus canard solutions as follows.
Based on this, we define singular torus canard solutions as follows.

\begin{definition}[\bf{Singular Torus Canards}] 	\label{def:singTC}
Suppose system \eqref{eq:main} under assumptions \ref{ass:man}, \ref{ass:fold}, and \ref{ass:homoclinic} has a toral folded singularity. A \emph{singular torus canard} is
a singular canard solution of the averaged radial-slow system \eqref{eq:averaging}. A \emph{singular faux torus canard} is a singular faux canard solution of the averaged radial-slow system \eqref{eq:averaging}.
\end{definition}

%----------------------------------------------------
\subsection{Classification of Toral Folded Singularities} 	\label{subsec:toralclass}
%----------------------------------------------------

Canard theory classifies and characterizes singular canards based on their associated folded singularity. 
We have the following classification scheme for toral folded singularities. 
Suppose system \eqref{eq:main} under assumptions \ref{ass:man}, \ref{ass:fold}, and \ref{ass:homoclinic} possesses a toral folded singularity $(\Gamma,y)$. Let $\lambda_1, \lambda_2$ denote the eigenvalues of the desingularized flow of \eqref{eq:averaging}, linearized about the origin (i.e., about $(\Gamma,y)$). 

\begin{definition}[\bf{Classification of Toral Folded Singularities}] 	\label{def:class}
The toral folded singularity is
\begin{itemize}
\setlength{\itemsep}{0pt}
\item a \emph{toral folded saddle} if $\lambda_1<0<\lambda_2$, 
\item a \emph{toral folded saddle-node} if $\lambda_1 \neq 0$ and $\lambda_2 = 0$,
\item a \emph{toral folded node} if $\lambda_1 < \lambda_2 <0$, or a \emph{faux toral folded node} if $0< \lambda_2<\lambda_1$,
\item a \emph{degenerate toral folded node} if $\lambda_1=\lambda_2$, or
\item a \emph{toral folded focus} if $\lambda_1,\lambda_2 \in \mathbb{C}$.
\end{itemize}
\end{definition}

Thus, we may exploit the known results about the existence and dynamics of canards near classical folded singularities of the averaged radial-slow system in order to learn about the existence and dynamics of torus canards near toral folded singularities. Folded nodes, folded saddles, and folded saddle-nodes are known to possess singular canard solutions. We examine their toral analogues in Sections \ref{subsec:toralFN}, \ref{subsec:toralFS}, and \ref{subsec:toralFSN}, respectively. Folded foci possess no singular canards and so any trajectory of the slow drift along $\mathcal{P}$ that reaches the neighbourhood of a toral folded focus simply falls off the manifold of periodics.

%----------------------------------------------------
\subsection{Toral Canards of Toral Folded Nodes} 	\label{subsec:toralFN}
%----------------------------------------------------

In this section, we consider the toral folded node (TFN) and its unfolding in terms of the averaged radial-slow flow \eqref{eq:Ruslowfast}. 
Fenichel theory guarantees that the normally hyperbolic segments, $\mathcal{S}_a$ and $\mathcal{S}_r$, of $\mathcal{S}$ persist as invariant slow manifolds, $\mathcal{S}_a^{\eps}$ and $\mathcal{S}_r^{\eps}$,  of \eqref{eq:Ruslowfast} respectively for sufficiently small $\eps$. Fenichel theory breaks down in neighbourhoods of the fold curve $\mathcal{L}$ where normal hyperbolicity fails. The extension of $\mathcal{S}_a^{\eps}$ and $\mathcal{S}_r^{\eps}$ by the flow of \eqref{eq:Ruslowfast} into the neighbourhood of the TFN leads to a local twisting of the invariant slow manifolds around a common axis of rotation. This spiralling of $\mathcal{S}_a^{\eps}$ and $\mathcal{S}_r^{\eps}$ becomes more pronounced as the perturbation parameter $\eps$ is increased. Moreover, there is a finite number of intersections between $\mathcal{S}_a^{\eps}$ and $\mathcal{S}_r^{\eps}$. These intersections are known as (non-singular) maximal canards.    

The properties of the maximal canards associated to a TFN are encoded in the TFN itself in the following way. Let $\lambda_s < \lambda_w<0$ be the eigenvalues of the TFN, where we treat the TFN as an equilibrium of \eqref{eq:Rudesing}, and let $\mu := \lambda_w/\lambda_s$ be the eigenvalue ratio. Then, provided $\mu$ is bounded away from zero, the total number of maximal canards of \eqref{eq:Ruslowfast} that persist for sufficiently small $\eps$ is $s_{\max}+1$, where 
\begin{equation}  \label{eq:smax}
s_{\max} := \lfloor \frac{1+\mu}{2\mu} \rfloor,
\end{equation}
and $\lfloor \cdot \rfloor$ is the floor function \cite{Wechselberger2005,Wechselberger2012}. This follows by direct application of canard theory \cite{Szmolyan2001,Wechselberger2005} to the averaged radial-slow system \eqref{eq:Ruslowfast} in the case of a folded node. 
The first or outermost intersection of $\mathcal{S}_a^{\eps}$ and $\mathcal{S}_r^{\eps}$ is called the primary strong maximal canard, $\gamma_0$, and corresponds to the strong stable manifold of the TFN. The strong canard is also the local separatrix that separates the solutions which exhibit local small-amplitude oscillatory behaviour from monotone escape. That is, trajectories on $\mathcal{S}_a^{\eps}$ on one side of $\gamma_0$ execute a finite number of small oscillations, whilst trajectories on the other side of $\gamma_0$ simply jump away. 

The innermost intersection of $\mathcal{S}_a^{\eps}$ and $\mathcal{S}_r^{\eps}$ is the primary weak canard, $\gamma_w$, and corresponds to the weak eigendirection of the TFN. The weak canard plays the role of the axis of rotation for the invariant slow manifolds, which again follows directly from the results of \cite{Szmolyan2001,Wechselberger2005} applied to system \eqref{eq:Ruslowfast}. The remaining $s_{\max}-1$ secondary canards, $\gamma_i$, $i=1, \ldots, s_{\max}-1$, further partition $\mathcal{S}_a^{\eps}$ and $\mathcal{S}_r^{\eps}$ into sectors based on the rotational properties. That is, for $k=1,2,\ldots, s_{\max}$, the segments of $\mathcal{S}_a^{\eps}$ and $\mathcal{S}_r^{\eps}$ between $\gamma_{k-1}$ and $\gamma_k$ consist of orbit segments that exhibit $k$ small-amplitude oscillations in an $\mathcal{O}(\sqrt{\eps})$ neighbourhood of the TFN, where $\gamma_{s_{\max}} = \gamma_w$. 

Bifurcations of maximal canards occur at odd integer resonances in the eigenvalue ratio $\mu$. When $\mu^{-1}$ is an odd integer, there is a tangency between the invariant slow manifolds $\mathcal{S}_a^{\eps}$ and $\mathcal{S}_r^{\eps}$.  Increasing $\mu^{-1}$ through this odd integer value breaks the tangency between the slow manifolds resulting in two transverse intersections, i.e., two maximal canards, one of which is the weak canard. Thus, increasing $\mu^{-1}$ through an odd integer results in a branch of secondary canards that bifurcates from the axis of rotation \cite{Wechselberger2005}. 

Thus, in the case of a TFN, the averaged radial-slow dynamics \eqref{eq:Ruslowfast} are capable of generating singular canards, which perturb to maximal canards that twist around a common axis of rotation. Since a maximal canard of \eqref{eq:Ruslowfast}, by definition, lives at the intersection of (the extensions of) $\mathcal{S}_a^{\eps}$ and $\mathcal{S}_r^{\eps}$ in a neighbourhood of a folded node of \eqref{eq:Ruslowfast}, the corresponding trajectory in the original variables lives at the intersection of the extensions of $\mathcal{P}_a^{\eps}$ and $\mathcal{P}_r^{\eps}$ in the neighbourhood of the TFN (see Figure \ref{fig:PHinvariantmanifolds}). As such, we define a non-singular maximal torus canard as follows. 

\begin{definition}[\bf{Maximal Torus Canard}] 	\label{def:maximalTC}
Suppose system \eqref{eq:main} under assumptions \ref{ass:man}, \ref{ass:fold}, and \ref{ass:homoclinic} has a TFN singularity. A \emph{maximal torus canard} of \eqref{eq:main} is a trajectory corresponding to the intersection of the attracting and repelling invariant manifolds of limit cycles, $\mathcal{P}_a^{\eps}$ and $\mathcal{P}_r^{\eps}$, respectively, in a neighbourhood of the TFN. 
\end{definition}

The implication in moving from maximal canards of the averaged radial-slow system \eqref{eq:Ruslowfast} to maximal torus canards of the original system \eqref{eq:main} is that a torus canard associated to a TFN consists of three different types of motion working in concert. First, when the orbit is on $\mathcal{P}_a^{\eps}$, we have rapid oscillations due to limit cycles of the layer problem. The slow drift along $\mathcal{P}_a^{\eps}$ moves the rapidly oscillating orbit towards the manifold of SNPOs and in particular, towards the TFN. In a neighbourhood of the TFN, we have canard dynamics occurring within the envelope of the waveform. Combined, these motions (rapid oscillations due to limit cycles of the layer problem, slow drift along the manifold of limit cycles, and canard dynamics on the radial envelope) manifest as amplitude-modulated spiking rhythms. 

The maximal number of oscillations that the envelope of the waveform can execute is dictated by \eqref{eq:smax}. There are two ways in which oscillations can be added to or removed from the envelope. First is the creation of an additional secondary torus canard via odd integer resonances in $\mu^{-1}$. We conjecture that this bifurcation of torus canards will correspond to a torus doubling bifurcation. The other method of creating/destroying oscillations in the envelope is by keeping $\eps$ and $\mu$ fixed, and varying the position of the trajectory relative to the maximal torus canards. When the trajectory crosses a maximal torus canard, it moves into a different rotational sector, resulting in a change in the number of oscillations in the envelope (see Figure \ref{fig:PHsectors}).

%----------------------------------------------------
\subsection{Toral Canards of Toral Folded Saddles} 	\label{subsec:toralFS}
%----------------------------------------------------

In the case of a toral folded saddle, system \eqref{eq:Ruslowfast} possesses exactly one singular canard and one singular faux canard, which correspond to the stable and unstable manifolds of the toral folded saddle, respectively, when considered as an equilibrium of the desingularized flow \eqref{eq:Rudesing}. The singular canard in this case plays the role of a separatrix, which divides the phase space $\mathcal{S}$ between those trajectories that encounter the fold curve $\mathcal{L}$, and those that turn away from it. 

The unfolding in $\eps$ of the toral folded saddle of \eqref{eq:Ruslowfast} shows that only the singular canard persists as a transverse intersection of the invariant slow manifolds. There is no oscillatory behaviour associated with this maximal canard, and only those trajectories that are exponentially close to the maximal canard can follow the repelling slow manifold. Returning to the original $(x,y)$ variables, the toral folded saddle has precisely one maximal torus canard solution, which plays the role of a separatrix between those solutions that fall off the manifold of limit cycles at the manifold of SNPOs, and those that turn away from $\mathcal{P}_L$ and stay on $\mathcal{P}$. 

We remark that the toral faux canard of a toral folded saddle plays the role of an axis of rotation for local oscillatory solutions of \eqref{eq:Ruslowfast}, analogous to the toral weak canard in the TFN case. However, these oscillatory solutions of the averaged radial-slow system start on $\mathcal{S}_r^{\eps}$ and move to $\mathcal{S}_a^{\eps}$. That is, there is a family of faux torus canard solutions associated to the toral folded saddle that start on $\mathcal{P}_r^{\eps}$ and move to $\mathcal{P}_a^{\eps}$. Since these solutions are inherently unstable, we leave further investigation of their dynamics to future work.

%----------------------------------------------------
\subsection{Toral Folded Saddle-Node of Type II \& Torus Bifurcation} 	\label{subsec:toralFSN}
%----------------------------------------------------

In canard theory, the special case in which one of the eigenvalues of the folded singularity is zero is called a folded saddle-node (FSN). The FSN comes in a variety of flavours, each corresponding to a different codimension-1 bifurcation of the desingularized reduced flow. The most common types seen in applications are the FSN I \cite{Vo2015} and the FSN II \cite{Krupa2010}. The FSN I occurs when a folded node and a folded saddle collide and annihilate each other in a saddle-node bifurcation of folded singularities. Geometrically, the center manifold of the FSN I is tangent to the fold curve. It has been shown that $\mathcal{O}(\eps^{-1/4})$ canards persist near the FSN I limit for sufficiently small $\eps$. The FSN II occurs when a folded singularity and an ordinary singularity coalesce and swap stability in a transcritical bifurcation of the desingularized reduced flow \eqref{eq:Rudesing}. In this case, the center manifold of the FSN II is transverse to the fold curve and it has been shown that $\mathcal{O}(\eps^{-1/2})$ canards persist for sufficiently small $\eps$ \cite{Krupa2010}. 

Here we show that toral folded singularities may also be of the FSN types. We focus on the toral FSN II and its implications for torus canards. In analogy with the classic FSN II points, we define a \emph{toral FSN II} to be a FSN II of the averaged radial-slow system \eqref{eq:Ruslowfast}. These occur when an ordinary singularity of \eqref{eq:Ruslowfast} crosses the fold curve, i.e., under the conditions 
\[ F=0, \quad F_R=0, \quad G_1 =0, \quad \text{ and } \quad G_2=0, \]
in which case the folded singularity automatically has a zero eigenvalue. In the classic FSN II, there is always a Hopf bifurcation at an $\mathcal{O}(\eps)$-distance from the fold curve \cite{Guckenheimer2008,Krupa2010}. Since the toral FSN II is (by definition) a FSN II of system \eqref{eq:Ruslowfast}, we have that system \eqref{eq:Ruslowfast} will possess a Hopf bifurcation located at an $\mathcal{O}(\eps)$-distance from the fold curve. 

In terms of the original, non-averaged $(x,y)$ coordinates, the toral FSN II is detected as a limit cycle, $\Gamma$, of the layer problem such that
\[ \int_0^{T(y)} \operatorname{tr}\, D_x f \, dt = 0, \quad \int_0^{T(y)} g_1 \, dt =0, \quad \text{ and } \quad \int_0^{T(y)} g_2 \, dt =0, \]
along $\Gamma$. That is, the toral FSN II occurs when the averaged slow nullclines intersect the manifold of SNPOs. Moreover, the averaged radial-slow dynamics undergo a singular Hopf bifurcation \cite{Guckenheimer2008,Krupa2010,Kuehn2015}, from which a family of small-amplitude limit cycles of the averaged radial-slow system emanate. This creation of limit cycles in the radial envelope corresponds to the birth of an invariant phase space torus in the non-averaged, fully perturbed problem.  As such, we conjecture that the toral FSN II unfolds in $\eps$ to a singular torus bifurcation of the fully perturbed problem. We provide numerical evidence to support this conjecture in Sections \ref{subsec:PHsingvsnonsing}, \ref{subsubsec:MLTtoralFS}, \ref{subsec:TCHR}, and \ref{subsec:TCWCI}.

%---------------------------------------------------------------------------------
\section{The Politi-H\"{o}fer Model for Intracellular Calcium Dynamics}		\label{sec:PH}
%---------------------------------------------------------------------------------

We now demonstrate (in Sections \ref{sec:PH} -- \ref{sec:TCIMMO}) the predictive power of our analytic framework for generic torus canards, developed in Sections \ref{sec:averaging} and \ref{sec:classification}, in the Politi-H\"{o}fer (PH) model \cite{Politi2006}. This model describes the interaction between calcium transport processes and the metabolism of inositol (1,4,5)-trisphosphate (\ip3), which is a calcium-releasing messenger. In Section \ref{subsec:PHmodel}, we describe the PH model for intracellular calcium dynamics. In Section \ref{subsec:PHbifn}, we investigate the bifurcation structure of the PH model, and report on a novel class of amplitude-modulated subcritical elliptic bursting rhythms. 
In the parameter space, these exist between the tonic spiking and bursting regimes. 
In Section \ref{subsec:PHlayer} we formally show that the PH model is a 2-fast/2-slow system and follow in Section \ref{subsec:PHgeometry} by showing that it falls under the framework of our torus canard theory. 
% examine the geometric configuration of the layer problem. 

%----------------------------------------------------
\subsection{The Politi-H\"{o}fer Model} 	\label{subsec:PHmodel}
%----------------------------------------------------

Changes in the concentration of free intracellular calcium play a crucial role in the biological function of most cell types \cite{Keener2008}. In many of these cells, the calcium concentration is seen to oscillate, and an understanding of how these oscillations arise and determining the mechanisms that generate them is a significant mathematical and biological pursuit.   

The biological process in which calcium is able to activate calcium release from internal stores is known as calcium-induced calcium release. The sequence of events leading to calcium-induced calcium release is as follows. An agonist binds to a receptor in the external plasma membrane of a cell, which initiates a chain of reactions that lead to the release of \ip3 inside that cell. The \ip3 binds to \ip3 receptors on the endoplasmic reticulum, which leads to the release of calcium from the internal store through the \ip3 receptors. 

That there are calcium oscillations is indicative of the fact that there are feedback mechanisms from calcium to the metabolism of \ip3 at work. Mathematical modelling of these feedback mechanisms is broadly split into three classes. Class I models assume that the \ip3 receptors are quickly activated by the binding of calcium, and then slowly inactivated by slow binding of the calcium to separate binding sites. That is, class I models feature sequential (fast) positive and then (slow) negative feedback on the \ip3 receptor. Class II models assume that the calcium itself regulates the production and degradation rates of \ip3, which provides an alternative mechanism for negative and positive feedback. The most biologically realistic scenario incorporates both mechanisms (i.e., calcium feedback on \ip3 receptor dynamics as well as on \ip3 dynamics), and such models are known as hybrid models. 

One hybrid model for calcium oscillations is the PH model \cite{Politi2006}, which includes four variables; the calcium concentration $c$ in the cytoplasm, the calcium concentration $c_t$ in the endoplasmic reticulum stores, the fraction $r$ of \ip3 receptors that have not been inactivated by calcium, and the concentration $p$ of \ip3 in the cytoplasm. The calcium flux through the \ip3 receptors is given by
\begin{align}
J_{\text{release}} = \left( k_1 \left( r \frac{c}{K_a+c}\, \frac{p}{K_p+p} \right)^3 + k_2 \right) (\gamma c_t - (1+\gamma)c),
\end{align}
where $\gamma$ is the ratio of the cytosolic volume to the endoplasmic reticulum volume. The active transport of calcium across the endoplasmic reticulum (via sarco/endoplasmic reticulum ATP-ase or SERCA pumps) and plasma membrane are given respectively by the Hill functions
\begin{align}
J_{\text{serca}} &= V_{\text{serca}} \frac{c^2}{K_{\text{serca}}^2+c^2}, \,\, \text{ and }\,\, 
J_{\text{pm}} = V_{\text{pm}} \frac{c^2}{K_{\text{pm}}^2+c^2}.
\end{align}
The calcium flux into the cell via the plasma membrane is given by 
\begin{align} 
J_{\text{in}} = \nu_0 + \phi V_{PLC}, 
\end{align}
where $\nu_0$ is the leak into the cell, and $V_{PLC}$ is the steady-state concentration of \ip3 in the absence of any feedback effects of calcium on \ip3 concentration. 

The PH model equations are then given by
\begin{equation}	\label{eq:PHmodel}
\begin{split}
\dot{c} &= J_{\text{release}}-J_{\text{serca}}+\delta \left( J_{\text{in}} - J_{\text{pm}} \right), \\
\dot{c}_t &= \delta \left( J_{\text{in}} - J_{\text{pm}} \right), \\
\dot{r} &= \frac{1}{\tau_r} \left( 1-r \frac{K_i+c}{K_i} \right), \\
\dot{p} &= k_{3K} \left( V_{PLC} - \frac{c^2}{K_{3K}^2+c^2} p \right).
\end{split}
\end{equation}
Here the $c$ and $c_t$ equations describe the balance of calcium flux across the plasma membrane ($J_{\text{in}} - J_{\text{pm}}$) of the cell and the endoplasmic reticulum ($J_{\text{release}}-J_{\text{serca}}$). The  parameter $\delta$ measures the relative strength of the plasma membrane flux to the flux across the endoplasmic reticulum. Note that if $\delta =0$, then this creates a closed cell model wherein the total calcium in the cell is conserved. The $r$-equation describes the inactivation of \ip3 receptors by calcium whilst the $p$-equation describes the balance of (calcium-independent) \ip3 production and calcium-activated \ip3 degradation. The kinetic parameters, their standard values, and their biological significance are detailed in Table \ref{tab:PHparams} (Appendix \ref{app:PH}). Unless stated otherwise, all parameters will be fixed at these standard values.

%----------------------------------------------------
\subsection{Dynamics of the Politi-H\"{o}fer Model} 	\label{subsec:PHbifn}
%----------------------------------------------------

Following \cite{Harvey2011}, we take $V_{PLC}$ to be the principal bifurcation parameter, since it is relatively easy to manipulate in an experimental setting. The parameter $V_{PLC}$ represents the steady-state \ip3 concentration in the absence of calcium feedback. Note that the maximal rate of \ip3 formation is given by $k_{3K} V_{PLC}$. Variations in $V_{PLC}$ can generate a wide array of different behaviours in the PH model. Representative traces are shown in Figure \ref{fig:PHbifn}. 

\begin{figure}[ht!!]
\centering
\includegraphics[width=5in]{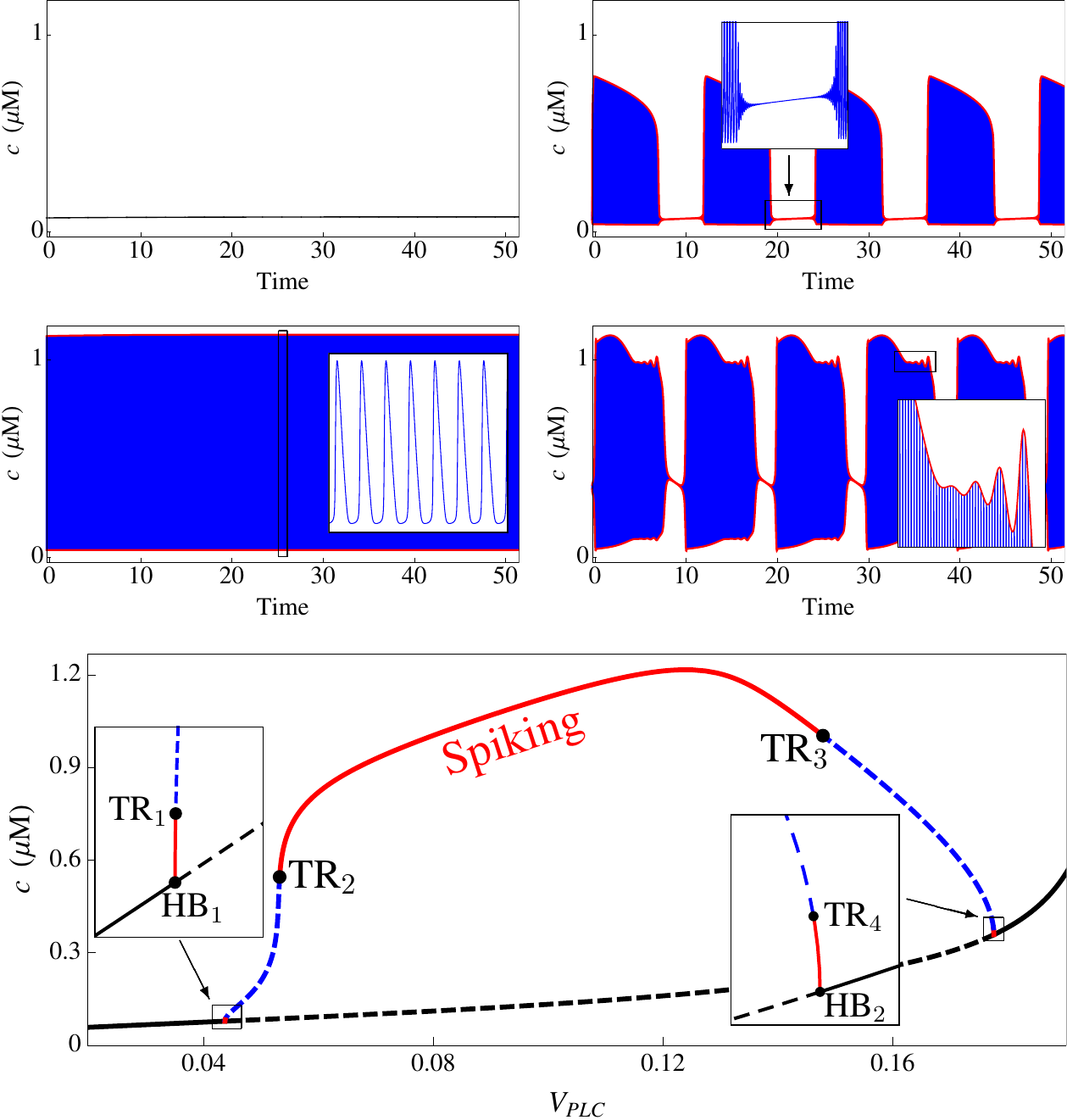}
\put(-364,372){(a)}
\put(-181,372){(b)}
\put(-364,262){(c)}
\put(-181,262){(d)}
\put(-364,150){(e)}
\caption{Dynamics of the PH model \eqref{eq:PHmodel} for $\eps=0.0035$ (i.e., $\delta = 0.472938$), and (a) $V_{PLC}=0.04\, \mu$M, (b) $V_{PLC} = 0.05 \, \mu$M, (c) $V_{PLC}=0.1\, \mu$M, and (d) $V_{PLC}=0.149\, \mu$M. (a) The attractor is a stable equilibrium. (b) The system exhibits subcritical elliptic bursting. The inset shows the small oscillations due to slow passage through a delayed Hopf bifurcation. (c) The model can also generate rapid spiking solutions. (d) The subcritical elliptic bursts here feature amplitude-modulation during the active burst phase. Inset: magnified view of the oscillations in the envelope of the waveform. (e) Bifurcation structure of \eqref{eq:PHmodel} with respect to $V_{PLC}$. The equilibria (black) change stability at Hopf bifurcations (HB). Emanating from the Hopf bifurcations are families of limit cycles, which change stability at torus bifurcations (TR). The torus bifurcations act as the boundaries between spiking (red) and bursting (blue) regions.}
\label{fig:PHbifn}
\end{figure}
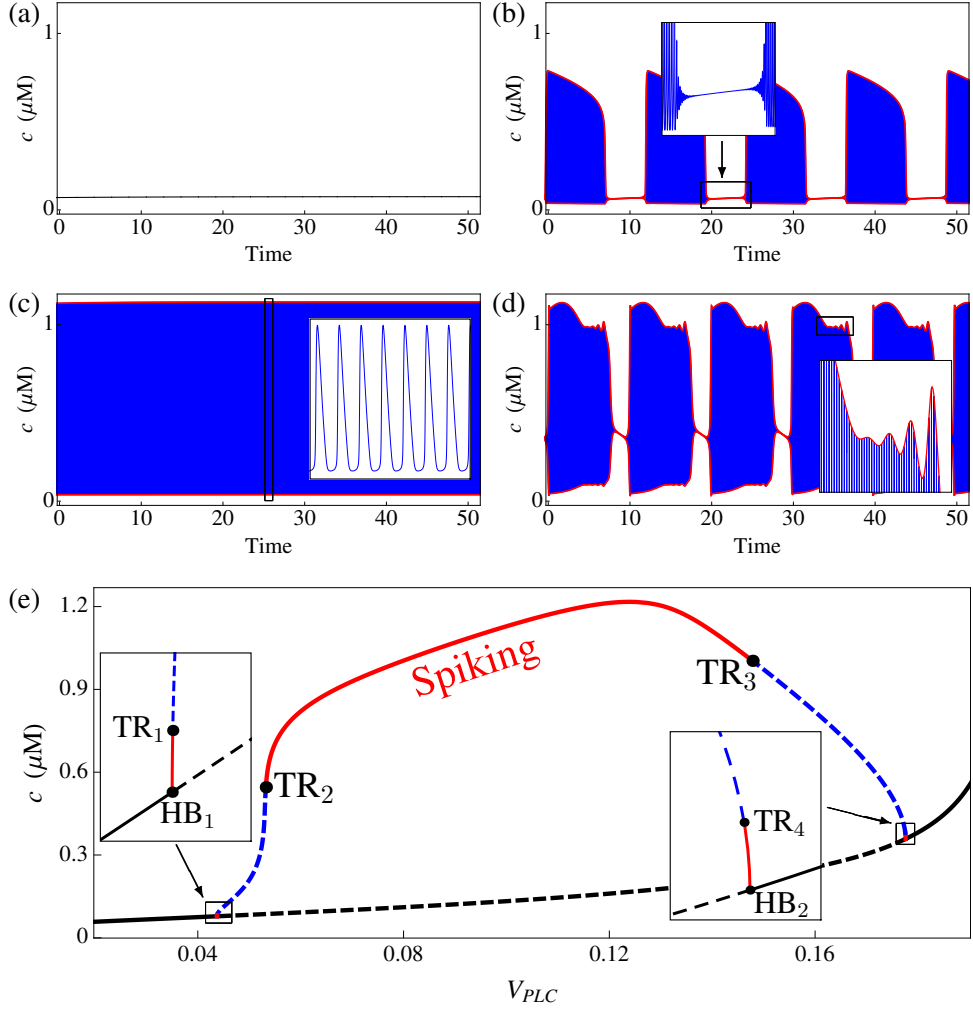

For small \ip3 production rates (i.e., small $V_{PLC}$), the negative feedback of calcium on \ip3 metabolism overwhelms the production of \ip3. There are no calcium oscillations, and the system settles to a stable equilibrium (Figure \ref{fig:PHbifn}(a)). With increased \ip3 production rate, the system undergoes a supercritical Hopf bifurcation (labelled HB$_1$) at $V_{PLC} \approx 0.04370\,\mu$M, from which stable periodic orbits emanate. This family of periodics becomes unstable at a torus bifurcation (TR$_1$) at $V_{PLC} \approx 0.04734\,\mu$M and the stable spiking solutions give way to bursting trajectories (Figure \ref{fig:PHbifn}(b)). These bursting solutions are in fact subcritical elliptic (or subHopf/fold-cycle) bursts \cite{Izhikevich2000}. A subcritical elliptic burster has two primary bifurcations that determine its outcome. The active phase of the burst is initiated when the trajectory passes through a subcritical Hopf bifurcation of the layer problem, and terminates when the trajectory reaches a SNPO and falls off the manifold of limit cycles of the layer flow. These subcritical elliptic bursts persist in $V_{PLC}$ until there is another torus bifurcation (TR$_2$) at $V_{PLC} \approx 0.05323\,\mu$M, after which the system exhibits rapid spiking (Figure \ref{fig:PHbifn}(c)). 

The branch of spiking solutions remains stable until another torus bifurcation (TR$_3$) at $V_{PLC} \approx 0.1479~\mu$M is encountered. 
Initially, for $V_{PLC}$ values $\mathcal{O}(\eps)$ close to TR$_3$, the system exhibits amplitude-modulated spiking (not shown). 
The amplitude modulated spiking only exists on a very thin $V_{PLC}$ interval. Moreover, the amplitude modulation becomes more dramatic as $V_{PLC}$ increases until $V_{PLC} \approx 0.1489~\mu$M, after which the trajectory is a novel type of solution that combines features of amplitude-modulated spiking and bursting. These hybrid \emph{amplitude-modulated bursting} (AMB) solutions appear to be subcritical elliptic bursts with the added twist that there is amplitude modulation in the envelope of the waveform during the active burst phase (compare Figures \ref{fig:PHbifn}(b) and (d)). We will carefully examine these AMB rhythms in Section \ref{sec:TCIMMO}.

For sufficiently large $V_{PLC}$, we recover subcritical elliptic bursting solutions like those shown in Figure \ref{fig:PHbifn}(b), but with increasingly long silent phases. Eventually these subcritical elliptic bursting solutions disappear in the torus bifurcation TR$_4$ at $V_{PLC} \approx 0.1775~\mu$M and the attractor of the system is a spiking solution, which disappears in a supercritical Hopf bifurcation (HB$_2$) at $V_{PLC} \approx 0.1776\,\mu$M. The bifurcation structure of \eqref{eq:PHmodel} with respect to $V_{PLC}$ described above was computed using AUTO \cite{Doedel1981,Doedel2007} and is shown in Figure \ref{fig:PHbifn}(e).

Two primary features of our bifurcation analysis here signal the presence of multiple time-scale dynamics in system \eqref{eq:PHmodel}. Firstly, the presence of trajectories that have epochs of rapid spiking interspersed with silent phases (i.e., the bursting solutions) indicates that there is an intrinsic slow/fast structure. Second, the transition from rapid spiking to bursting via a torus bifurcation, together with the appearance of amplitude-modulated waveforms suggests the presence of torus canards, which naturally arise in slow/fast systems. Motivated by this, we now turn our attention to the problem of understanding the underlying mechanisms that generate these novel bursting rhythms. We will use the analytic framework developed in Sections \ref{sec:averaging} and \ref{sec:classification} as the basis of our understanding.

%----------------------------------------------------
\subsection{Slow/Fast Decomposition of the Politi-H\"{o}fer Model} 	\label{subsec:PHlayer}
%----------------------------------------------------

To demonstrate the existence of a separation of time-scales in system \eqref{eq:PHmodel}, we first perform a dimensional analysis, following a procedure similar to that of \cite{Harvey2011}. We define new dimensionless variables $C, C_t, P$, and $t_s$ via
\[ c = Q_c C, \quad c_t = Q_c C_t, \quad p = Q_p P, \,\, \text{ and }\,\, t = Q_t \, t_s, \]
where $Q_c$ and $Q_p$ are reference calcium and \ip3 concentrations, respectively, and $Q_t$ is a reference time-scale. Details of the non-dimensionalization are given in Appendix \ref{app:PH}. For the parameter set in Table \ref{tab:PHparams}, natural choices for $Q_c$ and $Q_p$ are $Q_c=1\, \mu$M and $Q_p = 1\, \mu$M. With these choices, a typical time scale for the dynamics of the calcium concentration $c$ is given by $T_c = \frac{Q_c}{V_{\text{serca}} \gamma} \approx 0.74$ s. The $c_t$ dynamics evolve much more slowly with a typical time scale $T_{c_t} = \frac{Q_c}{\delta V_{\text{pm}}} \approx 200$ s. The $r$ dynamics have a typical time-scale $T_r = \tau_r = 6.6$ s. The time-scale $T_p = \frac{Q_p}{k_{3k} V_{PLC}}$ for the $p$ dynamics depends on $V_{PLC}$ and can range from $50$ s (for $V_{PLC}=0.2~\mu$M) to $1000$ s (for $V_{PLC}=0.01~\mu$M). 

Setting the reference time-scale to be the slow time-scale (i.e., $Q_t = T_{c_t}$), and defining the dimensionless parameters
\[ \eps := \frac{\delta V_{\text{pm}}}{V_{\text{serca}} \gamma}, \,\, \hat{\tau}_r = \tau_r \frac{V_{\text{serca}} \gamma}{Q_c}, \,\, \text{ and }\,\, \hat{k}_{3K} = \frac{k_{3K}Q_c}{\delta V_{\text{pm}}}, \]
leads to the dimensionless version of the PH model
\begin{equation}	\label{eq:PHdimless}
\begin{split}
\eps \dot{C} &= f_1(C,r,C_t,P)+\eps g_1(C,r,C_t,P), \\
\eps \dot{r} &= f_2(C,r,C_t,P), \\
\dot{C_t} &= g_1(C,r,C_t,P), \\
\dot{P} &= g_2(C,r,C_t,P),
\end{split}
\end{equation}
where the overdot denotes derivatives with respect to $t_s$, and the functions $f_1,f_2,g_1$, and $g_2$ are given in Appendix \ref{app:PH}.
For the parameter set in Appendix \ref{app:PH}, we have that $\eps=0.0035$ is small. Thus, for a large regime of parameter space, the PH model is singularly perturbed, with two fast variables ($c,r$) and two slow variables ($c_t,p$).

%----------------------------------------------------
\subsection{The Geometry of the Layer Problem} 	\label{subsec:PHgeometry}
%----------------------------------------------------

%We now seek toral folded singularities and hence torus canard solutions using our averaging method. 
We now show that the PH model satsfies Assumptions \ref{ass:man} -- \ref{ass:homoclinic}.  
The first step is to examine the bifurcation structure of the layer problem
\begin{equation} \label{eq:PHlayer}
\begin{split}
C^{\prime} &= f_1(C,r,C_t,P), \\
r^{\prime} &= f_2(C,r,C_t,P), 
\end{split}
\end{equation}
where the prime denotes derivatives with respect to the (dimensionless) fast time $t_f$, which is related (for non-zero perturbations) to the dimensionless slow time $t_s$ by $t_s = \eps t_f$. The geometric configuration of the layer problem \eqref{eq:PHlayer} is illustrated in Figure \ref{fig:PHlayerbifn}. 

\begin{figure}[ht!!]
\centering
\includegraphics[width=3.25in]{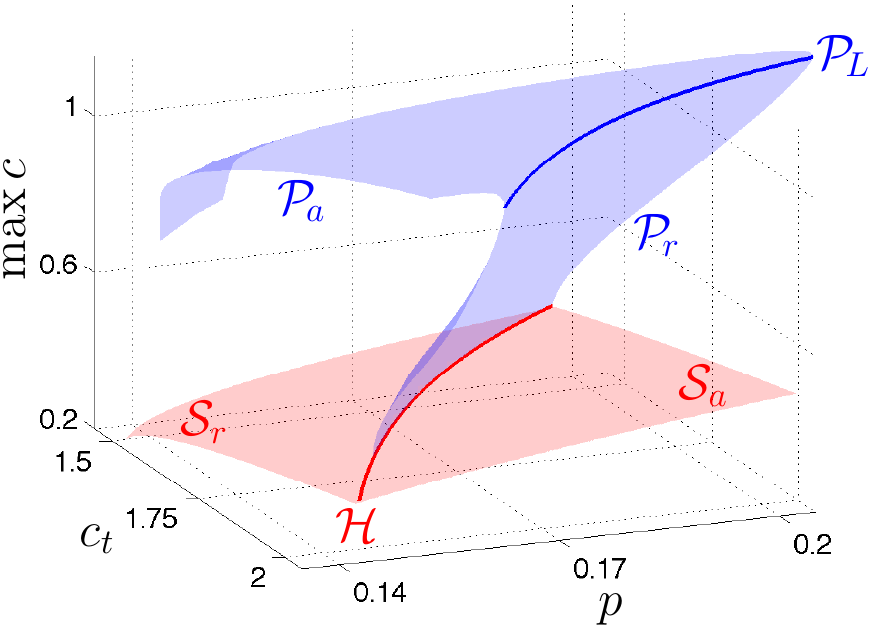}
\caption{Bifurcation structure of the layer problem \eqref{eq:PHlayer}, which is independent of $V_{PLC}$. The critical manifold (red surface) possesses a curve of subcritical Hopf bifurcations $\mathcal{H}$ (red curve) that separates the attracting and repelling sheets, $\mathcal{S}_a$ and $\mathcal{S}_r$.
%
%is cubic-shaped, with fold curves $\mathcal{L}^{\pm}$ (black) that meet in a cusp bifurcation. There are also curves of Hopf bifurcations $\mathcal{H}^{\pm}$ (red). The region enclosed by the upper Hopf curve, and the union of the lower Hopf curve together with the lower fold curve is the repelling sheet, $\mathcal{S}_r$, of the critical manifold. The outer sheets, $\mathcal{S}_a^{\pm}$, of the critical manifold are attracting. 
Also shown is the maximum $c$-value for the manifold of limit cycles (blue surface) emanating from $\mathcal{H}$. The manifold of limit cycles consists of attracting and repelling subsets, $\mathcal{P}_a$ and $\mathcal{P}_r$, which meet in a manifold of SNPOs, $\mathcal{P}_L$.}
\label{fig:PHlayerbifn}
\end{figure}

System \eqref{eq:PHlayer} has a critical manifold, $\mathcal{S}$, with a curve of subcritical Hopf bifurcations, $\mathcal{H}$, that divide $\mathcal{S}$ between its attracting and repelling sheets, $\mathcal{S}_a$ and $\mathcal{S}_r$, respectively. 
The manifold, $\mathcal{P}_r$, of limit cycles that emerges from $\mathcal{H}$ is repelling. The repelling family of limit cycles meets an attracting family of limit cycles, $\mathcal{P}_a$, at a manifold of SNPOs, $\mathcal{P}_L$. Thus, the PH model has precisely the geometric configuration described in Section \ref{subsec:assumptions}. 

\begin{remark}
The bifurcation structure of \eqref{eq:PHlayer} contains many other features outside of the region shown here. The full critical manifold is cubic-shaped. Only part of that lies in the region shown in Figure \ref{fig:PHlayerbifn}, and outside this region, there are additional curves of fold and Hopf bifurcations, cusp bifurcations, and Bogdanov-Takens bifurcations. In this work, we are only concerned with the region of phase space presented in Figure \ref{fig:PHlayerbifn}.
\end{remark}

%---------------------------------------------------------------------------------
\section{Generic Torus Canards in the Politi-H\"{o}fer Model}		\label{sec:PHTC}
%---------------------------------------------------------------------------------

We now apply the results of Section \ref{sec:classification} to the PH model.  
In Section \ref{subsec:PHtoralFS}, we show that the PH model possesses toral folded singularities. We carefully examine the geometry and maximal torus canards in the case of a TFN in Section \ref{subsec:PHmanifolds}. In order to do so, we must compute the invariant manifolds of limit cycles, the numerical method for which is outlined in Section \ref{subsec:PHnumerical}. 
We then show in Section \ref{subsec:PHtoralFSclass} that the torus canards are generic and robust phenomena, and occur on open parameter sets. In this manner, we demonstrate the practical utility of our torus canard theory. 

%----------------------------------------------------
\subsection{Identification of Toral Folded Singularities: A Representative Example} 	\label{subsec:PHtoralFS}
%----------------------------------------------------

We now proceed to locate and classify toral folded singularities of the PH model.
For each limit cycle in $\mathcal{P}_L$, we numerically check the condition for toral folded singularities given in equation \eqref{eq:switching}, 
\[ \rho_0 := \overline{a}_1 \, \overline{g}_1 + \overline{a}_2 \, \overline{g}_2 =0, \]
where the overlined quantities are the averaged coefficients that appear in Theorem \ref{thm:averaging}. Recall, that the condition $\rho_0=0$ corresponds geometrically to the scenario in which the projection of the averaged slow drift into the slow variable plane is tangent to $\mathcal{P}_L$ (Figure \ref{fig:PHtoralfn}). For our computations, we take this condition to be satisfied if $\left| \rho_0 \right| < 10^{-9}$. 

For the representative parameter set given in Table \ref{tab:PHparams} (Appendix \ref{app:PH}), we find that the PH model possesses a toral folded singularity for $V_{PLC}= 0.149~\mu$M at 
\[ (c_t,p) = (c_t^*,p^*) \approx (1.912089102,0.174866719), \]
where $\rho_0 \approx -1.67 \times 10^{-10}$. 

\begin{figure}[h!]
\centering
\includegraphics[width=5in]{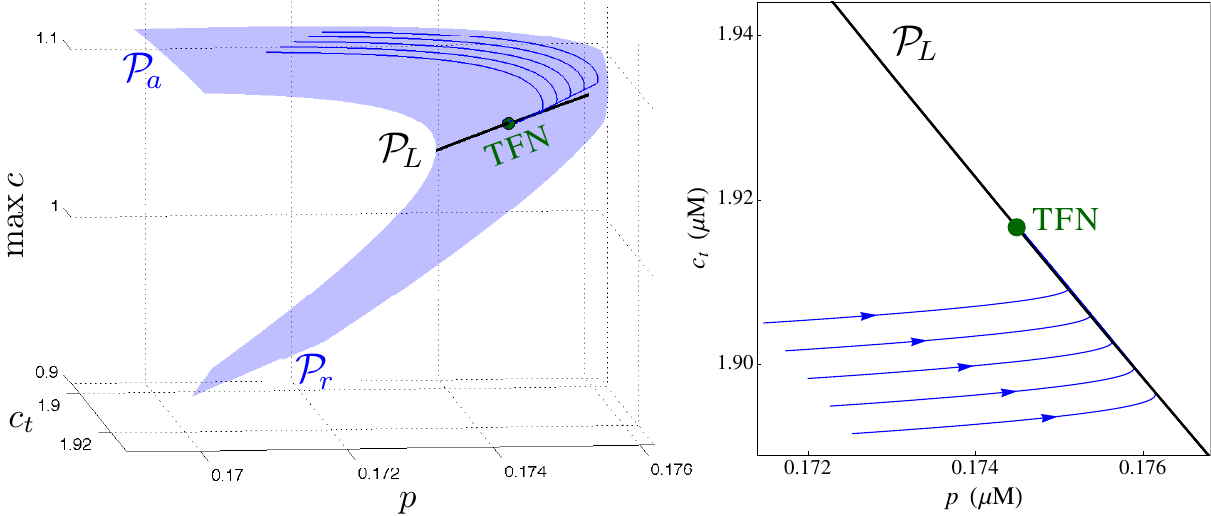}
\put(-364,146){(a)}
\put(-160,146){(b)}
\caption{TFN of the PH model for $V_{PLC}=0.149~\mu$M at $(c_t^*,p^*)$. (a) Projection of $\mathcal{P}$ in a neighbourhood of the TFN into the $(c_t,p,c)$ phase space. The blue curves show the envelopes of the slow drift along $\mathcal{P}$ for different initial conditions. (b) Projection into the slow variable plane. The singular slow drift along $\mathcal{P}$ is tangent to $\mathcal{P}_L$ at the TFN.}
\label{fig:PHtoralfn}
\end{figure}

Once the toral folded singularity has been located, we simply compute the remaining averaged coefficients from Theorem \ref{thm:averaging}, which allows us to compute the eigenvalues of the toral folded singularity and hence classify it according to the scheme in Definition \ref{def:class}. 
For the toral folded singularity at $(c_t^*,p^*)$, the eigenvalues are $\lambda_s \approx -3.2052$ and $\lambda_w \approx -0.127402$, so that we have a TFN. The associated eigenvalue ratio of the TFN is $\mu \approx 0.039749$, and so by \eqref{eq:smax}, the maximal number of oscillations that the envelope of the waveform can execute for sufficiently small $\eps$ is $s_{\max} = 13$ (compare with Figure \ref{fig:PHbifn}(d) where the envelope only oscillates 4 times).
We point out that the type of toral folded singularity can change with the parameters (see Section \ref{subsec:PHtoralFSclass}).

All other points on $\mathcal{P}_L$ are regular folded limit cycles. That is, $\rho_0 \neq 0$ at all other points on $\mathcal{P}_L$; and so, for $V_{PLC}=0.149~\mu$M, there is only a single TFN. 
Note that the TFN identified here is a simple zero of $\rho_0$, i.e., $\rho_0$ has opposite sign for points on $\mathcal{P}_L$ on either side of the TFN. 
This TFN is the natural candidate mechanism for generating torus canard dynamics.

%----------------------------------------------------
\subsection{Toral Folded Node Canards in the Politi-H\"{o}fer Model} 	\label{subsec:PHmanifolds}
%----------------------------------------------------

We now examine the geometry of the PH model in a neighbourhood of the TFN away from the singular limit. Recall that averaging theory \cite{Pontryagin1960,Sanders2007} together with Fenichel theory \cite{Fenichel1979,Jones1995} guarantees that normally hyperbolic manifolds of limit cycles, $\mathcal{P}_a$ and $\mathcal{P}_r$, persist as invariant manifolds of limit cycles, $\mathcal{P}_a^{\eps}$ and $\mathcal{P}_r^{\eps}$, for sufficiently small $\eps$. We showed in Section \ref{subsec:toralFN} that the extensions of $\mathcal{P}_a^{\eps}$ and $\mathcal{P}_r^{\eps}$ into a neighbourhood of a TFN results in a local twisting of these manifolds of limit cycles. 

\begin{figure}[h!!]
\centering
\includegraphics[width=5in]{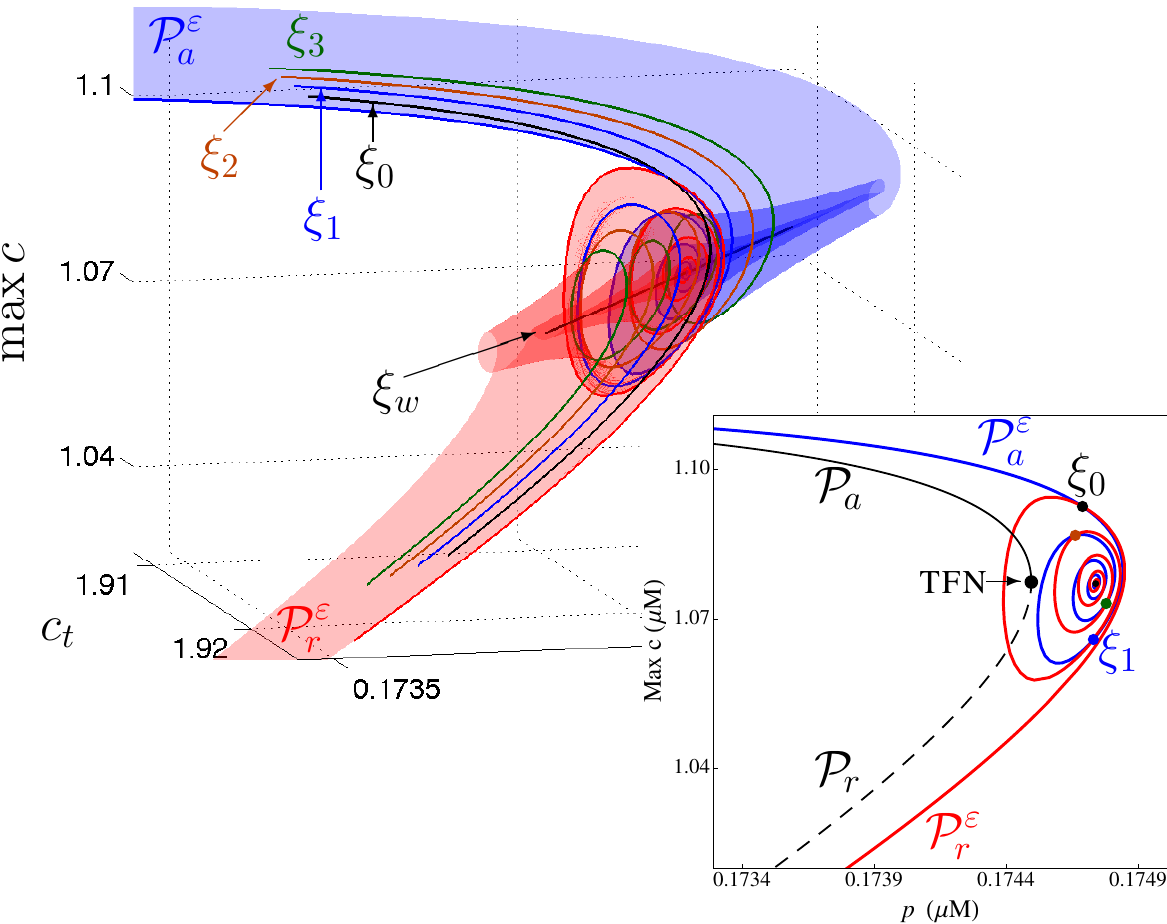}
\caption{Extensions of $\mathcal{P}_a^{\eps}$ and $\mathcal{P}_r^{\eps}$ into a neighbourhood of the TFN projected into the $(c_t,p,c)$ phase space for $V_{PLC}=0.149~\mu$M and $\eps=5 \times 10^{-5}$. The attracting invariant manifold of limit cycles (blue) is computed until it intersects the hyperplane $\Sigma: \{ c_t=c_t^*\}$. Similarly, the repelling invariant manifold of limit cycles (red) is computed up to its intersection with $\Sigma$. The inset shows $\mathcal{P}_a^{\eps} \cap \Sigma$ (blue) and $\mathcal{P}_r^{\eps}\cap \Sigma$ (red). Their singular limit counterparts, $\mathcal{P}_a \cap \Sigma$ and $\mathcal{P}_r \cap \Sigma$, are also shown for comparison. There are 13 intersections, $\xi_i$, $i=0,1,\ldots,12$, of the invariant manifolds, each corresponding to a maximal torus canard (with $\xi_{12}=:\xi_w$).}
\label{fig:PHinvariantmanifolds}
\end{figure}

Figure \ref{fig:PHinvariantmanifolds} demonstrates that the attracting and repelling manifolds of limit cycles twist in a neighbourhood of the TFN, and intersect a countable number of times. For $V_{PLC}=0.149~\mu$M, we find that there are 13 intersections, consistent with the prediction from Section \ref{subsec:PHtoralFS}. These intersections of $\mathcal{P}_a^{\eps}$ and $\mathcal{P}_r^{\eps}$ are the maximal torus canards (by Definition \ref{def:maximalTC}). The outermost intersection of $\mathcal{P}_a^{\eps}$ and $\mathcal{P}_r^{\eps}$, denoted $\xi_0$, is the maximal strong torus canard. The intersections, $\xi_i$, $i=1,\ldots, 11$, are the maximal secondary torus canards. The innermost intersection is the maximal weak torus canard, $\xi_w$.  

The maximal strong torus canard, $\xi_0$, is the local phase space separatrix that divides between rapidly oscillating solutions that exhibit amplitude modulation and those that do not. The maximal weak torus canard, $\xi_w$, plays the role of a local axis of rotation. That is, the invariant manifolds twist around $\xi_w$. The maximal secondary torus canards partition $\mathcal{P}_a^{\eps}$ and $\mathcal{P}_r^{\eps}$ into rotational sectors. Every orbit segment on $\mathcal{P}_a^{\eps}$ between $\xi_{n-1}$ and $\xi_{n}$ for $n=1,2,\ldots,13$, is an amplitude-modulated waveform where the envelope executes $n$ oscillations in a neighbourhood of the TFN.  

\begin{figure}[h!!]
\centering
\includegraphics[width=5in]{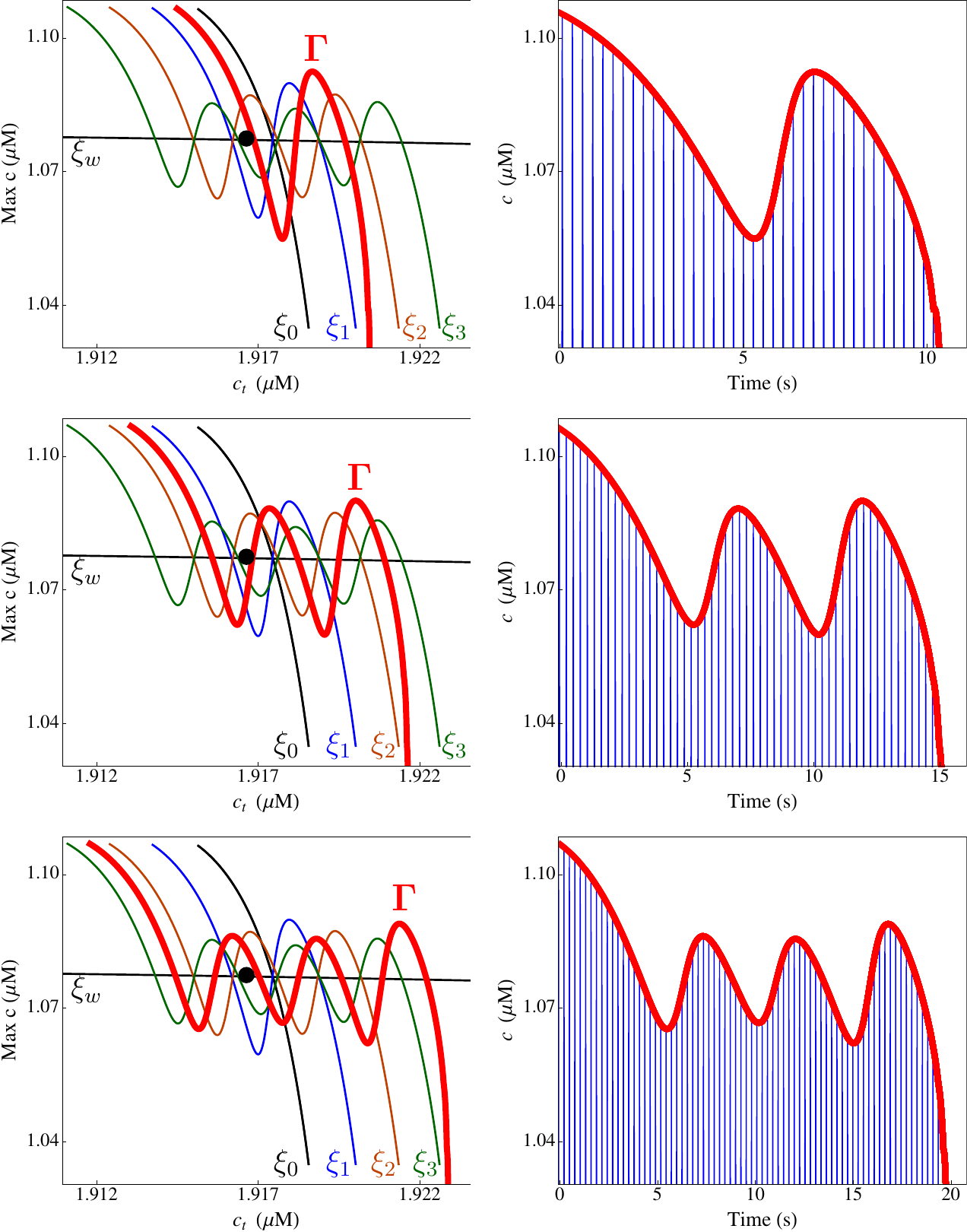}
\put(-364,452){(a)}
\put(-364,296){(b)}
\put(-364,140){(c)}
\caption{Sectors of amplitude modulation formed by the maximal torus canards for the PH model for $V_{PLC}=0.149~\mu$M and $\eps=5 \times 10^{-5}$. Left column: projection of $\mathcal{P}_a^{\eps}$ and $\mathcal{P}_r^{\eps}$, onto the $(c_t,c)$ plane along with the maximal torus canards $\xi_n, n=0,1,2,3,w$. Also shown is the envelope of the transient solution, $\Gamma$, of \eqref{eq:PHmodel} in the (a) sector bounded by $\xi_0$ and $\xi_1$, (b) sector bounded by $\xi_1$ and $\xi_2$, and (c) sector bounded by $\xi_2$ and $\xi_3$. Right column: corresponding time traces of the transient solution $\Gamma$. Note that time is given in seconds.}
\label{fig:PHsectors}
\end{figure}

Figure \ref{fig:PHsectors} illustrates the sectors of amplitude-modulation formed by the maximal torus canards. For fixed parameters, it is possible to change the number of oscillations in the envelope of the rapidly oscillating waveform by adjusting the initial condition. More specifically, the trajectory of \eqref{eq:PHmodel} for an initial condition on $\mathcal{P}_a^{\eps}$ between $\xi_0$ and $\xi_1$ is an AMB with one oscillation in the envelope (Figure \ref{fig:PHsectors}(a)). By changing the initial condition to lie in the rotational sector bounded by $\xi_1$ and $\xi_2$ (Figure \ref{fig:PHsectors}(b)), the amplitude-modulated waveform exhibits two oscillations in its envelope. The deeper into the funnel of the TFN, the smaller and more numerous the oscillations in the envelope of the rapidly oscillating waveform (Figure \ref{fig:PHsectors}(c)). Moreover, each oscillation significantly extends the burst duration.

%----------------------------------------------------
\subsection{Numerical Computation of Invariant Manifolds of Limit Cycles} 	\label{subsec:PHnumerical}
%----------------------------------------------------

Figure \ref{fig:PHinvariantmanifolds} is the first instance of the numerical computation of twisted, intersecting, invariant manifolds of limit cycles of a slow/fast system with at least two fast and two slow variables. Here, we outline the numerical method (inspired by the homotopic continuation algorithms for maximal canards of folded singularities \cite{Desroches2008,Desroches2010}) used to generate Figure \ref{fig:PHinvariantmanifolds}. 

The idea of the computation of $\mathcal{P}_a^{\eps}$ is to take a set of initial conditions on $\mathcal{P}_a$ sufficiently far from both the TFN and manifold of SNPOs, and flow it forward until the trajectories reach the hyperplane 
\[ \Sigma := \left\{ (c,r,c_t,p): c_t = c_t^* \right\}, \]
where $c_t^*$ is the $c_t$ coordinate of the TFN identified in Section \ref{subsec:PHtoralFS}. 
This generates a family of rapidly oscillating solutions that form a mesh of the manifold $\mathcal{P}_a^{\eps}$. The envelope of each of those rapidly oscillating solutions is then used to form a mesh of the projection of $\mathcal{P}_a^{\eps}$. 

\begin{figure}[h!!]
\centering
\includegraphics[width=3.25in]{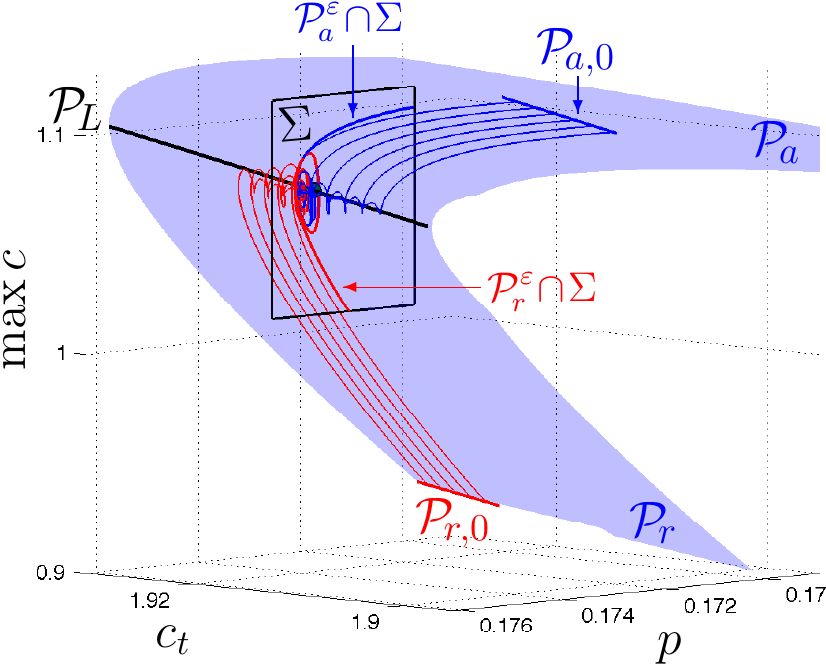}
\caption{Numerical computation of $\mathcal{P}_a^{\eps}$ and $\mathcal{P}_r^{\eps}$ for the PH model for the same parameter set as Figure \ref{fig:PHinvariantmanifolds}. The attracting invariant manifold of limit cycles, $\mathcal{P}_a^{\eps}$, is computed by taking the set $\mathcal{P}_{a,0} \subset \mathcal{P}_a$ and flowing it forward until it intersects the hyperplane $\Sigma$. The blue curves illustrate the behaviour of the envelopes of the rapidly oscillating orbit segments that comprise $\mathcal{P}_a^{\eps}$. Similarly, the repelling invariant manifold of limit cycles, $\mathcal{P}_r^{\eps}$, is computed by flowing $\mathcal{P}_{r,0} \subset \mathcal{P}_r$ backwards in time up to the hyperplane $\Sigma$. The red curves illustrate the envelopes of these rapidly oscillating orbit segments that comprise $\mathcal{P}_r^{\eps}$.}
\label{fig:PHnumericalmethod}
\end{figure}

To initialise the computation, a suitable set of initial conditions must be chosen. Note that the projection of $\mathcal{P}_L$ into the slow variable plane is a curve, $\mathcal{C}$, say (see Figure \ref{fig:PHtoralfn}(b)). We choose our initial conditions to be a manifold of attracting limit cycles, $\mathcal{P}_{a,0} \subset \mathcal{P}_a$, such that the projection of $\mathcal{P}_{a,0}$ into the $(c_t,p)$ plane is approximately parallel to $\mathcal{C}$, and is sufficiently far from $\mathcal{C}$ (Figure \ref{fig:PHnumericalmethod}).

Similarly, to compute $\mathcal{P}_r^{\eps}$, we initialize the computation by choosing a set of repelling limit cycles, $\mathcal{P}_{r,0} \subset \mathcal{P}_r$, where the projection of $\mathcal{P}_{r,0}$ into the $(c_t,p)$ plane is approximately parallel to, and sufficiently distant from, the curve $\mathcal{C}$. We then flow that set of initial conditions $\mathcal{P}_{r,0}$ backwards in time until the trajectory hits the hyperplane $\Sigma$. The envelopes of these rapidly oscillating trajectories are then used to visualize $\mathcal{P}_r^{\eps}$. 

To locate the maximal torus canard $\xi_n, n=0,1,2,\ldots, 12$, we locate the initial condition within $\mathcal{P}_{a,0}$ that forms the boundary between those orbit segments with $n$ oscillations in their envelope and orbit segments with $(n+1)$ oscillations in their envelope. 

We point out that whilst numerical methods exist for the computation and continuation of maximal canards of folded singularities \cite{Desroches2008}, these methods will not work for maximal torus canards. There are currently no existing methods to numerically continue $\mathcal{P}_a^{\eps}$ and $\mathcal{P}_r^{\eps}$. Consequently, there are no numerical methods that will allow for the numerical continuation of maximal torus canards in parameters, which is essential for detecting bifurcations of torus canards.

%----------------------------------------------------
\subsection{Genericity of the Torus Canards in the Politi-H\"{o}fer Model} 	\label{subsec:PHtoralFSclass}
%----------------------------------------------------

We have now carefully examined the torus canards associated to a TFN for a single parameter set. However, the PH model has two slow variables and so it supports TFNs on open parameter sets. The theoretical framework developed in Sections \ref{sec:averaging} and \ref{sec:classification} allows to determine how those TFNs and their associated torus canards depend on parameters. Figure \ref{fig:PHeigenvalues}(a) shows the eigenvalue ratio, $\mu$, of the toral folded singularity as a function of $V_{PLC}$, for instance.

\begin{figure}[h!!!t!]
\centering
\includegraphics[width=5in]{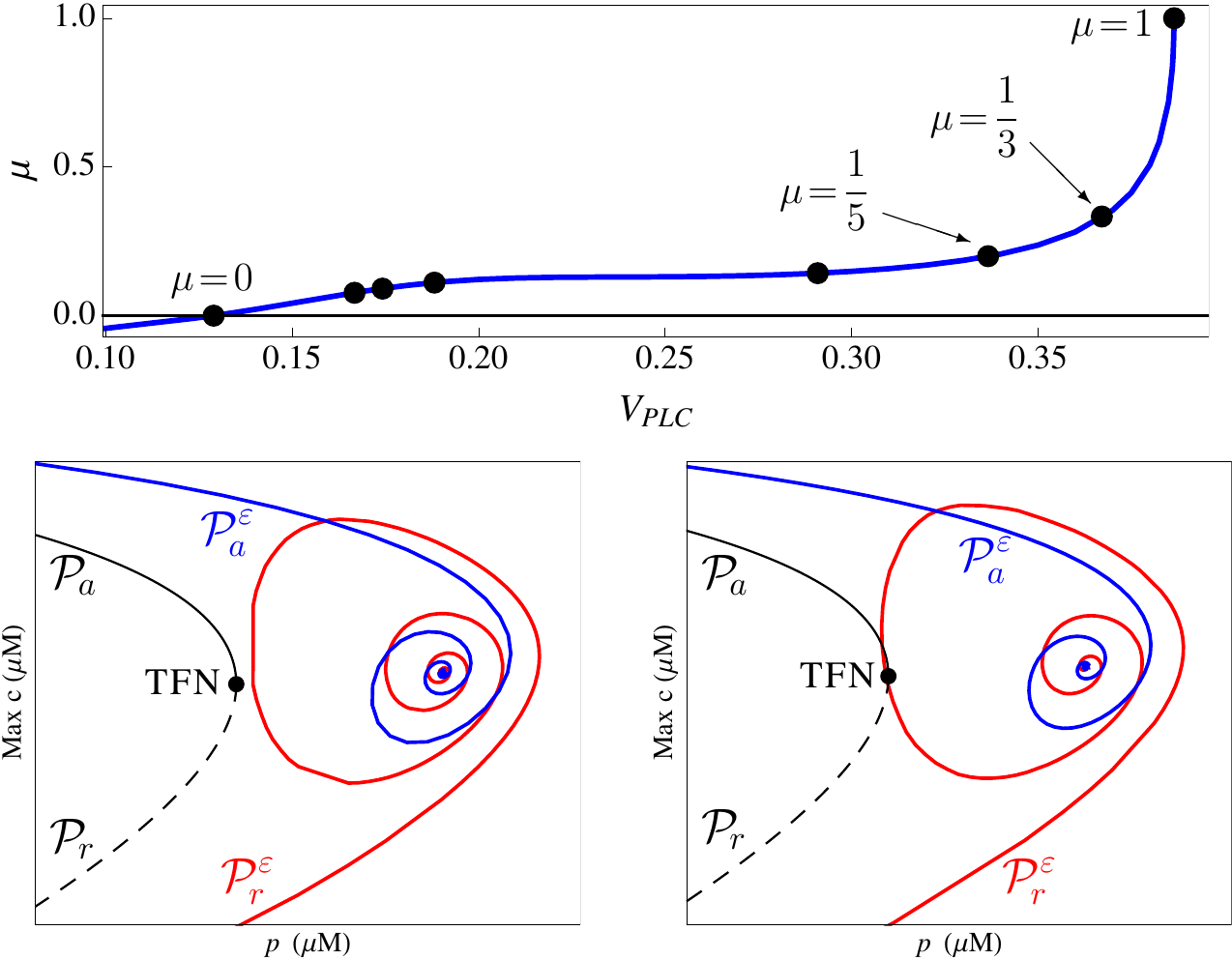}
\put(-360,274){(a)}
\put(-366,138){(b)}
\put(-180,138){(c)}
\caption{Dependence of the toral folded singularities and maximal torus canards on $V_{PLC}$. (a) The eigenvalue ratio, $\mu$, of the toral folded singularity as a function of $V_{PLC}$. The black markers indicate odd integer resonances in the eigenvalue ratio, where secondary torus canards bifurcate from the weak torus canard. There is a toral FSN II at $V_{PLC} \approx 0.129011~\mu$M. For $V_{PLC}<0.129011~\mu$M, the toral folded singularity is a toral folded saddle. Bottom row: $\mathcal{P}_a^{\eps}$ and $\mathcal{P}_r^{\eps}$, in a hyperplane passing through the TFN for $\eps = 5 \times 10^{-5}$ and (b) $V_{PLC} = 0.152~\mu$M where $s_{\max}=11$ and (c) $V_{PLC}=0.18~\mu$M where $s_{\max}=5$. In both cases, the invariant manifolds are shown in an $\mathcal{O}(\sqrt{\eps})$ neighbourhood of the TFN. Also shown are the attracting and repelling manifolds of limit cycles, $\mathcal{P}_a$ and $\mathcal{P}_r$, of the layer problem.}
\label{fig:PHeigenvalues}
\end{figure}

We find that the PH model has TFNs and hence torus canard dynamics for $0.129011~\mu$M $< V_{PLC} < 0.38642~\mu$M. 
%Moreover, the PH model is such that whenever there is a toral folded singularity, the averaged radial-slow system also possesses an ordinary singularity. This ordinary singularity of the averaged radial-slow flow is repelling. 
The black markers in Figure \ref{fig:PHeigenvalues}(a) indicate odd integer resonances in the eigenvalue ratio, $\mu$, of the TFN. 
These resonances signal the creation of new secondary canards in the averaged radial-slow system. As such, when $\mu^{-1}$ increases through an odd integer, we expect additional torus canards to appear. Figures \ref{fig:PHeigenvalues}(b) and (c) illustrate the mechanism by which these additional torus canards appear. Namely, as $V_{PLC}$ decreases and $\mu^{-1}$ increases, the invariant manifolds of limit cycles become more and more twisted, resulting in additional intersections.  

Thus, Figure \ref{fig:PHeigenvalues}(a) essentially determines the number of torus canards that exist for a given parameter value. An alternative viewpoint is that Figure \ref{fig:PHeigenvalues}(a) determines the maximal number of oscillations that the envelope of the rapidly oscillating waveform can execute. For example, for $V_{PLC}$ on the interval between $\mu = 1$ and $\mu=\frac{1}{3}$, the amplitude-modulated waveform can have, at most, one oscillation in the envelope. For $V_{PLC}$ on the interval between $\mu= \frac{1}{3}$ and $\mu=\frac{1}{5}$, the amplitude-modulated waveform can have, at most, two oscillations in the envelope, and so on. 

The PH model supports other types of toral folded singularities. For $V_{PLC}<0.129011~\mu$M, system \eqref{eq:PHmodel} has toral folded saddles. 
%and the ordinary singularity of the averaged radial-slow system has moved to the attracting manifold of limit cycles. 
The toral folded saddle has precisely one torus canard associated to it. This torus canard, however, has no rotational behaviour, and instead acts as a local phase space separatrix between trajectories that fall off the manifold of periodics at $\mathcal{P}_L$ and those that turn away from $\mathcal{P}_L$ and stay on $\mathcal{P}_a$. 
The other main type of toral folded singularity that can occur is the toral folded focus. In the PH model, we find a set of toral folded foci for $V_{PLC} > 0.38642~\mu$M. As stated in Section \ref{subsec:toralclass}, toral folded foci have no torus canard dynamics. 

The transition between TFN and toral folded saddle occurs at $V_{PLC} \approx 0.129011~\mu$M in a toral FSN of type II (corresponding to $\mu=0$), in which an ordinary singularity of the averaged radial-slow system coincides with the toral folded singularity. 
%We note that the PH model is such that whenever there is a toral folded singularity, there is also an ordinary singularity of the averaged radial-slow system. This ordinary singularity of the averaged radial-slow system is repelling when the toral folded singularity is of node or focus type, and is attracting when the toral folded singularity is of saddle type. 

%---------------------------------------------------------------------------------
\section{Torus Canard-Induced Bursting Rhythms}	\label{sec:TCIMMO}
%---------------------------------------------------------------------------------

Having carefully examined the local oscillatory behaviour of the PH model due to TFNs and their associated torus canards, we proceed in this section to identify the local and global dynamic mechanisms responsible for the AMB rhythms. 
In Section \ref{subsec:AMB}, we study the effects of parameter variations on the AMB solutions. 
We then show in Section \ref{subsec:PHsingTC} that the AMBs are torus canard-induced mixed-mode oscillations. 
In Section \ref{subsec:PHsingvsnonsing}, we examine where the spiking, bursting, and AMB rhythms exist in parameter space. 
In so doing, we demonstrate the origin of the AMB rhythm and show how it varies in parameters.

%----------------------------------------------------
\subsection{Amplitude-Modulated Bursting in the Politi-H\"{o}fer Model} 	\label{subsec:AMB}
%----------------------------------------------------

In Section \ref{subsec:PHbifn}, we reported on the existence of AMB solutions in the PH model (see Figure \ref{fig:PHbifn}(d)). The novel features of these AMBs are the oscillations in the envelope of the rapidly oscillating waveform during the active phase, which significantly extend the burst duration. 

Changes in the parameter $V_{PLC}$ have a measurable effect on the amplitude modulation in these AMB rhythms. Increasing $V_{PLC}$ causes a decrease in the number of oscillations that the profile of the waveform exhibits. That is, as $V_{PLC}$ increases, the envelope of the bursting waveform gradually loses oscillations and the burst duration decreases. This progressive loss of oscillations in the envelope continues until $V_{PLC}$ has been increased sufficiently that all of the small oscillations disappear. For instance, for $V_{PLC} = 0.1489~\mu$M, we observe 5 oscillations in the envelope (Figure \ref{fig:PHampmod}(a)). This decreases to 4 oscillations for $V_{PLC}=0.149~\mu$M (Figure \ref{fig:PHbifn}(d)), down to 3 for $V_{PLC}=0.1492~\mu$M (Figure \ref{fig:PHampmod}(b)), and then to 2 for $V_{PLC}=0.1495~\mu$M (Figure \ref{fig:PHampmod}(c)). Further increases in $V_{PLC}$ result in just 1 oscillation in the envelope (not shown) until, for sufficiently large $V_{PLC}$, the oscillations disappear. 
Once the amplitude modulation disappears (Figure \ref{fig:PHampmod}(d)), the trajectories resemble the elliptic bursting rhythms discussed previously. 

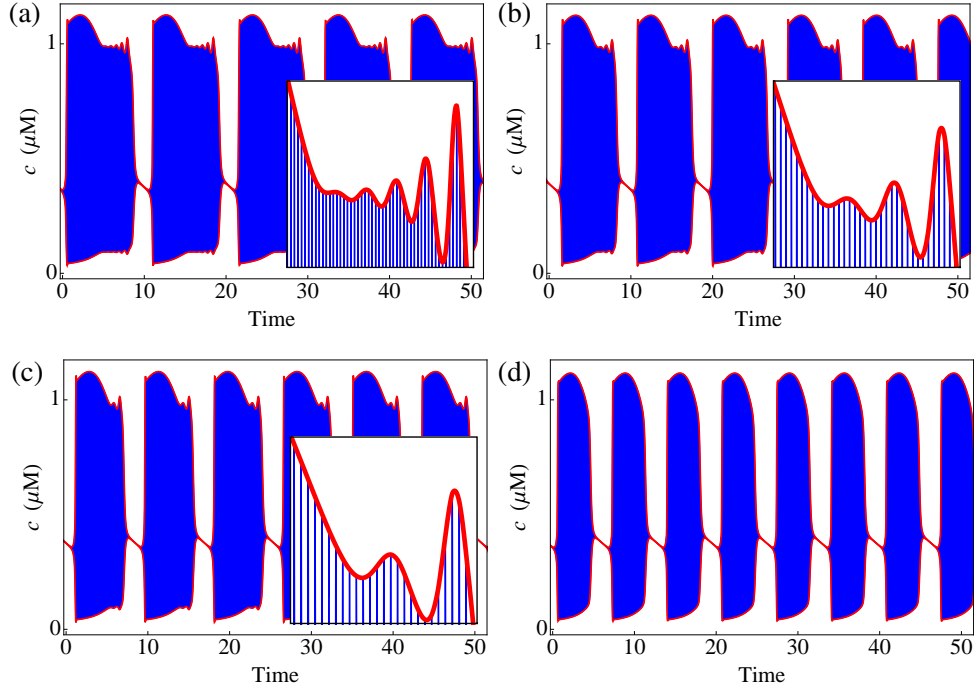
\begin{figure}[ht!]
\centering
\includegraphics[width=5in]{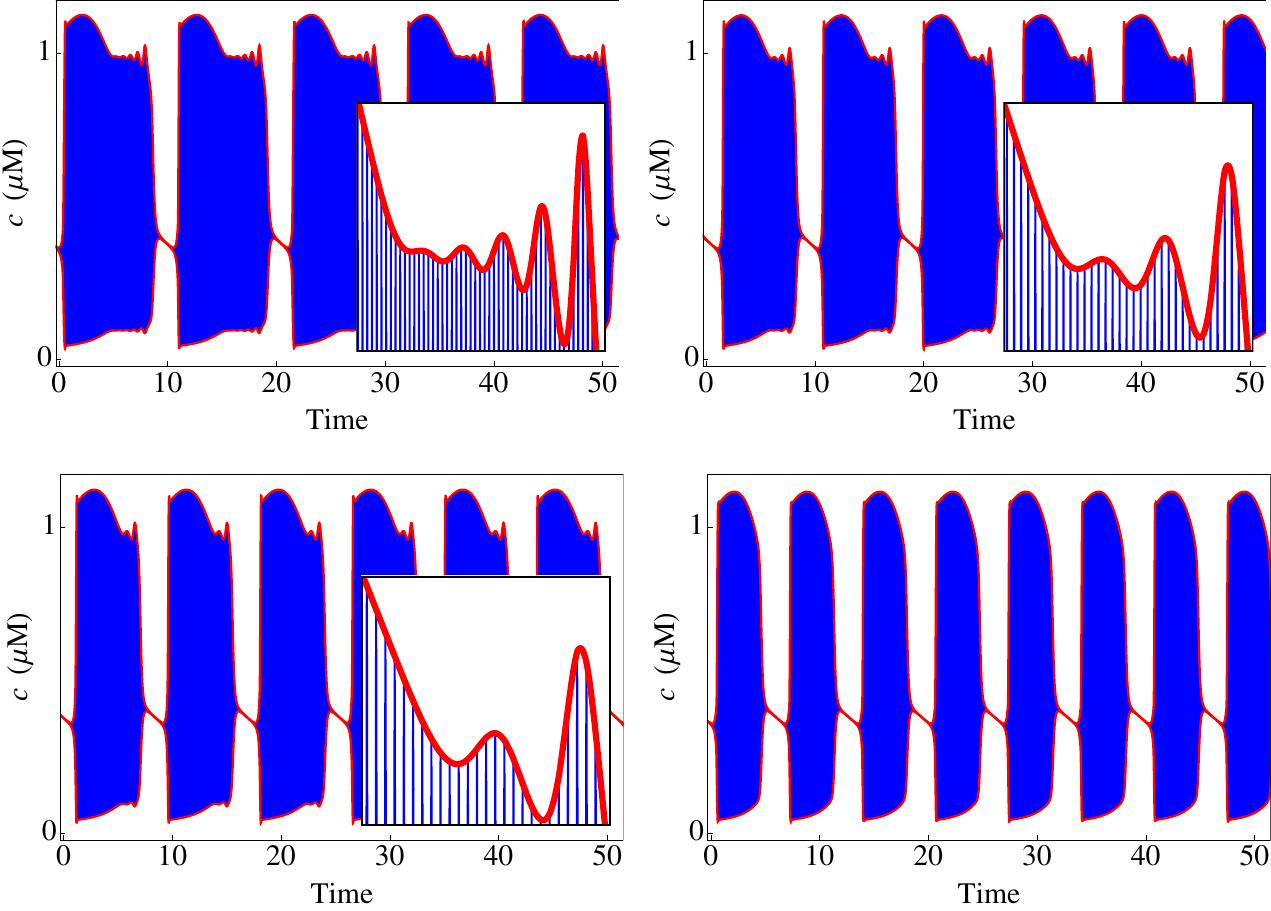}
\put(-365,252){(a)}
\put(-180,252){(b)}
\put(-364,116){(c)}
\put(-180,116){(d)}
\caption{Amplitude-modulated bursting rhythms of system \eqref{eq:PHmodel} for (a) $V_{PLC}=0.1489~\mu$M, (b) $V_{PLC}=0.1492~\mu$M, (c) $V_{PLC}=0.1495~\mu$M, and (d) $V_{PLC}=0.1510~\mu$M. Increasing $V_{PLC}$ decreases the number of oscillations that the envelope of the waveform executes, and consequently decreases the burst duration.}
\label{fig:PHampmod}
\end{figure}

It is currently unknown what kinds of bifurcations, if any, occur in the transitions between AMB waveforms with different numbers of oscillations in the envelope. We conjecture that these transitions occur via torus doubling bifurcations (since they are associated with TFNs; see Section \ref{subsec:PHsingTC}). Further investigation of the bifurcations that organise these transitions is beyond the scope of the current article.

%----------------------------------------------------
\subsection{Origin of the Amplitude-Modulated Bursting} 	\label{subsec:PHsingTC}
%----------------------------------------------------

We first concentrate on understanding the mechanisms that generate the amplitude-modualted bursting rhythm seen in Figure \ref{fig:PHbifn}(d), corresponding to $V_{PLC} = 0.149~\mu$M. To do this, we construct the singular attractor of the PH model for $V_{PLC} = 0.149~\mu$M (Figures \ref{fig:PHtoralmmo}(a) and (b)). The singular attractor is the concatenation of four orbit segments. Starting in the silent phase of the burst, there is a slow drift (black, single arrow) along the critical manifold $\mathcal{S}_a$ that takes the orbit up to the curve $\mathcal{H}$ of Hopf bifurcations, where the stability of $\mathcal{S}$ changes. 
This initiates a fast upward transition (black, double arrows) away from $\mathcal{H}$ towards the attracting manifold of limit cycles, $\mathcal{P}_a$. 
Once the trajectory reaches $\mathcal{P}_a$, there is a net slow drift (black, single arrow) that moves the orbit segment along $\mathcal{P}_a$ towards $\mathcal{P}_L$. This net slow drift along $\mathcal{P}_a$ can be described by an appropriate averaged system (Theorem \ref{thm:averaging}). 
We find that for $V_{PLC}=0.149~\mu$M, the fast up-jump from $\mathcal{H}$ to $\mathcal{P}_a$ projects the trajectory into the funnel of the TFN. As such, the slow drift brings the trajectory to the TFN itself (green marker). At the TFN, there is a fast downward transition (black, double arrows) that projects the trajectory down to the attracting sheet of the critical manifold, thus completing one cycle. 

\begin{figure}[h!]
\centering
\includegraphics[width=5in]{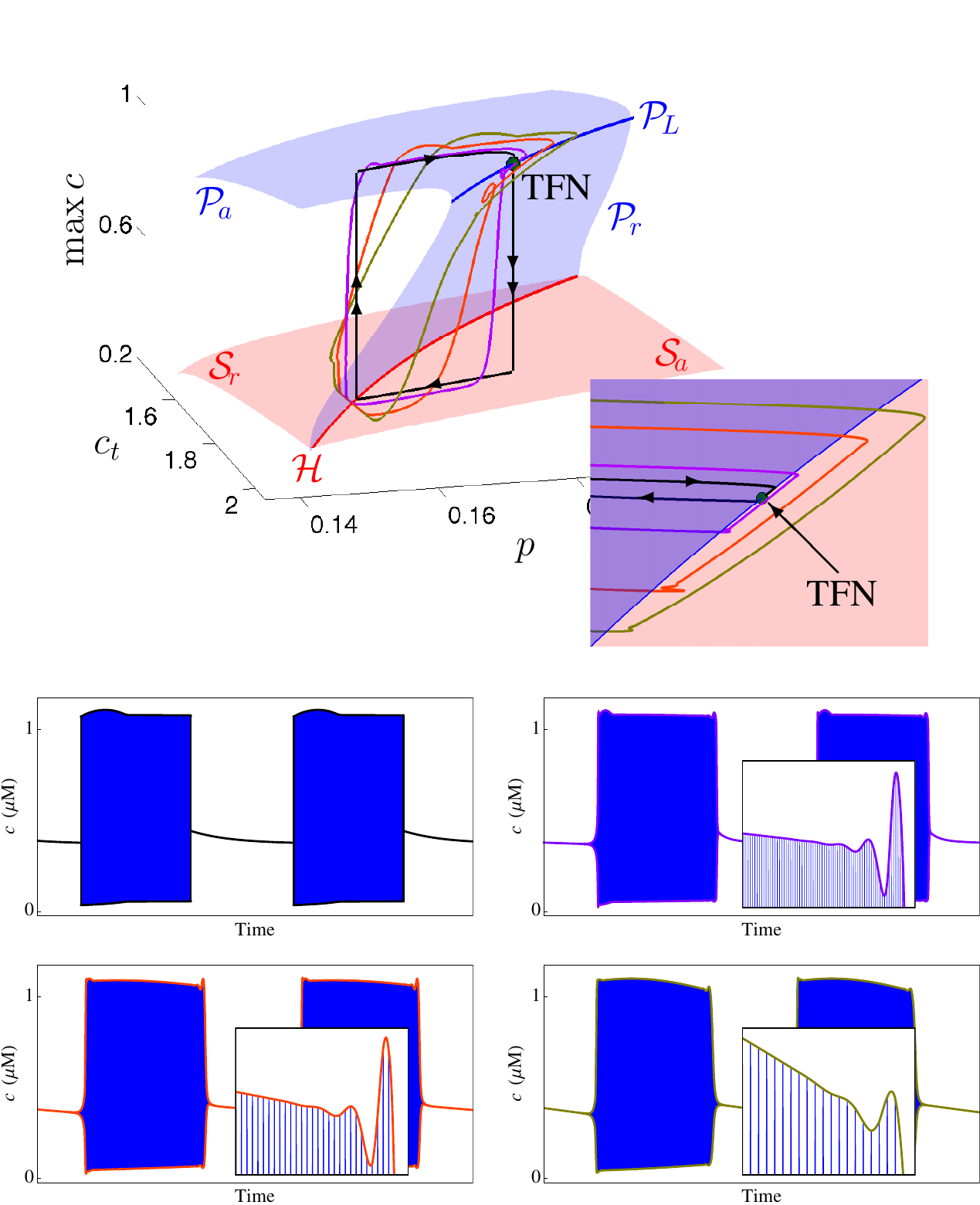}
\put(-340,406){(a)}
\put(-364,180){(b)}
\put(-180,180){(c)}
\put(-364,81){(d)}
\put(-180,81){(e)}
\caption{Torus canard-induced mixed-mode oscillations in the PH model for $V_{PLC}=0.149~\mu$M and $\eps=0$ (black), $\eps=0.0001$ (purple), $\eps=0.0005$ (orange), and $\eps=0.001$ (olive). (a) Trajectories superimposed on the critical manifold, $\mathcal{S}_a \cup \mathcal{H} \cup \mathcal{S}_r$, and manifold of limit cycles, $\mathcal{P}_a \cup \mathcal{P}_L \cup \mathcal{P}_r$. The singular attractor alternates between slow epochs (single arrows) on $\mathcal{S}_a$ and $\mathcal{P}_a$, with fast jumps (double arrows) between them. The $\eps$-unfoldings of the singular attractor (coloured trajectories) spend long times near the TFN (green marker). Inset: projection into the slow variable plane in an $\mathcal{O}(\sqrt{\eps})$ neighbourhood of the TFN. The associated time traces of \eqref{eq:PHmodel} are shown in (b) for $\eps=0$, (c) for $\eps=0.0001$, (d) for $\eps=0.0005$, and (e) for $\eps=0.001$. The coloured envelopes in (b)--(e) correspond to the coloured trajectories in (a).}
\label{fig:PHtoralmmo}
\end{figure}
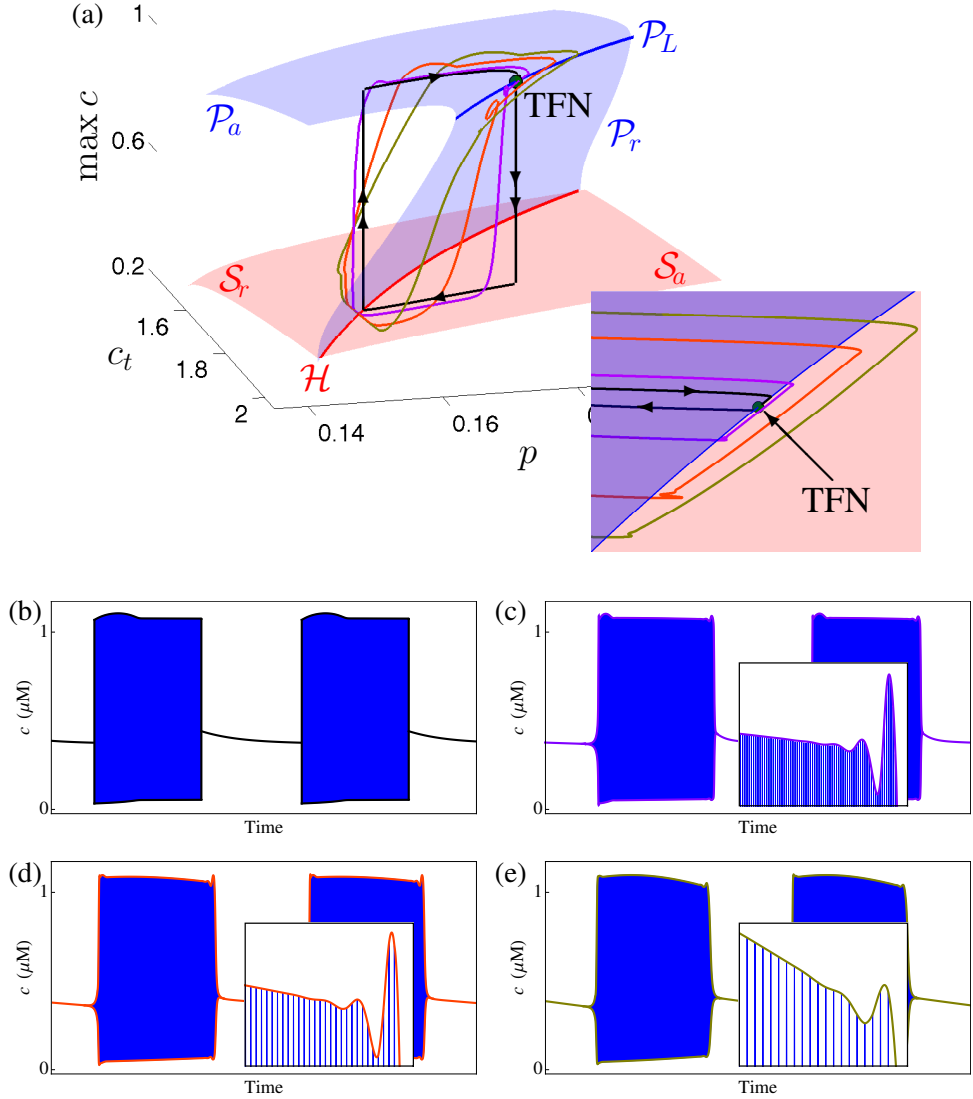

Figure \ref{fig:PHtoralmmo} shows that the singular attractor perturbs to the AMB rhythm for sufficiently small $\eps$ (purple, orange, and olive trajectories). 
That is, for small non-zero perturbations, the silent phase of the orbit is a small $\mathcal{O}(\eps)$-perturbation of the slow drift on the critical manifold. Note that the trajectory does not immediately leave the silent phase when it reaches the Hopf curve. Dynamic bifurcation theory shows that the initial exponential contraction along $\mathcal{S}_a$ allows trajectories to follow the repelling slow manifold for $\mathcal{O}(1)$ times on the slow time-scale \cite{Neishtadt1987,Neishtadt1988}. However, there eventually comes a moment where the repulsion on $\mathcal{S}_r$ overwhelms the accumulative contraction on $\mathcal{S}_a$ and the trajectory jumps away to the invariant manifold of limit cycles $\mathcal{P}_a^{\eps}$. We have established that for the segments of $\mathcal{P}_a$ that are an $\mathcal{O}(1)$-distance from $\mathcal{P}_L$, the slow drift along $\mathcal{P}_a^{\eps}$ is a smooth $\mathcal{O}(\eps)$ perturbation of the averaged slow flow along $\mathcal{P}_a$.  

In a neighbourhood of the manifold of SNPOs, and the TFN in particular, we have shown in Section \ref{subsec:PHmanifolds} that torus canards are the local phase space mechanisms responsible for oscillations in the envelope of the rapidly oscillating waveform. These oscillations are restricted to an $\mathcal{O}(\sqrt{\eps})$-neighbourhood of the TFN (Figure \ref{fig:PHtoralmmo}(a), inset). Note that as $\eps$ increases the position of the AMB trajectory changes relative to the maximal torus canards. 
For instance, the purple and orange trajectories in Figure \ref{fig:PHtoralmmo} are closer to one of the maximal torus canards than the olive trajectory. In fact, the olive solution lies in a different rotational sector than the purple and orange solutions and hence has fewer oscillations.

Thus, the AMB consists of a local mechanism (torus canard dynamics due to the TFN) and a global mechanism (the slow passage of the trajectory through a delayed Hopf bifurcation, which re-injects the orbit into the funnel of the TFN). Consequently, the AMB can be regarded as a \emph{torus canard-induced mixed-mode oscillation}.

\begin{remark}
The PH model supports canard-induced mixed-mode dynamics \cite{Brons2006}, since it has a cubic-shaped critical manifold with folded singularities. In fact, careful analyses of the canard-induced mixed-mode oscillations in \eqref{eq:PHmodel} were performed in \cite{Harvey2011}. The difference between our work and \cite{Harvey2011} is that we are concentrating on the AMB behaviour near the torus bifurcation TR$_3$ (see Figure \ref{fig:PHbifn}(e)), whereas \cite{Harvey2011} focuses on the mixed-mode oscillations near HB$_1$.
\end{remark}

We have demonstrated the origin of the AMB rhythm for the specific parameter value $V_{PLC} = 0.149~\mu$M. The other AMB rhythms observed in the PH model (such as in Figure \ref{fig:PHampmod}) can also be shown to be torus canard-induced mixed-mode oscillations. The number of oscillations that the envelopes of the AMBs exhibit is determined by two key diagnostics: the eigenvalue ratio of the TFN which determines how many maximal torus canards exist, and the global return mechanism (slow passage through the delayed Hopf) which determines how many oscillations are actually observed.

Whilst we have carefully studied the local mechanism in Sections \ref{subsec:PHtoralFS}--\ref{subsec:PHtoralFSclass} and identified the global return mechanism, we have not performed any careful analysis of the global return or its dependence on parameters. In particular, the boundary $d=0$, corresponding to the special scenario in which the singular trajectory is re-injected exactly on the singular strong torus canard marks the boundary between those trajectories that reach the TFN (and exhibit torus canard dynamics) and those that simply reach the manifold of SNPOs and fall off without any oscillations in the envelope. Furthermore, the global return is able to generate or lose oscillations in the envelope by re-injecting orbits into the different rotational sectors formed by the maximal torus canards. We leave the investigation of the global return for these AMBs to future work. 

\begin{remark}
The bursting and spiking rhythms shown in Figures \ref{fig:PHbifn}(b) and (c), respectively, can also be understood in terms of the bifurcation structure of the layer problem \eqref{eq:PHlayer}. In the bursting case, the trajectory can be decomposed into four distinct segments, analogous to the AMB rhythm. The only difference is that the bursting orbit encounters $\mathcal{P}_L$ at a regular folded limit cycle instead of a TFN, and so it simply falls off $\mathcal{P}$ without exhibiting torus canard dynamics.  Such a bursting solution, with active phase initiated by slow passage through a fast subsystem subcritical Hopf bifurcation, and active phase terminated at an SNPO, is known as a subcritical elliptic burster \cite{Izhikevich2000}.
Note that in the subcritical elliptic bursting (Figure \ref{fig:PHbifn}(b)) and AMB (Figure \ref{fig:PHbifn}(d)) cases, the averaged radial-slow flow possesses a repelling ordinary singularity (i.e., unstable spiking solutions). 

In the tonic spiking case, the trajectory of \eqref{eq:PHmodel} can be understood by locating ordinary singularities of the averaged radial-slow system \eqref{eq:Ruslowfast}. We find that the averaged radial-slow system of the PH model has an attracting ordinary singularity, 
which corresponds to a stable limit cycle $\Gamma$ of the layer problem \eqref{eq:PHlayer}. Since $\Gamma$ is hyperbolic, the full PH model exhibits periodic solutions which are $\mathcal{O}(\eps)$ perturbations of the normally hyperbolic limit cycle $\Gamma$ for sufficiently small $\eps$. In this spiking regime, the system possesses a toral folded saddle (corresponding to the region of negative $\mu$ in Figure \ref{fig:PHeigenvalues}(a)). 
\end{remark}

%----------------------------------------------------
\subsection{Toral FSN II and the Amplitude-Modulated Spiking/Bursting Boundary}	 	\label{subsec:PHsingvsnonsing}
%----------------------------------------------------

We are interested in the transitions between the different dynamic regimes (spiking, bursting, and AMB) of \eqref{eq:PHmodel}. Figure \ref{fig:PHTR}(a) shows the two-parameter bifurcation structure of \eqref{eq:PHmodel} in the $(V_{PLC},\eps)$ plane. Continuation of the torus bifurcations TR$_3$ and TR$_4$ (from Figure \ref{fig:PHbifn}(e)) generates a single curve, which separates the spiking and bursting regimes. The region enclosed by the TR$_3$/TR$_4$ curve consists of subcritical elliptic bursting solutions (including the AMB). By similarly continuing the Hopf bifurcation HB$_2$, we find that the spiking regime is the region bounded by the HB$_2$ curve and the curve of torus bifurcations. Note that the branch TR$_3$, which separates the rapid spiking and AMB waveforms, converges to the toral FSN II at $V_{PLC} \approx 0.129011~\mu$M in the singular limit $\eps \to 0$. This supports our conjecture from Section \ref{subsec:toralFSN} that the $\eps$-unfolding of the toral FSN II is a singular torus bifurcation. 

\begin{figure}[ht!]
\centering
\includegraphics[width=5in]{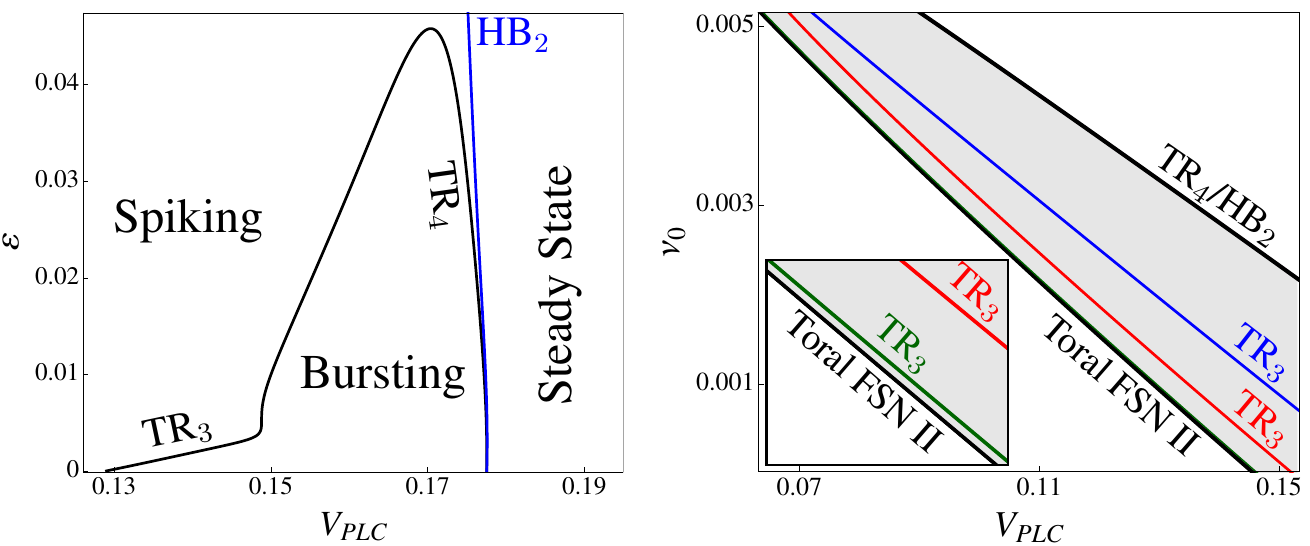}
\put(-364,140){(a)}
\put(-184,140){(b)}
\caption{Two-parameter bifurcation structure of the PH model obtained by following the torus bifurcations, TR$_3$ and TR$_4$, and the Hopf bifurcation, HB$_2$, from Figure \ref{fig:PHbifn}(e) in (a) the $(V_{PLC},\eps)$ plane, and (b) the $(V_{PLC},\nu_0)$ plane. (a) The torus bifurcation curve encloses the bursting region. In the limit as $\eps \to 0$, TR$_3$ converges to the toral FSN II at $V_{PLC}\approx 0.129011~\mu$M. The region between the Hopf curve and the torus curve is the spiking region. (b) In the singular limit $\eps \to 0$, the bursting region is enclosed by the toral FSN II and TR$_4$ curves. 
The TR$_3$ curve unfolds (in $\eps$) from the toral FSN II curve. Thus, for $0 < \eps \ll 1$, the AMB solutions exhibit more oscillations in their envelopes the closer the parameters are chosen to the TR$_3$ boundary.}
\label{fig:PHTR}
\end{figure}

We provide further numerical evidence to support this conjecture in Figure \ref{fig:PHTR}(b), where we compare the loci of the toral FSNs of type II and torus bifurcation TR$_3$ in the $(V_{PLC},\nu_0)$ parameter plane for various $\eps$. 
The coloured curves correspond to the torus bifurcation TR$_3$ for $\eps = 0.0035$ (blue), $\eps = 0.001$ (red), and $\eps = 0.0001$ (green; inset). 
As demonstrated in Figure \ref{fig:PHTR}(b), the TR$_3$ curve converges to the toral FSN II curve in the singular limit. 
Also shown are the TR$_4$ and HB$_2$ curves, which remain close to each other in the $(V_{PLC},\nu_0)$ plane and enclose a very thin wedge of spiking solutions in the parameter space. 

Thus, in the singular limit, the bursting region is enclosed by the toral FSN II and TR$_4$ curves (Figure \ref{fig:PHTR}(b); shaded region). 
Moreover, the curve TR$_3$ of torus bifurcations that unfolds from the toral FSN II curve forms the boundary between AMB and amplitude-modulated spiking rhythms. 
That is, the closer the parameters are to the TR$_3$ curve, the more oscillations in the envelope of the AMB trajectories and hence the longer the burst duration. Similarly, the spiking solutions that exist near the toral FSN II curve exhibit amplitude modulation. 

\begin{remark}
The numerical computation and continuation of the curve of toral FSNs of type II requires careful numerics; it requires solutions of a periodic boundary value problem subject to phase and integral conditions. We outline the procedure in Appendix \ref{app:TFScurve}. 
\end{remark}

We saw from Figure \ref{fig:PHeigenvalues}(a) that the PH model supports TFN-type torus canards for $0.129011~\mu$M $< V_{PLC} < 0.38642~\mu$M. Figure \ref{fig:PHTR}, however, shows that the amplitude modulated bursting exists on a more restricted interval of $V_{PLC}$. This indicates the importance of the global return mechanism in shaping the outcome of the torus canard-induced mixed-mode dynamics. The bifurcations that separate the different amplitude-modulated waveforms are currently unknown and left to future work.

%---------------------------------------------------------------------------------
\section{Torus Canards In $\mathbb{R}^3$ \& The Spiking/Bursting Transition}		\label{sec:explosion}
%---------------------------------------------------------------------------------

Having established the predictive power of our analysis of torus canards in $\mathbb{R}^4$, we now examine the connection between our analysis and prior work on torus canards in $\mathbb{R}^3$, namely in 2-fast/1-slow systems. 
In Section \ref{subsec:TCexpmodels}, we carefully study the transition from tonic spiking to bursting via amplitude-modulated spiking in the Morris-Lecar-Terman system for neural bursting. We show that the boundary between spiking and bursting is given by the toral folded singularity of the system. 
We further demonstrate the power of our theoretical framework by tracking the toral folded singularities in the Hindmarsh-Rose (Section \ref{subsec:TCHR}) and Wilson-Cowan-Izhikevich (Section \ref{subsec:TCWCI}) models. 
We note that the analysis in this section relies on the results in Section \ref{sec:arbitrarydimensions}, namely Theorem \ref{thm:averagingarbitraryslow}, which extends the averaging method for folded manifolds of limit cycles to slow/fast systems with two fast variables and an arbitrary number of slow variables. 

\begin{remark}
Our theoretical framework also allows one to determine the parameter values for which a torus canard explosion occurs in 2-fast/1-slow systems (Theorems \ref{thm:3Daveraging} and \ref{cor:explosion}). This predictive power is illustrated in Appendix \ref{app:TCexplosion} in the case of the forced van der Pol equation. 
\end{remark}

%----------------------------------------------------
\subsection{Torus Canards in the Morris-Lecar-Terman Model} 	\label{subsec:TCexpmodels}
%----------------------------------------------------

The Morris-Lecar-Terman (MLT) model \cite{Burke2012,Terman1991} is an extension of the planar Morris-Lecar model for neural excitability in which the constant applied current is replaced with a linear feedback control, $y$. The (dimensionless) model equations are
\begin{equation}	\label{eq:MLT}
\begin{split}
\dot{v} &= y- g_L (V-V_L) - g_K w (V-V_K) -g_{Ca} m_{\infty} (V) (V-V_{Ca}), \\
\dot{w} &= \frac{w_{\infty}(v)-w}{\tau_w (v)}, \\
\dot{y} &= \eps \left( k-V \right),
\end{split}
\end{equation}
where $v$ is the (dimensionless) voltage, $w$ is the recovery variable, and $y$ is the (dimensionless) applied current. The steady-state activation functions are given by 
\[ m_{\infty}(v) = \frac{1}{2} \left( 1+\tanh\left( \frac{v-c_1}{c_2} \right) \right), \quad w_{\infty}(v) = \frac{1}{2} \left( 1+\tanh\left( \frac{v-c_3}{c_4} \right) \right), \]
and the voltage-dependent time-scale, $\tau_w$, of the recovery variable $w$ is 
\[ \tau_w(v) = \tau_0 \operatorname{sech} \left( \frac{v-c_3}{2c_4} \right). \]
Following \cite{Burke2012}, we treat $k$ and $g_{Ca}$ as the principal control parameters, and fix all other parameters at the standard values listed in Table \ref{tab:MLT}.

\begin{table}[h!]
\centering
\begin{tabular}{|c|c|c|c|c|c|c|c|c|c|}
\hline
Param. & Value & Param. & Value & Param. & Value & Param. & Value & Param. & Value \\
\hline
$g_L$ & $0.5$ & $g_K$ & $2$ & $V_L$ & $-0.5$ & $V_K$ & $-0.7$ & $V_{Ca}$ & $1.0$  \\
$c_1$ & $-0.01$ & $c_2$ & $0.15$ & $c_3$ & $0.1$ & $c_4$ & $0.16$ & $\tau_0$ & $3$ \\
$\eps$ & $0.001$ & \textcolor{white}{1} & \textcolor{white}{1} & \textcolor{white}{1} & \textcolor{white}{1} & \textcolor{white}{1} & \textcolor{white}{1} & \textcolor{white}{1} & \textcolor{white}{1}\\
\hline
\end{tabular}
\caption{Standard parameter set for the MLT model \eqref{eq:MLT}.}
\label{tab:MLT}
\end{table}

System \eqref{eq:MLT} is a slow/fast system, with fast variables $(v,w)$ and slow variable $y$. The MLT model has been shown to exhibit a wide variety of bursting dynamics, such as fold/homoclinic bursting \cite{Terman1991} and circle/fold-cycle bursting \cite{Izhikevich2000}. Moreover, the transition from spiking to (circle/fold-cycle) bursting has been shown to be mediated by a torus canard explosion \cite{Burke2012}. Here, we take a novel approach to the study of \eqref{eq:MLT} as follows. First, we examine the bifurcation structures of both the MLT model and its associated layer problem. We review how the spiking and bursting rhythms can be understood in terms of the underlying geometry. We then identify toral FSNs of type II and compute their loci in parameter space. We demonstrate that the $\eps$ unfolding of the toral FSN II is the torus bifurcation that mediates the spiking/bursting transition for $0< \eps \ll 1$.
%the torus bifurcation of \eqref{eq:MLT}, which mediates the spiking-bursting transition for $0<\eps \ll 1$, converges to the toral FSN II in the limit as $\eps \to 0$.

%----------------------------------------------------
\subsubsection{Bifurcation Structure} 	\label{subsubsec:MLTbifn}
%----------------------------------------------------

System \eqref{eq:MLT} can generate tonic spiking (Figure \ref{fig:mltfullsystem}, top row) or bursting (Figure \ref{fig:mltfullsystem}, middle row), depending on parameters. A useful way to describe and understand these rhythms is via the bifurcation structure of the layer problem of \eqref{eq:MLT} with respect to the slow variable $y$ \cite{Rinzel1998,Rubin2000}. The layer problem has a cubic shaped critical manifold, $\mathcal{S}$. The upper attracting branch of $\mathcal{S}$ is joined to the middle repelling branch via a subcritical Hopf bifurcation. Emanating from the subcritical Hopf is a branch of repelling limit cycles, $\mathcal{P}_r$, which meets an attracting family of limit cycles, $\mathcal{P}_a$, at an SNPO, $\mathcal{P}_L$. The attracting branch $\mathcal{P}_a$ terminates at a saddle-node on invariant circle (SNIC) bifurcation.

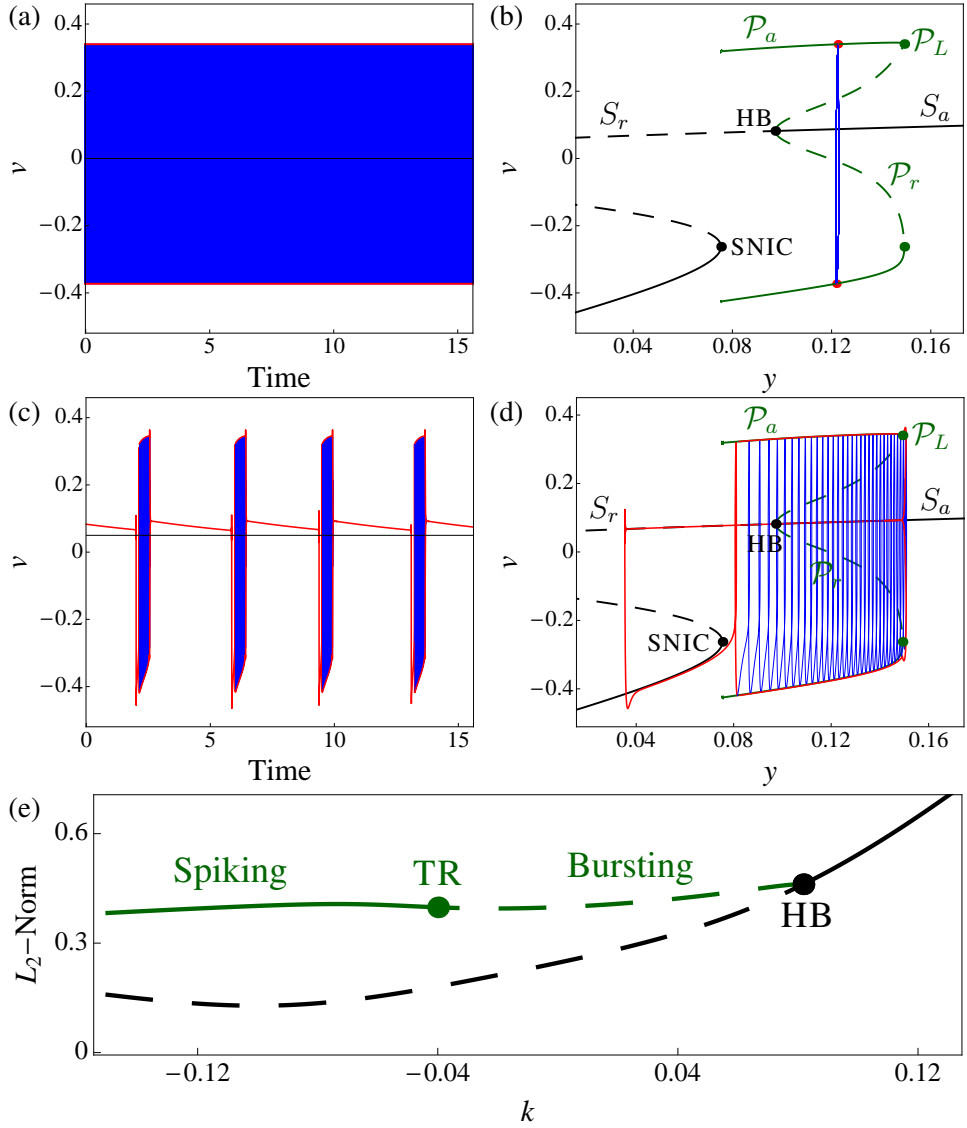
\begin{figure}[h!]
\centering
\includegraphics[width=5in]{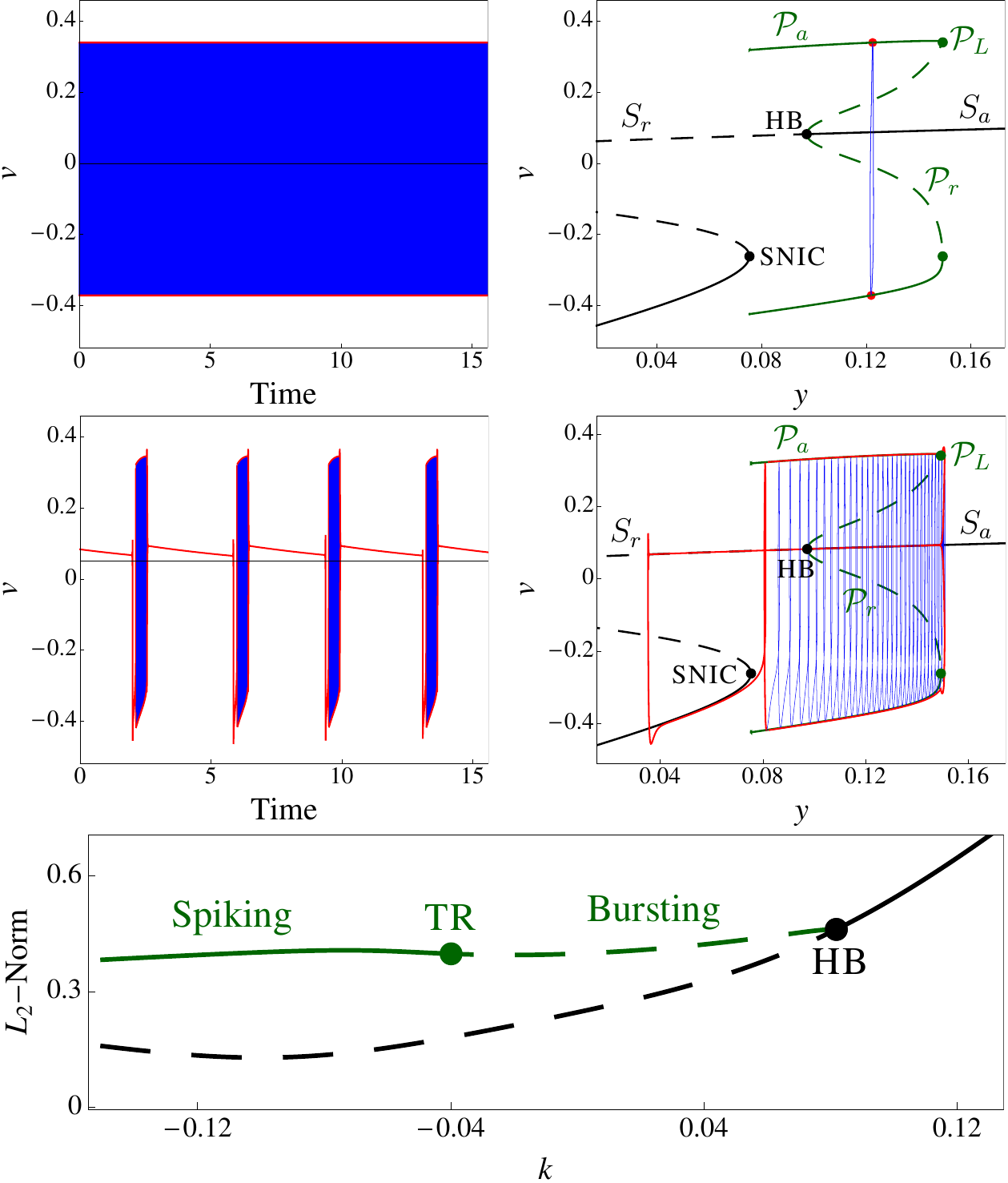}
\put(-362,416){(a)}
\put(-180,416){(b)}
\put(-362,266){(c)}
\put(-180,266){(d)}
\put(-362,116){(e)}
\caption{Spiking and bursting rhythms of the MLT model \eqref{eq:MLT} for $\eps=0.001, g_{Ca}=1.25$, and $k=-0.12$ (top row) and $k=0.04$ (middle row). The left column shows the time traces and the right column shows the trajectory superimposed on the bifurcation structure of the layer problem. In the spiking case, the full system spiking attractor stays in an $\mathcal{O}(\eps)$-neighbourhood of $\mathcal{P}_a$. The bursting attractor consists of alternating segments of slow drift along $\mathcal{P}_a$ and $\mathcal{S}$, with fast jumps between. The bifurcation structure of \eqref{eq:MLT} with respect to $k$ is summarized in (e). The spiking to bursting transition is mediated by a torus bifurcation (TR).}
\label{fig:mltfullsystem}
\end{figure}

Figure \ref{fig:mltfullsystem}(b) shows the spiking attractor superimposed on the bifurcation structure of the layer problem of \eqref{eq:MLT} with respect to $y$. The spiking rhythm corresponds to a stable equilibrium of the averaged radial-slow system. Similarly, Figure \ref{fig:mltfullsystem}(d) shows the bursting rhythm projected into the $(y,v)$ plane. In this case, the active burst phase is the result of the slow drift of the orbit along $\mathcal{P}_a$, until it reaches $\mathcal{P}_L$, where it falls off and returns to $\mathcal{S}_a$. The slow drift along the critical manifold then allows the orbit to drift past the Hopf bifurcation until the instability of $\mathcal{S}$ repels it. The trajectory then jumps to the lower attracting branch of $\mathcal{S}$, where the slow drift brings it to a neighbourhood of the SNIC, where it is repulsed and begins following $\mathcal{P}_a$ once again. 

The bifurcation structure of the full MLT model is summarized in Figure \ref{fig:mltfullsystem}(e). For sufficiently large $k$, the system is in a stable depolarized state. The quiescent state becomes unstable in a supercritical Hopf bifurcation. The stable limit cycles rapidly destabilize in a torus bifurcation (not shown), resulting in bursting solutions. As $k$ is decreased, another torus bifurcation is encountered and the bursting solutions become spiking trajectories. Note that at an $\mathcal{O}(\eps)$ distance from the torus bifurcation (on the spiking side), the MLT system is in an amplitude-modulated spiking state (not shown).

%----------------------------------------------------
\subsubsection{Toral Folded Singularities} 	\label{subsubsec:MLTtoralFS}
%----------------------------------------------------

We will show in Section \ref{sec:arbitrarydimensions} (Theorem \ref{thm:averagingarbitraryslow}) that a toral folded singularity in the case of one slow variable is a folded limit cycle, $\Gamma=(v_{\Gamma},w_{\Gamma})$, of the layer problem of \eqref{eq:MLT} such that 
\[ \overline{g} := \frac{1}{T(y)} \int_0^{T(y)} g(v_{\Gamma},w_{\Gamma},y)\, dt =0, \]
where $g(v,w,y)$ is the slow component of the vector field, and $T(y)$ is the period of $\Gamma$. For system \eqref{eq:MLT}, we have $g(v,w,y)=k-v$, and the condition of stationary average slow drift simplifies to
\[ \overline{v} := \frac{1}{T(y)} \int_0^{T(y)} v_{\Gamma}\, dt = k. \]
Thus, we interpret the toral folded singularity in this case as the intersection of the averaged slow nullcline $\{ \overline{g}=0 \}$ and the SNPO in the $(y,\overline{v})$ phase plane. Figure \ref{fig:mlttransition} shows the progression as the averaged slow nullcline increases through the SNPO. 

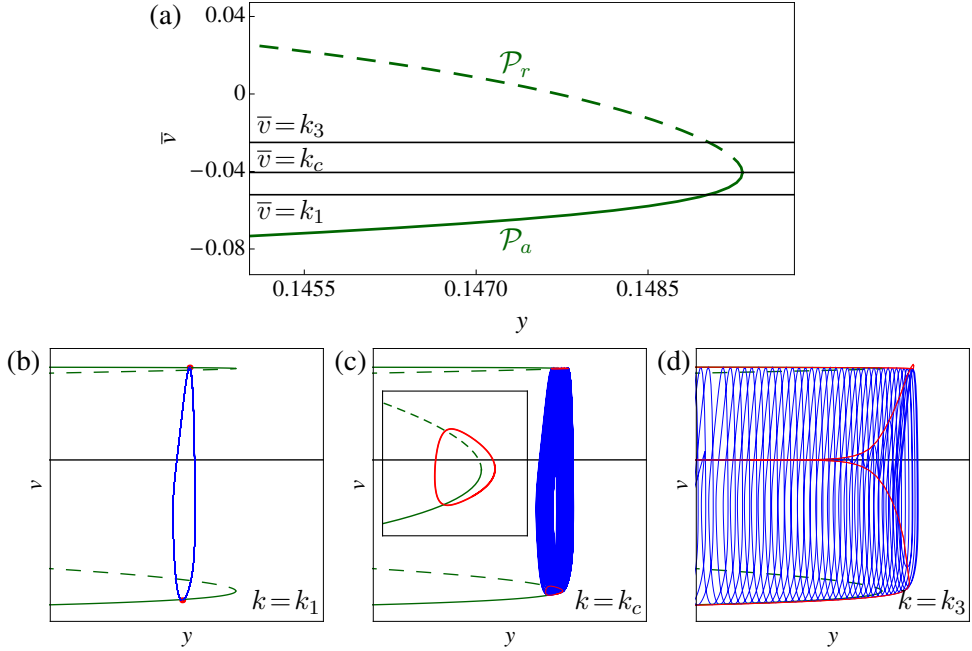
\begin{figure}[h!]
\centering
\includegraphics[width=5in]{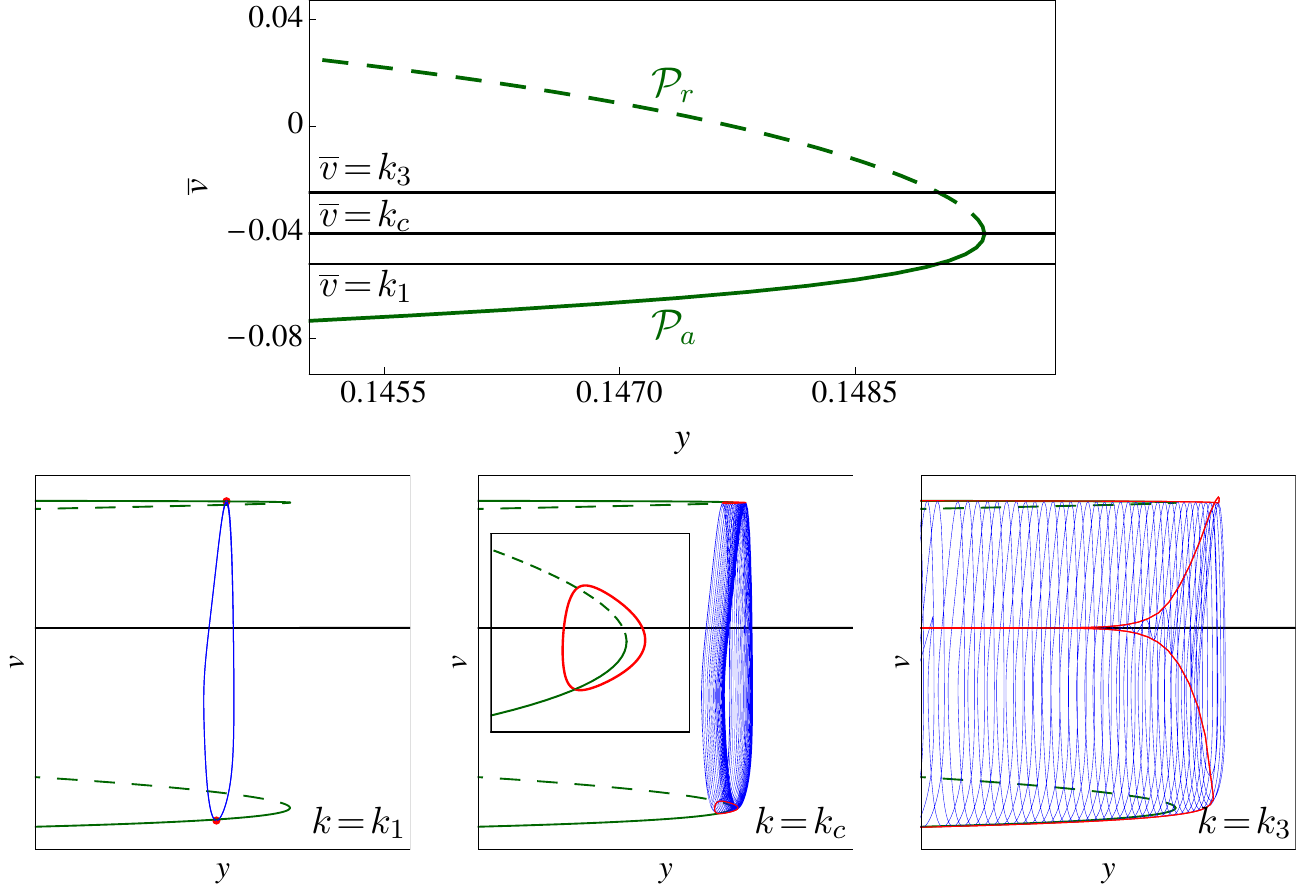}
\put(-314,236){(a)}
\put(-368,105){(b)}
\put(-244,105){(c)}
\put(-122,105){(d)}
\caption{Transition from spiking to bursting for $g_{Ca}=1.25$.
(a) Projection of $\mathcal{P}$ and the averaged slow nullcline $\{ \overline{v} = k \}$ into the $(y,\overline{v})$ plane for three different values of $k$. (b) For $k<k_c$, the averaged slow nullcline intersects $\mathcal{P}_a$, and the corresponding full system trajectory is a spiking rhythm. (c) For $k=k_c$, the averaged slow nullcline intersects the SNPO and the averaged radial-slow system has a canard point. The corresponding full system trajectory is an amplitude-modulated spiking rhythm. Inset: canard cycle of the envelope. (d) For $k>k_c$, the averaged slow nullcline intersects $\mathcal{P}_r$ and the full system trajectory is a (circle/fold-cycle) burst.}
\label{fig:mlttransition}
\end{figure}

For $k< k_c \approx -0.0405$, the averaged slow nullcline intersects $\mathcal{P}_a$. Theorem \ref{thm:averagingarbitraryslow} shows that the intersection $\{ \overline{g}=0 \} \cap \mathcal{P}_a$ is a normally hyperbolic ordinary singularity of the averaged radial-slow system. Thus, the ordinary singularity persists as an equilibrium of the fully perturbed averaged radial-slow flow. This corresponds to a stable spiking solution of the MLT model (Figure \ref{fig:mlttransition}(b)). 

When $k=k_c$, the averaged slow nullcline intersects the SNPO. In that case, the averaged radial-slow system possesses a canard point. Note that this is a toral FSN II since it is the intersection of a toral folded singularity and an ordinary singularity of the averaged radial-slow flow. 
As such, for $k$ values close to $k_c$, the trajectories of the averaged radial-slow system are canard cycles (Figure \ref{fig:mlttransition}(c), inset). Since the radial envelope exhibits canard cycles, the trajectories of the MLT system are torus canards, which manifest as amplitude-modulated spiking waveforms (Figure \ref{fig:mlttransition}(c)). 

As $k$ is further increased, the canard cycles of the averaged radial-slow system rapidly grow in amplitude and system \eqref{eq:MLT} eventually gives way to circle/fold-cycle bursting solutions (Figure \ref{fig:mlttransition}(d)). In that case, the intersection point $\{ \overline{g}=0 \} \cap \mathcal{P}$ lies on the repelling branch $\mathcal{P}_r$. Again, averaging theory and Fenichel theory guarantee that this persists as an unstable spiking trajectory of the full MLT system. 

Using this geometric intuition, we see that the toral folded singularity is the boundary where the spiking trajectories switch to bursting trajectories. In Figure \ref{fig:mlttwoparam}, we compute the locus of this toral folded singularity (TFS curve) in the $(k,g_{Ca})$ plane. 
The TFS curve locally partitions the parameter plane. The region below the TFS curve has the geometric configuration in which the ordinary singularity of the averaged radial-slow system lies on $\mathcal{P}_a$ (corresponding to spiking solutions). The region of the $(k,g_{Ca})$ plane above the TFS curve has the configuration in which the ordinary singularity of the averaged radial-slow system lies on $\mathcal{P}_r$ (corresponding to circle/fold-cycle bursting solutions).

\begin{figure}[h!]
\centering
\includegraphics[width=3.5in]{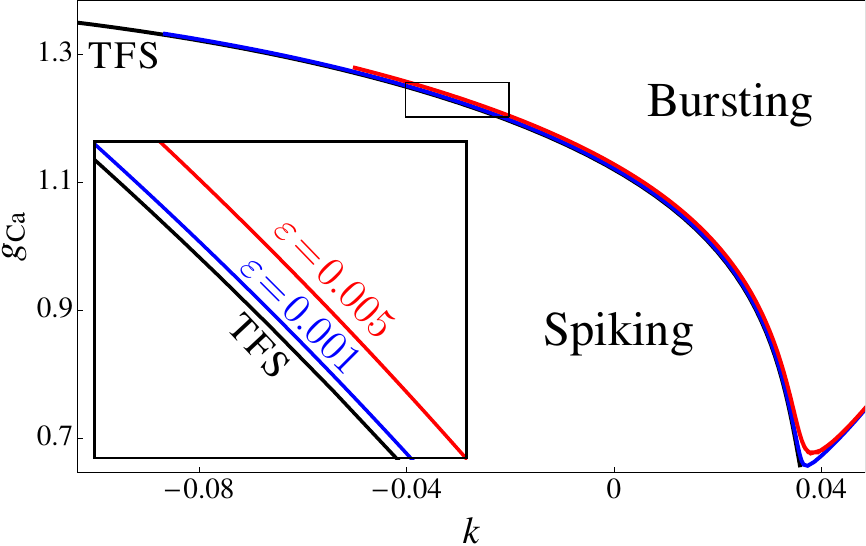}
\caption{Locus of the toral folded singularity (TFS; black curve) of the MLT model in the $(k,g_{Ca})$ plane. The TFS curve splits the parameter space between spiking and bursting behaviour, based on the geometry of the averaged radial-slow system. 
The blue and red curves are the two-parameter continuations of the torus bifurcation TR (see Figure \ref{fig:mltfullsystem}(e)) in the $(k,g_{Ca})$ plane. The curve of torus bifurcations converges to the TFS curve as $\eps \to 0$.}
\label{fig:mlttwoparam}
\end{figure}

Recall from Section \ref{subsec:toralFSN} that we conjectured that the unfolding of a toral FSN II would be a singular torus bifurcation. Figure \ref{fig:mlttwoparam} provides numerical evidence to support this. More specifically, two-parameter continuation of the torus bifurcation from Figure \ref{fig:mltfullsystem}(e) shows that the curve of torus bifurcations converges to the curve of toral folded singularities in the singular limit $\eps \to 0$.

\begin{remark}
In this degenerate setting of one slow variable, the toral folded singularity is a toral FSN II. Thus, we can apply the numerical continuation method outlined in Appendix \ref{app:TFScurve} to compute the loci of toral folded singularities in two-parameter families of 2-fast/1-slow systems.
%We outline the procedure in Appendix \ref{app:TFScurve}. 
\end{remark}

Our results here complement and extend those of \cite{Burke2012}. In \cite{Burke2012}, the location of the SNPOs in the $(y,k,g_{Ca})$ space was determined, but the spiking/bursting boundary was not computed. Here, we have shown that the toral FSN II is the appropriate singularity to track in order to determine the spiking/bursting boundary. Note that this is only a local partitioning of the parameter space; other bifurcations of the layer problem can appear and alter the dynamics \cite{Burke2012}.

%----------------------------------------------------
\subsection{Torus Canards in the Hindmarsh-Rose Model} 	\label{subsec:TCHR}
%----------------------------------------------------

We next consider the modified Hindmarsh-Rose system \cite{Burke2012,Tsaneva2010}
\begin{equation}	\label{eq:HR}
\begin{split}
\dot{x} &= sax^3-sx^2-y-bz, \\
\dot{y} &= \phi(x^2-y), \\
\dot{z} &= \eps \left( s\alpha x+\beta-kz \right),
\end{split}
\end{equation}
where $\beta$ and $s$ are taken to be the primary and secondary control parameters, respectively. The other parameters are fixed at
\[ a=0.5, \quad \phi=1, \quad \alpha = -0.1, \quad k=0.2, \quad b=10. \]
The Hindmarsh-Rose model is known to exhibit various types of bursting, such as plateau and pseudo-plateau \cite{Tsaneva2010}. In \cite{Burke2012}, the Hindmarsh-Rose system was also shown to possess subHopf/fold-cycle (i.e., elliptic) bursting. Moreover, torus canards were demonstrated to occur precisely in the transition region from spiking to elliptic bursting. 

Here, we extend those results by calculating the underlying toral folded singularity, and computing its locus in the $(\beta,s)$ plane (Figure \ref{fig:hr}). As in the MLT model, the toral folded singularity corresponds to the scenario in which an ordinary singularity of the averaged radial-slow flow crosses from $\mathcal{P}_a$ to $\mathcal{P}_r$ via an SNPO. 
%That is, the toral folded singularity is a limit cycle, $\Gamma$, of the layer problem of \eqref{eq:HR} such that 
%
%\[ \int_0^{T} \left. \left( 3sax^2-2sx-\phi \right) \right|_{\Gamma}\, dt = 0, \quad \text{ and } \quad \int_0^{T} \left. \left( s\alpha x+\beta - k z \right) \right|_{\Gamma}\, dt =0, \]
%
%where $T$ is the period of $\Gamma$. 
These toral folded singularities split the $(\beta,s)$ plane between spiking and bursting behaviour. 

\begin{figure}[h!]
\centering
\includegraphics[width=3.5in]{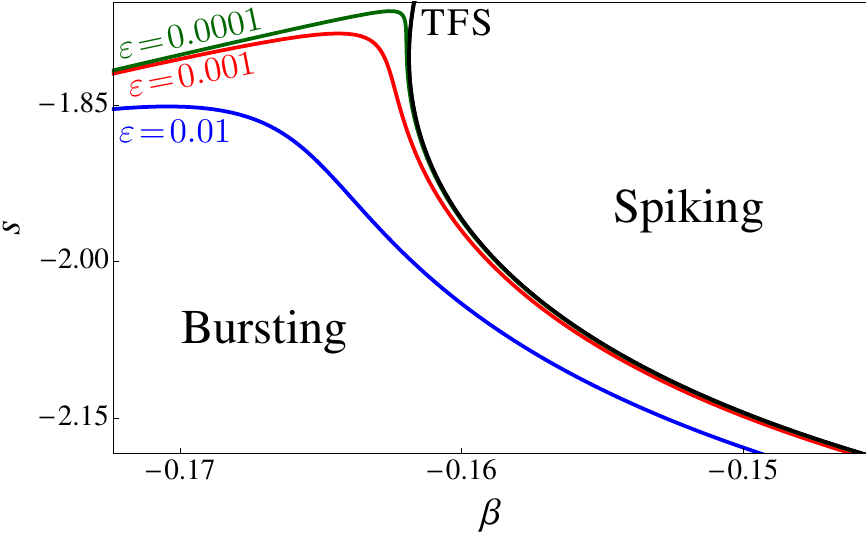}
\caption{Two parameter bifurcation structure of \eqref{eq:HR} with respect to $\beta$ and $s$. In the singular limit, the spiking/bursting boundary is given by the curve of toral folded singularities (TFS, black). Away from the singular limit, the spiking/bursting transition occurs at a torus bifurcation. The blue, red, and green curves are the two-parameter continuations of this torus bifurcation for $\eps=0.01$, $\eps=0.001$, and $\eps=0.0001$, respectively.}
\label{fig:hr}
\end{figure}

For non-zero perturbations, the Hindmarsh-Rose system transitions between spiking and bursting behaviour via a torus bifurcation. Figure \ref{fig:hr} shows that the two-parameter continuation of this torus bifurcation generates a folded curve in the $(\beta,s)$ plane (coloured curves). The right branch of this folded curve tends to the curve of toral folded singularities as $\eps \to 0$. Thus, we have further numerical evidence to support our conjecture from Section \ref{subsec:toralFSN} that the $\eps$-unfolding of a toral FSN II is a singular torus bifurcation. That is, the transition between spiking and bursting along the curve of singular torus bifurcations is mediated by a sequence of torus canards. 
Note that the left branch of the curve of torus bifurcations does not converge to the toral FSN II curve. That is, the toral FSN II is not the only way to create an invariant phase space torus.

%----------------------------------------------------
\subsection{Torus Canards in the Wilson-Cowan-Izhikevich Model} 	\label{subsec:TCWCI}
%----------------------------------------------------

As a final demonstration, we briefly consider the Wilson-Cowan-Izhikevich model \cite{Burke2012,Izhikevich2000} for interacting populations of excitatory and inhibitory neurons. The model equations are given by 
\begin{equation} 	\label{eq:WCI}
\begin{split}
\dot{x} &= -x+S(r_x+ax-by+u), \\
\dot{y} &= -y+S(r_y+cx-dy+fu), \\
\dot{u} &= \eps(k-x),
\end{split}
\end{equation}
where $S(x) = (1+\exp(-x))^{-1}$. The small parameter $\eps$ induces a separation of time-scales, so that $x$ and $y$ are fast, and $u$ is slow.  
The Wilson-Cowan-Izhikevich model can exhibit a wide array of different bursting dynamics \cite{Burke2012}. Here, we focus on its fold/fold-cycle bursting dynamics. We treat $k$ and $r_x$ as the principal bifurcation parameters, and keep all other parameters fixed at the values
\[ r_y=-9.7, \quad a=10.5, \quad b=10, \quad c=10, \quad d=-2, \quad f=0.3. \]
Figure \ref{fig:wci} shows the curve of toral folded singularities in the $(k,r_x)$ plane. As in the MLT and Hindmarsh-Rose models, the TFS curve is the singular limit boundary between spiking and bursting behaviour. Moreover, the toral folded singularities in this degenerate setting of one slow variable are toral FSNs of type II.

\begin{figure}[h!]
\centering
\includegraphics[width=3.5in]{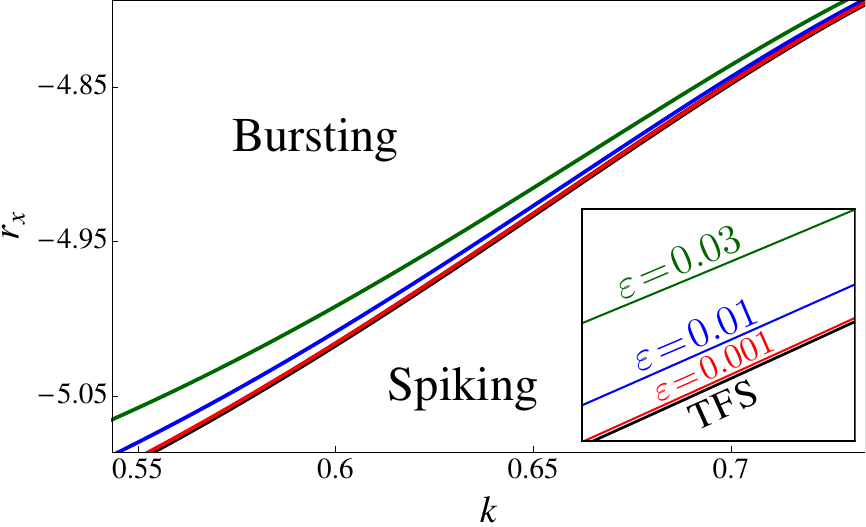}
\caption{Toral folded singularities (TFS, black) for system \eqref{eq:WCI} in the $(k,r_x)$ plane.   The curve of torus bifurcations (coloured) converges to the TFS curve in the singular limit $\eps \to 0$.}
\label{fig:wci}
\end{figure}

For $\eps$ small and positive, the transition from spiking to bursting occurs by way of a torus canard explosion. Two-parameter continuation of the associated torus bifurcation generates a curve that locally splits the parameter plane between spiking and bursting. 
Figure \ref{fig:wci} shows that these curves of torus bifurcations converge to the curve of toral folded singularities in the singular limit $\eps \to 0$, which again supports the conjecture from Section \ref{subsec:toralFSN}.

%---------------------------------------------------------------------------------
\section{Proof of the Averaging Theorem for Folded Manifolds of Limit Cycles}		\label{sec:arbitrarydimensions}
%---------------------------------------------------------------------------------

We now turn our attention to formally proving of our main theoretical result (Theorem \ref{thm:averaging}), namely the extension of the averaging method for slow/fast systems in the neighbourhood of a folded manifold of limit cycles. In this section, we state and prove the main result in the general case of two fast variables and $k$ slow variables, where $k$ is any positive integer. 

We consider slow/fast systems \eqref{eq:main} with $n=2$ fast variables and $k \geq 1$ slow variables. The only modification to Assumptions \ref{ass:man}, \ref{ass:fold}, and \ref{ass:homoclinic} is the dimension of the vector of slow variables. 
We start with the following result, which will be needed in the proof of the main result. 

\begin{lemma}	\label{lemma:linear}
Given Assumptions \ref{ass:man}, \ref{ass:fold}, and \ref{ass:homoclinic}, let $(\Gamma(t,y),y) \in \mathcal{P}_L$ and $q(t,y)$ be a unit normal for $\Gamma(t,y)$. Then
\[ \int_0^{T(y)} q \cdot (D_x f)\, q \, dt = \int_0^{T(y)} \left( \operatorname{tr}\, D_x f - \frac{f \cdot (D_x f) \, f}{\lVert f \rVert^2} \right) \, dt = 0, \]
where $f$ and its derivatives $D_xf$ are evaluated along $(\Gamma(t,y),y)$.
\end{lemma}

\begin{proof}
We first rewrite the expression $q \cdot (D_xf) \,q$ as follows.
\begin{align*}
q \cdot (D_x f)\, q &= \frac{1}{\lVert f \rVert^2} \begin{pmatrix} f_2 \\ -f_1 \end{pmatrix} \cdot \begin{pmatrix} f_{1x_1} & f_{1x_2} \\ f_{2x_1} & f_{2x_2} \end{pmatrix} \begin{pmatrix} f_2 \\ -f_1 \end{pmatrix}, \\
&= \frac{1}{\lVert f \rVert^2} \begin{pmatrix} f_1 \\ f_2 \end{pmatrix} \cdot \begin{pmatrix} f_{2x_2} & -f_{1x_2} \\ -f_{2x_1} & f_{1x_1} \end{pmatrix} \begin{pmatrix} f_1 \\ f_2 \end{pmatrix}, \\
&= \frac{1}{\lVert f \rVert^2} f \cdot \left( \operatorname{adj}\, D_x f \right) f,
\end{align*}
where $\operatorname{adj}$ denotes the classical adjoint. Since $D_x f$ is invertible along $\Gamma$, we can rewrite the adjoint in terms of the inverse, which gives
\begin{equation}	\label{eq:qdxfq}
q \cdot (D_x f)\, q = \frac{\det D_x f}{\lVert f \rVert^2} \,f \cdot \left( D_x f\right)^{-1} f.
\end{equation}
An application of the Cayley-Hamilton theorem gives the following equivalent expression for $(D_x f)^{-1}$:
\[ (D_xf)^{-1} = \frac{1}{\det D_x f} \left( \left(\operatorname{tr}\, D_x f \right) \, \mathbb{I}_2 - D_x f  \right), \]
where $\mathbb{I}_2$ denotes the $2 \times 2$ identity matrix. Substituting into \eqref{eq:qdxfq}, we have
\begin{equation} \label{eq:qDxfq}
q \cdot (D_x f)\, q = \operatorname{tr}\, D_xf - \frac{f \cdot \left( D_x f \right)\, f}{\lVert f \rVert^2},
\end{equation}
which holds for any limit cycle in $\mathcal{P}$.
Now, integrating over one period of $\Gamma(t,y)$, we see that the first term on the right hand side is just the Floquet exponent and hence vanishes since $(\Gamma(t,y),y) \in \mathcal{P}_L$. It remains to show that the second term on the right hand side has zero average. This follows from the fact that 
\[ \frac{f \cdot \left( D_x f \right)\, f}{\lVert f \rVert^2} = \frac{1}{2}\, \frac{d}{dt} \!\left( \log \left( \frac{1}{2} f \cdot f \right) \right), \]
and all functions are being evaluated over $T(y)$-periodic arguments.
\end{proof}

We are now in a position to prove Theorem \ref{thm:averaging}, which we restate in Theorem \ref{thm:averagingarbitraryslow} in the more general case of $k$ slow variables, where $k$ is any positive integer.

\begin{theorem}[\bf{Averaging Theorem for $k$-slow variables}] \label{thm:averagingarbitraryslow}
Consider system \eqref{eq:main} with $x \in \mathbb{R}^2$ and $y\in \mathbb{R}^k$ under assumptions \ref{ass:man}, \ref{ass:fold}, and \ref{ass:homoclinic}, and let $(\Gamma(t,y),y) \in \mathcal{P}_L$. Then there exists a sequence of near-identity transformations such that the averaged dynamics of \eqref{eq:main} in a neighbourhood of $(\Gamma(t,y),y)$ are approximated by
\begin{equation} \label{eq:averagingarbitraryslow}
\begin{split}
\dot{R} &= \overline{a} \cdot u + \overline{b} R^2 + \overline{c} \cdot R u + \mathcal{O}(\eps,R^3,R^2u,u^2), \\
\dot{u} &= \eps \left( \overline{g} + \overline{d} R + \overline{e} u + \mathcal{O}(\eps,R^2,Ru, u^2) \right),
\end{split}
\end{equation}
where an overline denotes an average over one period of $\Gamma(t,y)$ and the averaged coefficients can be computed explicitly (see Section \ref{subsec:averagedcoefficients}).
\end{theorem}

\begin{proof}
We first give the proof in the case $k=2$, and then describe the required modifications for arbitrary $k$. Without loss of generality, we assume that there is a limit cycle in $\mathcal{P}_L$ with $y=0$, and we denote this periodic solution by $\Gamma_0(t)$. In the first step, we make a coordinate transformation that switches the fast variables to a coordinate frame that moves with the limit cycle $\Gamma_0$, and splits the slow motions into their mean and (small) fluctuating parts. This is achieved via the coordinate transformation 
\begin{equation}	\label{eq:xytrans}
\begin{split}
x &= \Gamma_0(t) + r(t) \, q(t,0), \\
y &= u(t) + \eps w(t,u,r),
\end{split}
\end{equation}
where $r$ is the small (real-valued) radial perturbation from $\Gamma_0$ in the direction of a unit normal $q$ to $\Gamma_0$, and $u$ and $w$ are designed to be the mean and fluctuating components of $y$, respectively. Substituting \eqref{eq:xytrans} into \eqref{eq:main} and Taylor expanding $f$ and $g$ gives
\begin{align*}
\frac{dr}{dt} q + r \frac{dq}{dt} &= \left( D_x f \,q \right) r+ \frac{1}{2} \begin{pmatrix} (q \cdot \nabla_x)^2 \,f_1 \\ (q \cdot \nabla_x)^2 \,f_2 \end{pmatrix} r^2 + (D_yf)\, u + \begin{pmatrix} q \cdot (D_{xy}f_1 \,u) \\ q \cdot (D_{xy}f_2 \,u) \end{pmatrix} r + \mathcal{O}\left( \eps, r^3, r^2u, u^2 \right), \\
\frac{du}{dt} &= \eps \left( g+\left( D_x g\,q \right) r + (D_yg)\, u+\mathcal{O}(\eps,ru,r^2) - \frac{dw}{dt} \right),
\end{align*}
where $(D_{xy}f_k)_{ij} := \left( \frac{\partial^2 f_k}{\partial x_i \partial y_j} \right)$, for $i=1,2$ and $j=1,2$, denotes the matrix of mixed second order derivatives of $f_k$, where $k=1,2$. Note that $f, g$, and their derivatives are evaluated at $(\Gamma_0,0)$, and we have used the fact that $\Gamma_0 \in \mathcal{P}$ so that $\frac{d\Gamma_0}{dt} = f(\Gamma_0,0)$. To isolate the radial evolution, we project the fast components in the direction of $q$, which is equivalent to left-multiplication by the matrix
\[ \begin{pmatrix} q_1 & q_2 & 0 \\ 0 & 0 & 1 \end{pmatrix}, \]
where $q_1$ and $q_2$ are the components of the unit normal $q$. Thus, taking the projection in the direction of $q$, we obtain
\begin{align*}
\frac{dr}{dt} &= q \cdot \left( D_x f\right) q\, r+ \frac{1}{2} q\cdot \begin{pmatrix} (q \cdot \nabla_x)^2 \,f_1 \\ (q \cdot \nabla_x)^2 \,f_2 \end{pmatrix} r^2 + q\cdot (D_yf\, u) + q\cdot \begin{pmatrix} q \cdot (D_{xy}f_1 \,u) \\ q \cdot (D_{xy}f_2 \,u) \end{pmatrix} r+ \mathcal{O}\left( \eps, r^3, r^2u, u^2 \right), \\
\frac{du}{dt} &= \eps \left( g+\left( D_x g\,q \right) r + (D_yg)\, u+\mathcal{O}(\eps,ru,r^2) - \frac{dw}{dt} \right).
\end{align*}
Now, let $H$ be the matrix
\[ H_{ij} := \left( q\cdot \nabla_x \left( \frac{\partial f_i}{\partial y_j} \right) \right). \]
Then, by a straightforward calculation, we can rewrite the $\mathcal{O}(ru)$ terms in the radial equation as 
\[ q\cdot \begin{pmatrix} q \cdot (D_{xy}f_1 \,u) \\ q \cdot (D_{xy}f_2 \,u) \end{pmatrix} r  = q \cdot \left( Hu \right)\, r = \left( H^T q \right) \cdot r u. 
% = q\cdot \begin{pmatrix} ((D_{xy}f_1)^T q) \cdot u \\ ((D_{xy}f_2)^T q) \cdot u \end{pmatrix} r
\]
Moreover, by equation \eqref{eq:qDxfq}, we can rewrite linear $r$-term in the $r$-equation to give
\begin{equation}	\label{eq:ru}
\begin{split}
\frac{dr}{dt} &= \left( \operatorname{tr} D_xf - \frac{f \cdot \left( D_x f \right) f}{\lVert f \rVert^2} \right)\! r+ \frac{1}{2}q \cdot \! \begin{pmatrix} (q \cdot \! \nabla_x)^2 \,f_1 \\ (q \cdot \! \nabla_x)^2 \,f_2 \end{pmatrix} r^2 + q\cdot (D_yf\, u) + \left( H^T q \right) \cdot r u + \mathcal{O}\!\left( \eps, r^3,r^2u, u^2 \right), \\
\frac{du}{dt} &= \eps \left( g+\left( D_x g\,q \right) r + (D_yg)\, u+\mathcal{O}(\eps,ru,r^2) - \frac{dw}{dt} \right).
\end{split}
\end{equation}

Now, let $\Phi$ be the fundamental solution of the linear radial flow. That is, $\Phi$ satisfies
\[ \frac{d\Phi}{dt} = \left( \operatorname{tr} D_xf - \frac{f \cdot \left( D_x f \right) f}{\lVert f \rVert^2} \!\right)\, \Phi, \qquad \Phi(0)=1. \]
By Lemma \ref{lemma:linear}, the solution $\Phi$ is $T(0)$-periodic and bounded for all time with the explicit solution
\[ \Phi(t) = \exp \left\{ \int_0^t \operatorname{tr} D_xf(\Gamma_0(s),0,0) \, ds \right\} \frac{\lVert f(\Gamma_0(0),0,0) \rVert}{\lVert f(\Gamma_0(t),0,0) \rVert}. \]
Letting $r = \Phi\, \tilde{r}$ removes the linear $r$-term from the radial evolution equation. Thus, after transformation, system \eqref{eq:ru} becomes
\begin{equation}	\label{eq:rtildeu}
\begin{split}
\frac{d\tilde{r}}{dt} &= \frac{1}{2} q\cdot \begin{pmatrix} (q \cdot \nabla_x)^2 \,f_1 \\ (q \cdot \nabla_x)^2 \,f_2 \end{pmatrix} \Phi \,\tilde{r}^2 + \frac{1}{\Phi} ((D_yf)^T q ) \cdot u + \left( H^T q \right) \cdot \tilde{r} u + \mathcal{O}\left( \eps, \tilde{r}^3, \tilde{r}^2u, u^2 \right), \\
\frac{du}{dt} &= \eps \left( g+\left( D_x g\,q \right) \Phi\, \tilde{r} + (D_yg)\, u+\mathcal{O}(\eps,\tilde{r}u,\tilde{r}^2) - \frac{dw}{dt} \right).
\end{split}
\end{equation}
At this stage, we simplify the notation by letting
\begin{align*} 
a(t) := \frac{1}{\Phi} ((D_yf)^T q ), \,\, b(t) := \frac{1}{2} q\cdot \begin{pmatrix} (q \cdot \nabla_x)^2 \,f_1 \\ (q \cdot \nabla_x)^2 \,f_2 \end{pmatrix} \Phi, \,\, \tilde{c}(t) := H^T q,  \, \text{ and }\, d(t) := \left( D_x g\,q \right) \Phi,
\end{align*}
denote the vectors of coefficients in system \eqref{eq:rtildeu}. Note that $b(t)$ is actually a scalar function. Then system \eqref{eq:rtildeu} becomes 
\begin{equation}	\label{eq:rtildeu2}
\begin{split}
\frac{d\tilde{r}}{dt} &= b(t) \,\tilde{r}^2 + a(t) \cdot u + \tilde{c}(t) \cdot \tilde{r}\,u+ \mathcal{O}\left( \eps, \tilde{r}^3, \tilde{r}^2u, u^2 \right), \\
\frac{du}{dt} &= \eps \left( g+d(t)\, \tilde{r} + (D_yg)\, u+\mathcal{O}(\eps,\tilde{r}u,\tilde{r}^2) - \frac{dw}{dt} \right).
\end{split}
\end{equation}

Next, we introduce a near-identity coordinate transformation
\begin{align}	\label{eq:NIT}
\tilde{r} = R + \alpha \cdot u + \beta R^2+ \gamma \cdot R u + \mathcal{O}(R^3,R^2u, u^2), 
\end{align}
where $R$ and $u$ are small. The idea is to choose $\alpha, \beta, \gamma$, and $w$ to remove the small fluctuations over one period of $\Gamma_0$, leaving only the mean contributions. Substituting \eqref{eq:NIT} into system \eqref{eq:rtildeu2} gives
\begin{equation}	\label{eq:Runonav} 
\begin{split}
\frac{dR}{dt} &= \left( b-\frac{d\beta}{dt} \right) R^2 + \left( a-\frac{d\alpha}{dt} \right) \cdot u + \left( c - \frac{d\gamma}{dt} \right) \cdot R u + \mathcal{O}(\eps,R^3,R^2u, u^2), \\
\frac{du}{dt} &= \eps \left( g + d(t)\, R + (D_y g + d(t) \alpha(t)^T)\, u + \mathcal{O}(\eps,R u, R^2) - \frac{dw}{dt} \right),
\end{split}
\end{equation}
where $c:= \tilde{c}+2\alpha b-2\beta \left( a- \frac{d\alpha}{dt} \right)$. In general, the integral of $b(t)$ over one period of $\Gamma_0$ has nonzero average, and so we cannot choose $\beta$ to remove $b(t)$ completely, otherwise the $\beta R^2$ term in \eqref{eq:NIT} would become large over $\mathcal{O}\left( \eps^{-1} \right)$ times. Thus, at most, we can remove everything but the mean of $b(t)$ by choosing $\beta$ so that
\begin{align*}
\frac{d\beta}{dt} &= b(t) - \frac{1}{T} \int_0^{T(0)} b(t)\, dt, \quad \beta(0)=0, 
\end{align*}
which has a bounded $T$-periodic solution. Similarly, we cannot completely remove the linear $u$-terms in the $R$-equation. We can remove everything but the average of $a(t)$ by making the choice
\begin{align*}
\frac{d\alpha_j}{dt} &= a_j(t) - \frac{1}{T} \int_0^{T(0)} a_j(t)\, dt, \quad \alpha_j(0)=0, \text{  for } j=1,2,
\end{align*}
which has bounded $T$-periodic solutions. Again, everything but the mean of the coefficients of the $Ru$ terms in the $R$-equation can be removed by setting
\begin{align*}
\frac{d\gamma_j}{dt} &= c_j(t) - \frac{1}{T} \int_0^{T(0)} c_j(t)\, dt, \quad \gamma_j(0)=0, \text{  for } j=1,2,
\end{align*}
which yields bounded, $T$-periodic solutions. Iteratively choosing the higher order terms in \eqref{eq:NIT} in the same way allows us to average the fast radial equation in \eqref{eq:Runonav}, giving
\begin{align*} 
\frac{dR}{dt} &= \overline{b} R^2 + \overline{a} \cdot u + \overline{c} \cdot R\, u + \mathcal{O}(\eps,R^3, R^2u, u^2), \\
\frac{du}{dt} &= \eps \left( g + d(t)\, R + (D_y g + d(t) \alpha(t)^T)\, u + \mathcal{O}(\eps,R u, R^2) - \frac{dw}{dt} \right),
\end{align*}
where the overline denotes averages over $\Gamma_0(t)$. 

Finally, to complete the proof, we average the slow motions by expanding $w$ as a power series in $R$ and $u$, and choosing the coefficients in that expansion in a manner analogous to the above. More precisely, let 
\[ e_{jl}(t) := (D_y g)_{jl} + d_j(t) \,\alpha_l(t), \,\, \text{ for } j=1,2, \,\text{ and } l=1,2, \]
denote the matrix of coefficients of the linear $u$-terms in the $u$-equations, and choose the components of $w$ such that
\[ \frac{dw_j}{dt} = g_j-\overline{g}_j + \left( d_j(t) - \overline{d}_j \right)R + \sum_{l=1}^2 \left( e_{jl}-\overline{e}_{jl} \right)u_l+\mathcal{O}(\eps,Ru,R^2), \]
for $j=1,2$. 

To obtain the extension to an arbitrary number, $k$, of slow variables, we simply replace the ranges of the indices $j$ and $l$ above with $j = 1,2,\ldots, k$ and $l = 1,2,\ldots, k$. This completes the proof. 
\end{proof}

By Theorem \ref{thm:averagingarbitraryslow}, we now have the 2-fast/$k$-slow analogue for the detection of toral folded singularities, and hence torus canards. A toral folded singularity is a folded limit cycle $\Gamma \in \mathcal{P}_L$ such that
\[ \overline{a} \cdot \overline{g} = 0, \]
where the overline denotes an average over one period of $\Gamma$. Again, this normal switching condition is a statement of tangency (in the slow subspace) between the projection of the slow drift along $\mathcal{P}$ and the projection of $\mathcal{P}_L$, at the toral folded singularity. The classification of the toral folded singularity is again based on its classification as a folded singularity of the $1$-fast/$k$-slow averaged radial-slow system \eqref{eq:averagingarbitraryslow}. More specifically, for $k \geq 2$, the desingularized reduced system of \eqref{eq:averagingarbitraryslow} generically has a $(k-2)$-dimensional submanifold, $\mathcal{M}$, of folded singularities. As such, each folded singularity in $\mathcal{M}$ has $(k-2)$-zero eigenvalues \cite{Wechselberger2012}. Thus, the toral folded singularity is classified according to the remaining two eigenvalues (Definition \ref{def:class}). In the degenerate setting of $k=1$ slow variable, the toral folded singularity is typically a toral FSN of type II.

\begin{remark}	\label{rem:hyperbolic}
The averaging procedure developed in this article for folded manifolds of limit cycles recovers the results of averaging theory on normally hyperbolic manifolds of limit cycles \cite{Pontryagin1960}. More precisely, consider an attracting limit cycle $(\Gamma(t,y),y) \in \mathcal{P}_a$. Then Lemma \ref{lemma:linear} becomes
\[ \frac{1}{T}\int_0^T q \cdot (D_x f) \,q\,dt = \frac{1}{T} \int_0^T \operatorname{tr}\, D_x f\, dt =  \varphi_2 <0, \]
where $\varphi_2$ is the (non-trivial) Floquet exponent of $\Gamma$. This means that the averaging process implemented in the proof of Theorems \ref{thm:averaging} and \ref{thm:averagingarbitraryslow} cannot completely remove the linear $r$-term from the radial evolution equation. Consequently, the slow/fast system that results from the averaging process has a normally hyperbolic (attracting) critical manifold $\mathcal{S}_a$, which corresponds to the attracting manifold of limit cycles. Fenichel theory then states that $\mathcal{S}_a$ will persist as a (locally) invariant attracting slow manifold, $\mathcal{S}_a^{\eps}$, which corresponds to a locally invariant manifold of limit cycles $\mathcal{P}_a^{\eps}$ for sufficiently small $\eps$. Similarly, the normally hyperbolic segments of the repelling manifold of limit cycles $\mathcal{P}_r$ will persist as (locally) invariant repelling manifolds $\mathcal{P}_r^{\eps}$ for sufficiently small $\eps$.
\end{remark}

We now provide asymptotic error estimates for the averaging method (Theorem \ref{thm:averagingarbitraryslow}) in the main cases. Without loss of generality, let $(\Gamma_0(t),0)$ be a limit cycle on $\mathcal{P}_L$, and let $q(t,0)$ denote the unit normal to $\Gamma_0$. Let $(x(t),y(t))$ be a solution of \eqref{eq:main} and $(R(t),u(t))$ be a solution of the averaged radial-slow system \eqref{eq:averagingarbitraryslow}, where system \eqref{eq:averagingarbitraryslow} is obtained by expanding and averaging around $(\Gamma_0,0)$.

\begin{theorem} \label{cor:error}
If $\lVert \left(x(0) \cdot q(0,0) - R(0) , y(0) - u(0) \right) \rVert = \mathcal{O}(\eps)$, then for $0< \eps \ll 1$, we have
\begin{align*}
\left\lVert \begin{pmatrix} x(t) \cdot q(t,0) - R(t) \\ y(t) - u(t) \end{pmatrix} \right\rVert  &= \mathcal{O}\left( \eps^p \right),
\end{align*}
for $\mathcal{O} \left(\eps^{-p}  \right)$ times $t$, where 
\begin{itemize}
\item $p = \frac{1}{2}$ for a TFN, toral folded saddle, and toral FSN II, 
\item $p = \frac{1}{4}$ for a toral FSN I, and 
\item $p = \frac{1}{3}$ for a regular folded limit cycle.
\end{itemize}
\end{theorem}

\begin{proof}
We first sketch the proof in the classic case of a slow/fast system with a normally hyperbolic attracting manifold of limit cycles \cite{Berglund2006,Kuehn2015}, i.e., $\mathcal{P} = \mathcal{P}_a$. The averaging method produces a singularly perturbed slow/fast system with attracting critical manifold $\mathcal{S}_a$ (see Remark \ref{rem:hyperbolic}). Fenichel theory \cite{Fenichel1979,Jones1995} guarantees that $\mathcal{S}_a$ persists as an invariant slow manifold $\mathcal{S}_a^{\eps}$ of the averaged radial-slow system. The slow flow restricted to $\mathcal{S}_a^{\eps}$ is an $\mathcal{O}(\eps)$ perturbation of the reduced flow on $\mathcal{S}$. That is, the slow manifold has a local graph representation, $R = \hat{R}(u,\eps)$, where $\hat{R} = \mathcal{O}(\eps)$. Substituting this graph representation and Taylor expanding in powers of $\eps$, the averaged slow flow restricted to $\mathcal{S}_a^{\eps}$ is given by a system of the form
\[ \dot{u} = \eps \overline{G}_1(u)+ \eps^2 G_2(u,t) + \mathcal{O}\left( \eps^3 \right), \]
where $\overline{G}_1$ describes the leading order averaged dynamics and $G_2$ is periodic in time.
This slow flow on $\mathcal{S}_a^{\eps}$ is in the standard form for the classical averaging method \cite{Pontryagin1960,Sanders2007}. Thus, the full system trajectory and the averaged radial-slow flow are $\mathcal{O}(\eps)$ close for $\mathcal{O}(\eps^{-1})$ times on the fast time-scale. 

Suppose now that we have a folded manifold of limit cycles and consider the TFN case. By Theorem \ref{thm:averagingarbitraryslow}, the averaged radial-slow system has a folded critical manifold $\mathcal{S} = \mathcal{S}_a \cup \mathcal{L} \cup \mathcal{S}_r$. 
The blow-up technique \cite{Roussarie1993} extends Fenichel theory up to $\mathcal{O}(\sqrt{\eps})$ neighbourhoods of the fold curve $\mathcal{L}$. That is, $\mathcal{S}_a$ and $\mathcal{S}_r$, persist as invariant slow manifolds, $\mathcal{S}_a^{\sqrt{\eps}}$ and $\mathcal{S}_r^{\sqrt{\eps}}$, which are $\mathcal{O}(\sqrt{\eps})$ perturbations of $\mathcal{S}_a$ and $\mathcal{S}_r$. 
The slow flow restricted to $\mathcal{S}_a^{\sqrt{\eps}}$ and $\mathcal{S}_r^{\sqrt{\eps}}$ is a smooth $\mathcal{O}(\sqrt{\eps})$ perturbation of the reduced flow on $\mathcal{S}$ \cite{Krupa2010,Szmolyan2001}. 
Moreover, the solutions of the averaged radial-slow system that comprise $\mathcal{S}_a^{\sqrt{\eps}}$ and $\mathcal{S}_r^{\sqrt{\eps}}$ twist in the neighbourhood of the TFN as they pass the fold curve. These rotations are confined to an $\mathcal{O}(\sqrt{\eps})$ neighbourhood of the TFN (i.e., to the central chart of the blow-up; see \cite{Szmolyan2001,Wechselberger2005}).
Thus, for canard solutions of the averaged radial-slow flow, the classical averaging theorem \cite{Pontryagin1960,Sanders2007} applied to the slow flow on $\mathcal{S}_a^{\sqrt{\eps}}$ and $\mathcal{S}_r^{\sqrt{\eps}}$ gives the asymptotic error estimate. 

The result for the toral folded saddle is obtained in the same way. Note that for the toral folded saddle, the asymptotic error estimates only apply to torus canards of the toral folded saddle. Solutions that reach the manifold of SNPOs or toral faux canard solutions will have different asymptotic behaviour. We leave the analysis of toral faux canards to future work.

For the toral FSN I and toral FSN II, we again apply the same arguments to obtain the asymptotic estimates for validity of the averaging method. The different powers of $\eps$ come from the blow-up analysis of the classical FSN I \cite{Vo2015} and the classical FSN II \cite{Krupa2010}, respectively. 

Finally, we consider the case of a regular folded limit cycle (where $\rho_0 := \overline{a} \cdot \overline{g} \neq 0$). The averaging method produces a slow/fast system with folded critical manifold as before. 
Blow-up analysis of the flow past the fold $\mathcal{L}$ \cite{Krupa2001,Szmolyan2004} in the averaged radial-slow system gives the estimate $(R,u) = \mathcal{O}\left(\eps^{1/3} \right)$ for solutions in $S_a^{\eps}$ as they exit the central chart of the blow-up. Once again, the classical averaging theorem applied to the slow flow restricted to $\mathcal{S}_a^{\eps}$ gives the result.
\end{proof}

%---------------------------------------------------------------------------------
\section{Discussion}		\label{sec:discussion}
%---------------------------------------------------------------------------------

Torus canards are special solutions of slow/fast systems with at least two fast variables that alternately spend long times near attracting and repelling sets of limit cycles of the layer problem. Typically, the torus canards manifest as amplitude-modulated rhythms. They have been demonstrated to mediate the transition between tonic spiking and bursting states in several computational neural models, such as a cerebellar Purkinje cell model \cite{Kramer2008}, the Morris-Lecar-Terman, Hindmarsh-Rose and Wilson-Cowan-Izhikevich models \cite{Burke2012}, as well as in a model of synaptically coupled respiratory neurons in the pre-B\"{o}tzinger complex \cite{Roberts2015}. 

In slow/fast systems with only one slow variable, the torus canards are degenerate. They require one-parameter families of 2-fast/1-slow systems to be observed, and even then, they only occur on exponentially small parameter sets. The addition of a second slow variable makes the torus canards generic and robust, and therefore experimentally observable. The current approach in the literature to the study of torus canards is to analyse the average slow dynamics over limit cycles of the layer problem. Whilst these methods have, so far, led to reasonable conclusions about the dynamics, they are not rigorously justified. The averaging method \cite{Pontryagin1960} has only been developed for normally hyperbolic manifolds of limit cycles, whereas the torus canards occur in the neighbourhood of the manifold of SNPOs, where the averaging method breaks down. The primary aim of this article then was to develop a rigorous theoretical framework for the analysis of torus canards.

%----------------------------------------------------
\subsection{Summary of Main Results} 	\label{subsec:summary}
%----------------------------------------------------

In this article, we achieved three major results. First, we developed an extension of the averaging method for folded manifolds of limit cycles in slow/fast systems with two fast variables and $k$ slow variables, for any $k \geq 1$. We proved that the averaged radial-slow dynamics in the neighbourhood of a folded manifold of limit cycles is equivalent to the dynamics of a slow/fast system restricted to the neighbourhood of a folded critical manifold. By combining our averaging theory for folded manifolds of limit cycles with canard theory, we derived analytic criteria to identify and characterize torus canards based on a new class of singularities for differential equations: the toral folded singularities.  
These toral folded singularities encode the behaviour of the corresponding (non-singular) torus canards for $0< \eps \ll 1$ in the eigenvalues of the toral folded singularity itself, analogous to the way canards in $\mathbb{R}^3$ are characterized by their associated folded singularities.
We established that the torus canards of a toral folded node (TFN) are the result of three motions working in concert: (i) rapid oscillations due to limit cycles of the layer problem, (ii) net slow drift along the manifold of periodics towards the TFN, and (iii) amplitude-modulation due to canard dynamics in the envelope of the rapidly oscillating waveform. That is, the delayed passage properties past the manifold of SNPOs is inherited from canard dynamics on the envelope. 

The second major result was the discovery and analysis of amplitude-modulated subcritical elliptic bursting solutions, which alternate between epochs of rapid amplitude-modulated spiking, and quiescent periods. Using our torus canard theory, we showed that these novel AMB rhythms were torus canard-induced mixed-mode oscillations. The local mechanism that generated the amplitude-modulation was the TFN, which caused a local twisting of the invariant manifolds, $\mathcal{P}_a^{\eps}$ and $\mathcal{P}_r^{\eps}$, of limit cycles. The global return mechanism was the slow passage of trajectories through a delayed Hopf bifurcation, which allowed trajectories to escape the silent phase and return to one of the rotational sectors of $\mathcal{P}_a^{\eps}$ induced by the maximal torus canards. 

Our third main result was establishing the connection between our analysis and prior work on torus canards in the degenerate setting of one slow variable. We showed that the transition between tonic spiking and bursting in 2-fast/1-slow systems could be detected by simply tracking an appropriate singularity (the toral folded singularity) in parameter space. We illustrated our results in the Morris-Lecar-Terman, Hindmarsh-Rose and Wilson-Cowan-Izhikevich models for neural bursting. Hence, these results extend and explain some of those in \cite{Burke2012}.
Moreover, the torus canard theory allowed the determination of the parameter values for which torus canard explosions occur (Appendix \ref{app:TCexplosion}). 

%----------------------------------------------------
\subsection{Relation to Classical Canards and Prior Studies of Torus Canards} 	\label{subsec:context}
%----------------------------------------------------

To further place our work in context, we point out that until now, there were only two theoretical statements concerning torus canards. The first was the topological necessity of torus canards in $\mathbb{R}^3$, which follows from the continuous dependence of solutions on parameters \cite{Burke2012}. This continuous dependence shows that the only allowable homotopy from spiking to bursting solutions in 2-fast/1-slow systems is via the sequence of torus canards.
The other theoretical result is that torus canard explosion in single-frequency periodically driven slow/fast systems can be related to FSNs of type I \cite{Burke2015}. All other results concerning torus canards have been based on numerical averaging methods (for hyperbolic manifolds of limit cycles) and simulations. Thus, we have provided in this article an analytic scheme for the topological classification and characterization of torus canards, which we hope will pave the way for future analyses of torus canard-induced phenomena. 

Another consequence of our work is that we have established more explicit connections between the various families of canard solutions of slow/fast systems (Figure \ref{fig:schematic}). Canard theory has already shown that planar canard cycles always occur $\mathcal{O}(\eps)$ close to a singular Hopf bifurcation. The dynamic unfolding of the singular Hopf is the FSN II, which occurs in 1-fast/2-slow systems as a transcritical bifurcation of an ordinary singularity and a folded singularity \cite{Guckenheimer2008,Krupa2010}. The connection from degenerate torus canards to folded singularity canards was demonstrated to occur in a three time-scale forced van der Pol equation \cite{Burke2015} via a FSN of type I. More precisely, the strong canard of the FSN I for low-frequency forcing would continue to the maximal torus canard of the high-frequency forcing regime. 

\begin{figure}[ht]	
\centering
\includegraphics[width=4.5in]{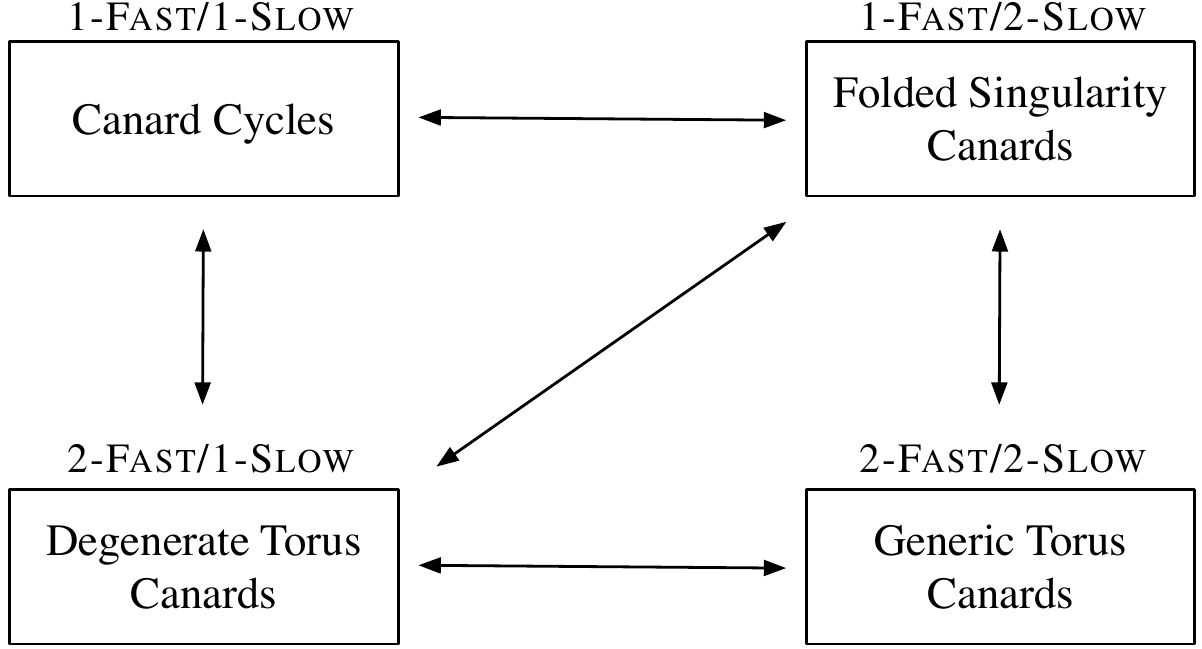}
\caption{Schematic of the various canard families and their connections. Canard cycles occur in planar slow/fast systems via a singular Hopf bifurcation. The dynamic unfolding of the singular Hopf is the FSN II, which occurs in 1-fast/2-slow systems. In single-frequency periodically driven slow/fast systems with one fast variable, one slow variable, and one intermediate variable, the strong canard of the FSN I in the low-frequency forcing regime becomes a maximal torus canard in the high-frequency forcing regime. Our averaging method on folded manifolds of limit cycles has shown that a torus canard problem can be converted into a classical canard problem and vice versa.}
\label{fig:schematic}
\end{figure}

Degenerate torus canards were (tenuously) suggested to be related to canard cycles via averaging \cite{Izhikevich2000b}, and generic torus canards were conjectured to be related to folded singularity canards via averaging \cite{Roberts2015}. Our work here has established the precise nature of those connections. Namely, a generic torus canard can be reduced to a folded singularity canard via the averaging method for folded manifolds of limit cycles (Theorems \ref{thm:averaging} and \ref{thm:averagingarbitraryslow}). That is, a canard solution of the averaged radial-slow system corresponds to a torus canard solution in the original problem.

Conversely, a folded singularity canard can be (trivially) converted to a torus canard by interpreting the fast variable as a radial variable and appending a fast (decoupled) angular variable to the system. Similarly, canard cycles of planar slow/fast systems and degenerate torus canards of 2-fast/1-slow systems can be related by interpreting the fast variable as a radial variable or vice versa. Since the degenerate torus canards are canard cycles of a related averaged radial-slow system, they can be related to generic torus canards via the toral FSN II. 

%----------------------------------------------------
\subsection{Open Problems} 	\label{subsec:problem}
%----------------------------------------------------

Our analysis of generic torus canards has revealed several interesting, open, and relevant problems. First and foremost, the major limitation of our averaging method for folded manifolds of limit cycles is that it has only been developed for systems with two fast variables. 
We currently have no analytic results for systems with more than two fast variables. The first major roadblock is obtaining a test function to detect an SNPO. In the setting of two fast variables, the average of the divergence of the layer problem is a suitable test function for the detection of SNPOs, i.e.,
\begin{equation} \label{eq:snpocondition} 
\frac{1}{T(y)} \int_0^{T(y)} \operatorname{tr} D_x f(\Gamma(t,y),y,0) \, dt = 0, 
\end{equation}
indicates a fold of limit cycles.
However, when there are $n \geq 3$ fast variables, there is no equivalent expression that can be used as a test function for an SNPO. More precisely, the $n$ Floquet exponents of the limit cycle $\Gamma$ of the layer problem satisfy
\[ \varphi_1 = 0, \quad \text{and } \quad \sum_{i=2}^n \varphi_i = \frac{1}{T(y)} \int_0^{T(y)} \operatorname{tr}\, D_x f(\Gamma(t,y),y,0) \, dt, \]
where the trivial exponent $\varphi_1$ indicates neutral stability to shifts along $\Gamma$. A fold of limit cycles occurs when one of the other Floquet exponents is zero, i.e., when $\varphi_i=0$ for some $i \in \{2, 3, \ldots, n\}$. In this higher dimensional setting, it is not clear if the fold condition for limit cycles ($\varphi_i = 0$) has a continuous analogue like \eqref{eq:snpocondition}. 
The second (and related) issue is that, to the best of our knowledge, a method still needs to be found to compute a unit normal to the folded limit cycle $\Gamma$ which would allow us to obtain a result analogous to Lemma \ref{lemma:linear}, which the proofs of Theorems \ref{thm:averaging} and \ref{thm:averagingarbitraryslow} rely on. 

We conjecture that there exists a sequence of coordinate transformations that switches the dynamics to a coordinate frame tangent and orthogonal to the limit cycles such that the resulting averaged system has linear part with the Floquet exponents along the main diagonal. Assuming that such a transformation exists and that the only zero Floquet exponents are $\varphi_1$ and $\varphi_i$ for some $i\in \{2,3,\ldots,n \}$, center-manifold reduction would recover the case studied herein. As such, whilst we cannot as yet make any theoretical statements, we are hopeful that generic torus canards in systems with additional fast directions will have envelope profiles which obey canard dynamics on average. Currently, there is numerical evidence for this in a 6-fast/2-slow model for respiratory rhythm generation in the pre-B\"{o}tzinger complex \cite{Roberts2015}. 

The next major issue is the question of how a bifurcation of torus canards occurs. In our formulation, we converted the torus canard problem to a canard (of folded singularity type) problem. In this averaged radial-slow setting, a bifurcation of maximal torus canards occurs when a branch of secondary torus canards emanates from the weak torus canard at odd integer resonances in the eigenvalue ratio. That is, under variation of the control parameter, the invariant manifolds of limit cycles twist in such a way that a new torus canard splits off from the axis of rotation. We provided numerical evidence for such behaviour in the Politi-H\"{o}fer model, but did not investigate it closely. The behaviour of the invariant manifolds of limit cycles warrants further investigation, since the ability to identify invariant manifolds, particularly repelling ones, is crucial to understanding the properties (such as resetting and firing) of a system.  

Our numerical computation of the intersecting twisted invariant manifolds, $\mathcal{P}_a^{\eps}$ and $\mathcal{P}_r^{\eps}$, of limit cycles was novel but limited.
The homotopic continuation methods \cite{Desroches2008} that are currently used to compute invariant slow manifolds and continue maximal canards of folded singularities will not work for torus canards. One issue is that the first step of the homotopic continuation method requires a known solution of the rescaled system
\begin{equation}	\label{eq:discussion}
\begin{split}
\dot{x} &= T f(x,y,\eps), \\
\dot{y} &= \eps T g(x,y,\eps), 
\end{split}
\end{equation}
subject to $u(0) \in \mathcal{P}_L$ and $u(1) \in \Sigma$, where $u=(x,y)$, $T$ is the actual integration time (so that all solutions are rescaled to the unit time interval), and $\Sigma$ is a hyperplane passing through the toral folded singularity. 
Unlike the classical folded singularities, the toral folded singularity cannot be used as a starting solution for this boundary value problem. As such, the computation of the invariant manifolds of limit cycles in AUTO cannot be done. Consequently, there is currently no way to continue the maximal torus canards in parameters, which is essential for the detection of their bifurcations. We propose one possible modification to begin the continuation: use the toral folded singularity as a solution for the layer problem ($\eps=0$) of \eqref{eq:discussion} and then numerically continue in $\eps$ and $T$ to homotopically grow a suitable starting solution of \eqref{eq:discussion}.

Another related issue that we touched on was the detailed bifurcation structure of the Politi-H\"{o}fer model, especially in the transitions between different AMBs (i.e., bifurcations of torus canard-induced mixed-mode oscillations). Two mechanisms control the number of small-oscillations that the envelope of the AMB waveform executes. The eigenvalue ratio of the TFN determines the maximal number of torus canards that can persist, and the global return of the AMB determines where the orbit is re-injected in $\mathcal{P}_a^{\eps}$ relative to the maximal torus canards. Thus, the loss/gain of an oscillation in the envelope of the waveform can occur either as the eigenvalue ratio crosses an odd integer resonance, or as the global return projects the orbit into a different rotational sector. Whilst we studied the dependence of the eigenvalue ratio on the control parameter, we did not examine the global return in any detail. As such, we have not yet been able to make any useful comparison between the singular and non-singular bifurcation structures. We conjecture that the transition between AMBs with different numbers of oscillations in the envelope will correspond to torus doubling bifurcations, analogous to the way that bifurcations of canard-induced mixed-mode oscillations tend to occur via period-doubling bifurcations \cite{Desroches2012a,Petrov1992,Wechselberger2009}. 

One special type of toral folded singularity that we encountered was the toral FSN II, where a toral folded singularity coincides with an equilibrium of the averaged radial-slow flow. We argued that since the toral FSN II corresponds to a singular Hopf bifurcation of the averaged radial-slow system (from which limit cycles would emanate), the toral FSN II must mark the moment when an invariant phase space torus is created. Thus, we conjectured that a toral FSN II would persist in the fully perturbed problem as a singular torus bifurcation, at an $\mathcal{O}(\eps)$-distance from the toral FSN II. We provided numerical evidence for our conjecture in the Politi-H\"{o}fer, Morris-Lecar-Terman, Hindmarsh-Rose, and Wilson-Cowan-Izhikevich models. However, we have yet to rigorously prove our conjecture. The analysis of the toral FSN II is significant, since it is the boundary between amplitude-modulated spiking and AMB. 

%----------------------------------------------------
\subsection{Outlook} 	\label{subsec:outlook}
%----------------------------------------------------

In this article, we provided the first rigorous framework for studying generic torus canards in slow/fast systems with two fast and $k$-slow variables, for any $k \geq 1$. We demonstrated the predictive power of our analytic tools, and discovered new dynamic phenomena (AMB due to torus canard-induced mixed-mode oscillations) in the process. 
We showed that these new phenomena are robust and are therefore expected to be experimentally observable. A possible example where AMB may have been observed experimentally is in leech heart interneurons \cite{Cymbalyuk2002}.
Our analysis has shown that there is still of whole vista of possibilities to be explored in terms of the theoretical development, numerical implementation, and practical applications of torus canards. We believe that the analysis of generic torus canards and their interactions with other dynamic phenomena will lead to rich, complicated, and important new dynamics.

%------------------------------------------------------------------------------------------------------------------------
\begin{appendices}
%------------------------------------------------------------------------------------------------------------------------

%---------------------------------------------------------------------------------
\section{Dimensional Analysis of the Politi-H\"{o}fer Model}		\label{app:PH}
%---------------------------------------------------------------------------------

In this appendix, we provide details of the non-dimensionalization of the PH model \eqref{eq:PHmodel}. We define dimensionless variables $C, C_t, P$, and $t_s$ via the rescalings
\[ c = Q_c C, \quad c_t = Q_c C_t, \quad p = Q_p P, \,\, \text{ and }\,\, t = Q_t \, t_s, \]
where $Q_c$ and $Q_p$ are reference calcium and \ip3 concentrations, respectively, and $Q_t$ is a reference time-scale. Substitution into \eqref{eq:PHmodel} leads to the dimensionless system
\begin{align*}
\dot{C} &= \frac{Q_t V_{\text{serca}} \gamma}{Q_c} \left( \hat{J}_{\text{release}}-\frac{1}{\gamma} \hat{J}_{\text{serca}} \right)+\frac{Q_t \delta V_{\text{pm}}}{Q_c} \left( \hat{J}_{\text{in}}-\hat{J}_{\text{pm}} \right), \\
\dot{C_t} &= \frac{Q_t \delta V_{\text{pm}}}{Q_c} \left( \hat{J}_{\text{in}}-\hat{J}_{\text{pm}} \right), \\
\dot{r} &= \frac{Q_t}{\tau_r} \left( 1- r \frac{K_i+Q_c C}{K_i} \right), \\
\dot{P} &= Q_t k_{3K} \left( \frac{V_{PLC}}{Q_p} - \frac{1}{1+\frac{K_{3K}^2}{Q_c^2 C^2}} P \right),
\end{align*}
where the overdot denotes the derivative with respect to the dimensionless slow time $t_s$, and the dimensionless fluxes are given by
\begin{equation}
\begin{split}
\hat{J}_{\text{release}} &= \frac{Q_c k_1}{V_{\text{serca}}} \left( \left( r \frac{1}{1+ \frac{K_a}{Q_c C}} \, \frac{1}{1+\frac{K_p}{Q_p P}} \right)^3 + \frac{k_2}{k_1} \right)  \left( C_t - \left( 1+\frac{1}{\gamma} \right) C \right), \\
\hat{J}_{\text{serca}} &= \frac{1}{1+ \frac{K_{\text{serca}}^2}{Q_c^2 C^2}}, \\
\hat{J}_{\text{pm}} &= \frac{1}{1+ \frac{K_{\text{pm}}^2}{Q_c^2 C^2}}, \\
\hat{J}_{\text{in}} &= \frac{\nu_0}{V_{\text{pm}}} + \phi \frac{V_{PLC}}{V_{\text{pm}}}.
\end{split}
\end{equation}
Setting the reference time-scale to be the slow time-scale ($Q_t=T_{c_t}$), and defining the dimensionless parameters
\[ \eps = \frac{\delta V_{\text{pm}}}{V_{\text{serca}} \gamma}, \,\, \hat{\tau}_r = \tau_r \frac{V_{\text{serca}} \gamma}{Q_c}, \,\, \text{ and }\,\, \hat{k}_{3K} = \frac{k_{3K}Q_c}{\delta V_{\text{pm}}}, \]
leads to the dimensionless PH model \eqref{eq:PHdimless}, where the functions $f_1,f_2,g_1$ and $g_2$ are given by
\begin{align*}
f_1(C,r,C_t,P) &= \hat{J}_{\text{release}}-\frac{1}{\gamma} \hat{J}_{\text{serca}}, \\
f_2(C,r,C_t,P) &= \frac{1}{\hat{\tau}_r} \left( 1-r \frac{K_i+Q_c C}{K_i} \right), \\
g_1(C,r,C_t,P) &= \hat{J}_{\text{in}}-\hat{J}_{\text{pm}}, \\
g_2(C,r,C_t,P) &= \hat{k}_{3K} \left( \frac{V_{PLC}}{Q_p} - \frac{P}{1+\frac{K_{3K}^2}{Q_c^2 c^2}} \right).
\end{align*}
%
%and expressions for the dimensionless fluxes $\hat{J}_{\text{release}}, \hat{J}_{\text{serca}}, \hat{J}_{\text{in}}$, and $\hat{J}_{\text{pm}}$ are given in Appendix \ref{app:PH}.
Details of the parameters and their standard values are provided in Table \ref{tab:PHparams}.

\begin{table}[ht!]
\centering
\begin{tabular}{|l|l|l|}
\hline 
Parameter & Value & Description \\
\hline
$K_{3K}$ & 0.4 $\mu$M & Half-maximal concentration constant \\
$k_{3k}$ & 0.1 s$^{-1}$ & \ip3 phosphorylation rate constant \\
$\gamma$ & 5.405 & Ratio of cytosolic volume to endoplasmic reticulum volume \\
$V_{\text{serca}}$ & 0.25 $\mu$M s$^{-1}$ & Maximal SERCA pump rate \\
$K_{\text{serca}}$ & 0.1 $\mu$M & Half-activation constant \\
$V_{\text{pm}}$ & 0.01~$\mu$M s$^{-1}$ & Maximal PMCA pump rate \\ 
$K_{\text{pm}}$ & 0.12 $\mu$M & Half-activation constant \\ 
$\delta$ & 0.472938 & Strength of plasma membrane fluxes \\
$\nu_0$ & 0.001 $\mu$Ms$^{-1}$ & Constant calcium influx \\
$\phi$ & 0.045 s$^{-1}$ & Stimulation-dependent influx \\
$k_1$ & 7.4 s$^{-1}$ & Maximal rate of Ca$^{2+}$ release \\
$k_2$ & 0.00148 s$^{-1}$ & Ca$^{2+}$ leak \\
$K_a$ & 0.2 $\mu$M & Ca$^{2+}$ binding to activating site \\
$K_i$ & 0.3 $\mu$M & Ca$^{2+}$ binding to inhibiting site \\
$K_p$ & 0.13 $\mu$M & \ip3 binding \\
$\tau_r$ & 6.6 s & Characteristic time of \ip3 receptor inactivation \\
\hline
\end{tabular}
\caption{Standard parameter values for the PH model. The values here are identical to those used in \cite{Harvey2011,Politi2006}, except for the values of $\nu_0$ and $\delta$, which are set at $\nu_0 = 0.0004~\mu$Ms$^{-1}$ and $\delta=1$ in \cite{Harvey2011,Politi2006}. Our $\delta$ is chosen so that the dimensionless small parameter $\eps$ will be `nice' (see Section \ref{subsec:PHlayer}).}
\label{tab:PHparams}
\end{table}

%---------------------------------------------------------------------------------
\section{Numerical Continuation of Toral FSNs of Type II} 	\label{app:TFScurve}
%---------------------------------------------------------------------------------

Here, we outline the procedure for computing (in AUTO \cite{Doedel2007}) sets of toral FSNs of type II in two-parameter families of 2-fast/2-slow systems of the form
\begin{align*}
\dot{x} &= f(x,y,\mu,\eps), \\
\dot{y} &= \eps g(x,y,\mu,\eps),
\end{align*}
where $x = (x_1,x_2) \in \mathbb{R}^2$ is fast, $y = (y_1,y_2) \in \mathbb{R}^2$ is slow, $\mu = (\mu_1,\mu_2) \in \mathbb{R}^2$ is a vector of parameters, $0< \eps \ll 1$, and $f = (f_1,f_2)$ and $g=(g_1,g_2)$ are sufficiently smooth. 

We first reformulate the layer problem as a boundary value problem subject to periodic boundary conditions 
\begin{equation} \label{eq:BVP}
\dot{x} = T f(x,y,\mu,0), \quad x(0)=x(1). 
\end{equation}
Here, solutions have been rescaled to the unit time interval and the actual integration time $T$ appears as an explicit parameter. For any periodic orbit, $\Gamma$, of this boundary value problem, there is an infinite family of periodic solutions corresponding to $\Gamma$ but starting at different initial points \cite{Doedel1991}. We pick out one of these solutions by appending a phase condition 
\begin{equation} \label{eq:phase}
x_1(0)=x_{10}, 
\end{equation}
where the cross-section $\{ x_1=x_{10}\}$ is chosen such that $\Gamma$ passes through $\{ x_1=x_{10} \}$. A stable limit cycle, $\Gamma$, of the layer problem is a solution of \eqref{eq:BVP} subject to \eqref{eq:phase}.

An important diagnostic that needs to be monitored is the averaged slow drift, $\overline{g}_1$ and $\overline{g}_2$, given by 
\[ \overline{g}_1 = \int_0^1 g_1(\Gamma,y,\mu,0)\, dt, \quad \text{ and }\quad \overline{g}_2 = \int_0^1 g_2(\Gamma,y,\mu,0)\, dt, \]
respectively. In particular, $\overline{g}_1$ and $\overline{g}_2$ must be computed for the stable limit cycle $\Gamma$. 
To initialize the computation of toral folded singularities, we import $\Gamma$ and these initial values for $\overline{g}_1$ and $\overline{g}_2$ into AUTO. 
We then numerically continue the periodic $\Gamma$ in one of the slow variables, $y_1$ say, and the integration time $T$ until a fold of limit cycles is detected. At the SNPO, we have
\begin{equation} \label{eq:bvpsnpo}
\int_0^1 \operatorname{tr}D_x f(\Gamma,y,\mu,0)\, dt=0,
\end{equation}
which is a necessary condition for an SNPO but not sufficient. 
We then compute a family of folded limit cycles by numerically continuing solutions of \eqref{eq:BVP} subject to the phase condition \eqref{eq:phase} and the integral condition \eqref{eq:bvpsnpo} in the parameters $(y_1,y_2,T)$ until 
\begin{equation} \label{eq:bvpaveqm1}
\overline{g}_1 = \int_0^1 g_1(\Gamma,y,\mu,0)\, dt=0.
\end{equation}
We then continue solutions of the boundary value problem \eqref{eq:BVP} subject to \eqref{eq:phase}, \eqref{eq:bvpsnpo}, and \eqref{eq:bvpaveqm1} in the parameters $(y_1,y_2,\mu_1,T)$ until 
\begin{equation} \label{eq:bvpaveqm2}
\overline{g}_2 = \int_0^1 g_2(\Gamma,y,\mu,0)\, dt=0,
\end{equation}
and we have a toral FSN II. The manifold of toral FSNs of type II is then obtained by enforcing the integral condition \eqref{eq:bvpaveqm2}, and continuing solutions in the parameters $(y_1,y_2,\mu_1,\mu_2,T)$. 

Thus, the curve of toral FSNs of type II is obtained by solving the periodic boundary value problem \eqref{eq:BVP} and \eqref{eq:phase}, subject to the integral constraints \eqref{eq:bvpsnpo}, \eqref{eq:bvpaveqm1}, and \eqref{eq:bvpaveqm2}. Note that this numerical continuation procedure generalizes to the computation and continuation of toral FSNs of type II in 2-fast/$k$-slow systems for any $k \geq 1$.

%---------------------------------------------------------------------------------
\section{Maximal Torus Canards in $\mathbb{R}^3$} 	\label{app:TCexplosion}
%---------------------------------------------------------------------------------

In this Appendix, we provide an extension of the averaging theorem which allows one to determine the parameter values for which torus canard explosions occur in 2-fast/1-slow systems. The assumptions are identical to those of Section \ref{subsec:assumptions}, with $y \in \mathbb{R}$. Under those conditions, we have the following refinement of the averaging theorem.

\begin{theorem}[\bf{Extended Averaging Transformation}] \label{thm:3Daveraging}
Consider system \eqref{eq:main} with $x\in \mathbb{R}^2$ and $y\in \mathbb{R}$, under assumptions \ref{ass:man}--\ref{ass:homoclinic}, and let $(\Gamma(t,y),y) \in \mathcal{P}_L$. Then there exists a sequence of near-identity coordinate transformations such that the averaged radial-slow dynamics in the neighbourhood of $(\Gamma(t,y),y)$ are described by
\begin{equation} \label{eq:3Daveraging}
\begin{split}
\dot{R} &= \overline{a} u + \overline{b} R^2 + \overline{c} Ru + \overline{\xi} R^3 + \mathcal{O}(\eps, R^2u, u^2), \\
\dot{u} &= \eps \left( \overline{g} + \overline{d} R + \overline{e} u + \overline{\eta} R^2 +\mathcal{O}(\eps, Ru,u^2) \right),
\end{split}
\end{equation}
where an overbar denotes an average over one period of $\Gamma(t,y)$, and the averaged coefficients can be computed explicitly. 
\end{theorem}

\begin{proof}
The proof follows in the same way as that of Theorems \ref{thm:averaging} and \ref{thm:averagingarbitraryslow}, but requires that we expand to higher-order in both the radial and slow directions. The averaged coefficient of the cubic radial term in the radial equation is given by 
\[ \overline{\xi} = \frac{1}{6T} \int_0^T q \cdot \begin{pmatrix} (q \cdot \nabla_x)^3 f_1 \\ (q \cdot \nabla_x)^3 f_2 \end{pmatrix} \Phi^2(t) \, dt, \]
where $q$ is a unit normal to the periodic orbit $\Gamma$ and $\Phi(t)$ is the fundamental solution of the linearized flow about $\Gamma$. The averaged coefficient of the $R^2$-term in the averaged slow equation is
\[ \overline{\eta} = \frac{1}{T} \int_0^T \left( (q\cdot D_x^2 g\, q)\Phi^2(t) + \beta d \right) \, dt, \]
where $\beta$ is the coefficient of the $R^2$ term in the near-identity transformation \eqref{eq:NIT} and is chosen to be the solution of
\[ \frac{d\beta}{dt} = b - \frac{1}{T} \int_0^T b\, dt, \quad \beta(0)=0, \]
and $d$ is the coefficient of the linear $R$-term in the $u$-equation. All other averaged coefficients are exactly as in Section \ref{subsec:averagedcoefficients}.
\end{proof}

\begin{remark}
We point out that only the average of $g$ over the folded limit cycle is needed to determine where the toral folded singularity occurs. The higher order terms in system \eqref{eq:3Daveraging} are necessary for tracking the maximal torus canard in its $\eps$-unfolding. 
\end{remark}

\begin{corollary}	\label{cor:explosion}
Suppose system \eqref{eq:3Daveraging} possesses a toral folded singularity, i.e. a limit cycle $\Gamma(t,y)$ of the layer problem of \eqref{eq:main} such that 
\[ \int_0^T \operatorname{tr}\, D_xf(\Gamma(t,y),y,0) \, dt = 0, \,\, \text{ and }\,\, \int_0^T g(\Gamma(t,y),y,0) \, dt = 0. \]
Suppose further that the averaged coefficients of Theorem \ref{thm:3Daveraging} are such that $\overline{a} \overline{d} <0$. 
Then, for $0<\eps \ll 1$, there is a torus canard explosion in an exponentially small neighbourhood of the parameter set such that
\[ \int_0^T g(\Gamma(t,y),y,0)\,dt = \overline{\lambda} \eps + \mathcal{O}\left( \eps^{3/2} \right), \]
where $\overline{\lambda}$ is given by 
\[ \overline{\lambda} := \frac{\overline{d}}{8\overline{b}^3} \,  \left( \overline{b}\overline{c}\overline{d}+2\overline{b}^2 \overline{e}+2\overline{a}\overline{b}\overline{\eta}-3\overline{a}\overline{d}\overline{\xi} \right) . \]
\end{corollary}

\begin{proof}
The averaged system \eqref{eq:3Daveraging} is a planar slow/fast system with effective control parameter $\overline{g}$, and canard point located at $(R,u,\overline{g})=(0,0,0)$. Careful analyses of such problems have been performed for instance in \cite{Krupa2001}. Therefore, we only sketch an outline of the steps.

In the first step, we apply the blow-up transformation
\[ \left( R, u, \overline{g}, \eps \right) = \left( \rho \tilde{R}, \rho^2 \tilde{u}, \rho^2 \tilde{\lambda}, \rho^2 \tilde{\eps} \right), \]
where $\rho \in [-p,p]$ for some small and positive $p$, which inflates the nilpotent equilibrium at the origin to the hypersphere $\mathbb{S}^3$, i.e.,
\[ \tilde{R}^2 + \tilde{u}^2 + \tilde{\lambda}^2 + \tilde{\eps}^2 =1. \]
On this hypersphere, trajectories with differing orders of tangency are teased apart and enough hyperbolicity is restored that a complete analysis can be performed using dynamical systems techniques. The dynamics on $\mathbb{S}^3$ are typically analysed using an overlapping set of coordinate charts, which cover the hypersphere by planes perpendicular to the coordinate axes. Two particularly useful charts are the entry chart (defined by $\tilde{u}=-1$) and the central (or rescaling) chart (defined by $\tilde{\eps}=1$). 

In the rescaling chart, the blown-up vector field (after desingularization by a factor of $\sqrt{\eps}$) is given by
\begin{equation} \label{eq:blownup}
\begin{split}
\dot{R}_2 &= \overline{a} u_2 + \overline{b} R_2^2 + \sqrt{\eps} \left( \overline{c} R_2 u_2 + \overline{\xi} R_2^3 \right) +\mathcal{O}(\eps), \\
\dot{u}_2 &= \overline{d} R_2 + \sqrt{\eps} \left( \overline{\lambda}+ \overline{e} u_2 + \overline{\eta} R_2^2 \right) + \mathcal{O}(\eps),
\end{split}
\end{equation}
where the overdot denotes time derivatives, $\rho = \sqrt{\eps}$, and the 2-subscript indicates a rescaled variable in the central chart of the blow-up. The unperturbed version (i.e., $\eps=0$) of \eqref{eq:blownup} is the Hamiltonian system 
\begin{equation}	\label{eq:unperturbed}
\begin{split}
\dot{R}_2 &= \overline{a} u_2 + \overline{b} R_2^2, \\
\dot{u}_2 &= \overline{d} R_2,
\end{split}
\end{equation}
which has the Hamiltonian function
\[ \mathscr{H}(R_2,u_2) := \exp \left( -\frac{2\overline{b}u_2}{\overline{d}} \right) \left\{ -\overline{d}R_2^2-\frac{\overline{a}\overline{d}}{\overline{b}}u_2 - \frac{\overline{a}\overline{d}^2}{2\overline{b}^2} \right\}. \]
The unperturbed problem possesses an explicit algebraic solution $\gamma(t_2)$ given by 
\[ \gamma(t_2) = \left( -\frac{\overline{a}\overline{d}}{2\overline{b}}t_2, - \frac{\overline{d}}{2\overline{b}}-\frac{\overline{a}\overline{d}^2}{4\overline{b}}t_2^2 \right). \]

The splitting distance between the attracting and repelling invariant slow manifolds is measured using the Melnikov function, $\mathcal{D}$. Expanding the Melnikov function as a series in the small parameter $\eps$ gives
\[ \mathcal{D}(\eps) = d_{\sqrt{\eps}} \sqrt{\eps} + \mathcal{O}(\eps), \]
where the Melnikov integral $d_{\sqrt{\eps}}$ is computed as
\[ d_{\sqrt{\eps}} = \int_{-\infty}^{\infty} \left. \nabla H \right|_{\gamma} \cdot \left. \begin{pmatrix} \overline{c} R_2 u_2 + \overline{\xi} R_2^3 \\ \overline{\lambda}_2 +\overline{e} u_2 + \overline{\eta} R_2^2 \end{pmatrix} \right|_{\gamma} \, dt_2. \]
Note that this integral only converges provided $\overline{a} \overline{d} <0$, which is also a necessary condition for the trajectories of \eqref{eq:unperturbed} to be closed orbits. Substituting $d_{\sqrt{\eps}}$ into the bifurcation equation $\mathcal{D} = 0$, solving for $\lambda_2$, and reverting to the original unscaled variables gives the result. 
\end{proof}

%----------------------------------------------------
\subsection{Torus Canard Explosion in the Forced van der Pol Equation} 	\label{app:subsec:fvdp}
%----------------------------------------------------

We now demonstrate the predictive power of Theorem \ref{thm:3Daveraging} in an analytically tractable example. We consider the forced van der Pol (fvdP) equation in the relaxation limit subject to periodic forcing, given by
\begin{equation}	\label{eq:fvdpcylindrical}
\begin{split}
\dot{x} &= y-\left( \frac{x^3}{3}-x \right), \\
\dot{y} &= \eps \left( -x+\alpha+\beta \cos \theta \right), \\
\dot{\theta} &= \omega,
\end{split}
\end{equation}
where $0<\eps \ll 1$ is small, the forcing is small-amplitude (i.e., $\beta= \mathcal{O}(\eps)$) and the forcing frequency is high (i.e., $\omega = \mathcal{O}(1)$) so that \eqref{eq:fvdpcylindrical} is a 2-fast/1-slow system. In order to implement our results, we first switch from the cylindrical coordinates of \eqref{eq:fvdpcylindrical} to Cartesian coordinates via the transformation $u = x \cos \theta$ and $v = x \sin \theta$. In these Cartesian coordinates, the fvdP equation becomes
\begin{equation}	\label{eq:fvdp}
\begin{split}
\dot{u} &= u-\omega v - \frac{1}{3} u(u^2+v^2) + \frac{u y}{\sqrt{u^2+v^2}}, \\
\dot{v} &= \omega u+v - \frac{1}{3} v(u^2+v^2) + \frac{v y}{\sqrt{u^2+v^2}}, \\
\dot{y} &= \eps \left( -\sqrt{u^2+v^2}+\alpha+\beta \frac{u}{\sqrt{u^2+v^2}} \right).
\end{split}
\end{equation}
The layer problem of \eqref{eq:fvdp} has the manifold of SNPOs, given by
\[ \mathcal{P}_L = \left\{ \Gamma = (u_\Gamma,v_\Gamma, y_\Gamma)=\left( \cos \omega t, \sin \omega t, - \frac{2}{3} \right)  \right\}. \]
That is, there is a single folded limit cycle, which has period $T = \frac{2\pi}{\omega}$, and corresponds to the fold point $(x,y)=(1,-\frac{2}{3})$ of the cubic $x^3/3-x$ in the original cylindrical coordinates. It has been shown that for high-frequency forcing, the fvdP equation has the following properties \cite{Burke2015}:
\begin{itemize}
\setlength{\itemsep}{0pt}
\item system \eqref{eq:fvdpcylindrical} undergoes a singular Hopf bifurcation when $\alpha=1$ and $\beta=0$,
\item system \eqref{eq:fvdpcylindrical} has a torus bifurcation at $\alpha=1$,
\item there is a torus canard explosion along the curves
\[ \alpha = 1 - \frac{\eps}{8} \pm \beta \exp\left( -\frac{\omega^2}{2\eps} \right), \]
in the $(\alpha,\omega)$ plane for $\beta$ at most $\mathcal{O}(1)$. In this high-frequency forcing (i.e., $\omega = \mathcal{O}(1)$) regime, the exponential term is negligible and the torus canard explosion occurs in an exponentially small parameter window, centred on the line $\alpha = 1- \frac{\eps}{8}$.
\end{itemize}

The results on torus canards in the fvdP equation were obtained by studying the fvdP equation in the low- and intermediate-frequency forcing regimes (i.e., $\omega = \mathcal{O}(\eps)$ and $\omega = \mathcal{O}(\sqrt{\eps})$) and then continuing those results into the high-frequency forcing regime \cite{Burke2015}. Until now, there was no direct way to explicitly compute the location of the singular torus canard and its corresponding explosion by starting in the high-frequency forcing regime. We now apply Theorem \ref{thm:3Daveraging} and its consequences to system \eqref{eq:fvdp}. 

A toral folded singularity of \eqref{eq:fvdp} occurs when
\[ \int_0^T g(u_\Gamma,v_\Gamma,y_\Gamma)\, dt = \int_0^T \left( -\sqrt{u^2+v^2} +\alpha + \frac{\beta u}{\sqrt{u^2+v^2}} \right)\, dt = 0.  \]
Solving for $\alpha$ explicitly gives the location of the toral folded singularity in parameter space as
\[ \alpha = \frac{\omega}{2\pi} \int_0^{\frac{2\pi}{\omega}} \left( 1-\beta \cos \omega t \right)\, dt = 1, \]
which agrees with the established result. Now, by Corollary \ref{cor:explosion}, a torus canard explosion occurs in the neighbourhood of the parameter set where
\[ \frac{1}{T} \int_0^T g \left( u_\Gamma,v_\Gamma,y_\Gamma \right) \, dt = \overline{\lambda} \eps \,+\, \mathcal{O}\left( \eps^{3/2} \right),\]
where $\overline{\lambda}$ is the specific combination of averaged coefficients given in Corollary \ref{cor:explosion}. 
The averaged coefficients of \eqref{eq:fvdp} are given in Table \ref{tab:fvdP}.
\begin{table}[h!!]
\renewcommand\arraystretch{1.5}
\centering
\begin{tabular}{|r|c|r|c|r|c|r|c|}
\hline
$\overline{a}$ & $1$ & $\overline{b}$ & $-1$ & $\overline{c}$ & $0$ & $\overline{\xi}$ & $-\frac{1}{3}$ \\
\hline
$\overline{g}$ & $\alpha-1$ & $\overline{d}$ & $-1$ & $\overline{e}$ & $0$ & $\overline{\eta}$ & $0$ \\
\hline
\end{tabular}
\caption{Averaged coefficients from Theorem \ref{thm:3Daveraging} for the fvdP equation.}
\label{tab:fvdP}
\end{table}

Using Table \ref{tab:fvdP}, we find that the expression for $\overline{\lambda}$ simplifies greatly and the condition for the torus canard explosion reduces to
\[ \alpha = 1-\frac{1}{8} \eps \,+\, \mathcal{O}\left( \eps^{3/2} \right). \]
Recall that the actual analytic result is $\alpha = 1-\frac{\eps}{8}$ (plus an exponentially small correction). Thus, in the case of the fvdP oscillator, Theorem \ref{thm:3Daveraging} and Corollary \ref{cor:explosion} give the location of the torus canard explosion (for $\omega = \mathcal{O}(1)$) correct up to exponentially small error.

%------------------------------------------------------------------------------------------------------------------------
\end{appendices}
%------------------------------------------------------------------------------------------------------------------------

%------------------------------------------------------------
\section*{Acknowledgments} 
%------------------------------------------------------------
This research was partially supported by NSF-DMS 1109587. 
I would like to thank Tasso Kaper, Mark Kramer, Jonathan Rubin, and Martin Wechselberger for helpful discussions. 
I am particularly grateful to Tasso Kaper and Mark Kramer for their careful and critical reading of the manuscript. 
I am especially indebted to Tasso Kaper for being an excellent sounding board for my ideas throughout the development of this project.

%%%%%%%%      REFERENCES      %%%%%%%%

%=========================================================================================
\end{document}